\documentclass[12 pt,reqno]{amsart}
\usepackage{pdfpages}
\usepackage{amsmath,amsthm,amssymb,amscd}
\usepackage{setspace}
\usepackage{upgreek}
\usepackage[utf8]{inputenc}
\usepackage{cite}
\usepackage{url}
\usepackage{mathtools}
\usepackage{fullpage}
\usepackage{float}
\usepackage{comment}
\usepackage{enumitem}
\usepackage{tikz-cd}
\usepackage{color, soul}
\usepackage{nicefrac,xfrac}
\usepackage{diagbox}
\usepackage{quiver}
\usepackage[bookmarksopen,bookmarksdepth=3]{hyperref}

\usepackage{enumitem}
\setlist[enumerate,1]{label={(\alph*)}}

\input xy
\xyoption{all}

\DeclareFontFamily{U}{rcjhbltx}{}
\DeclareFontShape{U}{rcjhbltx}{m}{n}{<->rcjhbltx}{}
\DeclareSymbolFont{hebrewletters}{U}{rcjhbltx}{m}{n}

\DeclareMathSymbol{\mem}{\mathord}{hebrewletters}{109}

\newtheorem{thm}{Theorem}
\newtheorem{prop}{Proposition}[section]
\newtheorem{lm}[prop]{Lemma}
\newtheorem{cor}[prop]{Corollary}

\theoremstyle{definition}
\newtheorem{dfn}[prop]{Definition}

\theoremstyle{remark}
\newtheorem{rem}[prop]{Remark}
\newtheorem{ex}[prop]{Example}

\newcommand{\R}{\mathbb{R}}
\newcommand{\Z}{\mathbb{Z}}
\newcommand{\Q}{\mathbb{Q}}
\newcommand{\mH}{\mathbb{H}}
\newcommand{\bmH}{\overline{\mathbb{H}}}

\newcommand{\C}{\mathbb{C}}
\newcommand{\oo}{\mathfrak{o}}
\newcommand{\s}{\mathfrak{s}}
\newcommand{\f}{\mathfrak{f}}
\newcommand{\m}{\mathfrak{m}}
\newcommand{\ogw}{\overline{OGW}\!}
\newcommand{\oogw}{OGW}
\newcommand{\su}{\text{SU}}
\newcommand{\sln}{\text{SL}}
\newcommand{\gw}{GW}
\newcommand{\x}{\mathbb{C}P^3}
\newcommand{\y}{Y_{\triangle}}
\newcommand{\tr}{\triangle}
\newcommand{\n}{N_{\triangle}}
\newcommand{\Om}{\overline{\Omega}}
\newcommand{\hH}{\widehat{H}}
\newcommand{\cA}{\mathcal{A}}
\newcommand{\cB}{\mathcal{B}}

\newcommand{\sym}{\text{Sym}}
\newcommand{\cp}{\mathbb{C}P}
\newcommand{\pd}{\text{PD}}
\newcommand{\sheafo}{\mathcal{O}}
\newcommand{\M}{\mathcal{M}}

\DeclareMathOperator{\Int}{int}
\DeclareMathOperator{\Stab}{Stab}
\DeclareMathOperator{\Ima}{Im}
\DeclareMathOperator{\id}{Id}
\DeclareMathOperator{\coker}{Coker}

\title{Relative quantum cohomology of the Chiang Lagrangian}
\author[A. Hollands]{Anna Hollands}
\address{Institute of Mathematics, Hebrew University and Department of Mathematics, Imperial College}
\email{anna.hollands18@imperial.ac.uk}
\author[E. Kosloff]{Elad Kosloff}
\address{Institute of Mathematics, Hebrew University}
\email{eladkosloff@gmail.com}
\author[M. Sela]{May Sela}
\address{Institute of Mathematics, Hebrew University}
\email{mayysela@gmail.com}
\author[Q. Shu]{Qianyi Shu}
\address{Institute of Mathematics, Hebrew University and Department of Mathematics, Imperial College}
\email{roryqy@gmail.com}
\author[J. Solomon]{Jake P. Solomon}
\address{Institute of Mathematics, Hebrew University}
\email{jake@math.huji.ac.il}

\begin{document}
\date{February 2025}
\begin{abstract}
We compute the open Gromov-Witten disk invariants and the relative quantum cohomology of the Chiang Lagrangian $L_\triangle \subset \mathbb{C}P^3$. Since $L_\triangle$ is not fixed by any anti-symplectic involution, the invariants may augment straightforward $J$-holomorphic disk counts with correction terms arising from the formalism of Fukaya $A_\infty$-algebras and bounding cochains. These correction terms are shown in fact to be non-trivial for many invariants. Moreover, examples of non-vanishing mixed disk and sphere invariants are obtained.

We characterize a class of open Gromov-Witten invariants, called \emph{basic}, which coincide with straightforward counts of $J$-holomorphic disks. Basic invariants for the Chiang Lagrangian are computed using the theory of axial disks developed by Evans-Lekili and Smith in the context of Floer cohomology. The open WDVV equations give recursive relations which determine all invariants from the basic ones. The denominators of all invariants are observed to be powers of $2$ indicating a non-trivial arithmetic structure of the open WDVV equations. The magnitude of invariants is not monotonically increasing with degree. Periodic behavior is observed with periods $8$ and~$16.$
\end{abstract}
\maketitle
\tableofcontents

\section{Introduction}\label{section1}
The main result of the present paper is a complete computation of the open Gromov-Witten disk invariants of the Chiang Lagrangian~\cite{chiang2004new} in $\cp^3.$ The Chiang Lagrangian is not fixed by any anti-symplectic involution~\cite{EvansLekili}. Multiple phenomena are observed that have not appeared in previous computations of open Gromov-Witten invariants for Lagrangian submanifolds fixed by anti-symplectic involutions. Our computation relies on direct geometric arguments combined with general structure theorems governing the invariants including the open WDVV equations of~\cite{RelativeQuantumCohomology} and results of the present paper.

\subsection{Invariants}\label{ssec:background}
To formulate our results, we recall relevant background on open and closed Gromov-Witten invariants. Let $(X,\omega)$ be a symplectic manifold of real dimension~$2n$, let $L \subset X$ be a Lagrangian submanifold and let $J$ be an $\omega$-tame almost complex structure. For simplicity, assume $L$ is a connected, spin Lagrangian submanifold with $H^*(L;\R) \simeq H^*(S^n;\R)$ and $[L] = 0 \in H_n(X;\R).$ We refer the reader to~\cite{solomon2016point} for a fuller account of settings where open Gromov-Witten invariants can be defined. Write $\widehat{H}^*(X,L;\R) = H^0(L;\R)\oplus H^{>0}(X,L;\R).$ Let $y : H^n(L;\R) \to \widehat{H}^{n+1}(X,L;\R)$ denote the boundary map from the short exact sequence of the pair $(X,L).$ Choose a map $P_\R : \widehat{H}^{n+1}(X,L;\R)\to H^n(L;\R)$ left-inverse to $y.$ For $\beta \in H_2(X,L;\mathbb{Z})$ and $k,l \geq 0,$ let
\[
\ogw_{\beta,k} : \hH^*(X,L;\mathbb{R})^{\otimes l}\rightarrow \mathbb{R}
\]
denote the open Gromov-Witten invariants of~\cite{RelativeQuantumCohomology}. The multilinear maps $\ogw_{\beta,k}$ are invariants of the symplectic manifold and the Lagrangian submanifold $L$ up to Hamiltonian isotopy. When $\beta$ belongs to the image of the natural map $\varpi : H_2(X;\Z) \to H_2(X,L;\Z),$ the invariants $\ogw_{\beta,0}$ depend on the choice of $P_\R.$ Explicit formulas for this dependence are given in Propositions~\ref{prop_map_p} and~\ref{prop:general_map_p} below.

The invariants $\ogw$ can be described as follows. Let $A_1,\ldots,A_l \in \hH^*(X,L;\R).$ Assume first that either $k > 0$ or $\beta$ does not belong to the image of $\varpi.$ The invariant $\ogw_{\beta ,k}(A_1,\ldots,A_l)$ counts $J$-holomorphic disks $u: (D,\partial D) \to (X,L)$ representing $\beta$ with boundary passing through $k$ chosen points on $L$ and interior passing through chosen cycles on $X$ Poincar\'e dual to $A_1,\ldots,A_l,$ together with correction terms that compensate for disk bubbling. The correction terms come from bounding cochains in the Fukaya $A_\infty$ algebra associated to~$L.$
The assumption $H^*(L;\R) \simeq H^*(S^n;\R)$ is used in~\cite{solomon2016point} to show vanishing of obstruction classes that arise in the construction of the bounding cochains following the method of Fukaya-Oh-Ohta-Ono~\cite{FOOO}.

In the case $k = 0$ and $\beta \in \Ima \varpi,$ it is necessary to add an additional correction term to the above mentioned counts of $J$-holomorphic disks to obtain an invariant. The additional correction term counts $J$-holomorphic spheres in $X$ passing through an $n+1$ chain $C$ in $X$ with $\partial C = L.$ This term compensates for the fact that the boundary of a $J$-holomorphic disk can collapse to a point in $L$ forming a $J$-holomorphic sphere. The topological type of the chain $C$ is characterized by the map $P_\R : \widehat{H}^{n+1}(X,L;\R)\to H^n(L;\R)$ mentioned above, which  is given by the integration of differential forms over $C.$ See Remark 4.12 in~\cite{RelativeQuantumCohomology}.

Suppose $L$ is fixed by an anti-symplectic involution, $\dim L = 2$ or $3$, and either $k > 0$ or $\beta \notin \Ima \varpi.$ Then, it is shown in~\cite{solomon2016point} that the invariants $\ogw_{\beta,k}(\cdots)$ coincide with invariants defined by straightforward counting of $J$-holomorphic disks~\cite{Cho08,jakethesis} or real $J$-holomorphic spheres~\cite{Wel05a,Wel05b}. When $L$ is the Chiang Lagrangian, which is not fixed by any anti-symplectic involution, the invariants $\ogw$ do not coincide with straightforward disk counts in general. See Section~\ref{sssec:bcc&d} below. Closely related invariants were defined in~\cite{Wel13}, and a comparison can be found in~\cite{Chen22}.

Although the definition of the invariants $\ogw$ is quite abstract, they have a rich structure that makes them explicitly computable in many situations. This structure includes the open WDVV equations recalled in Theorem~\ref{thm:open_wdvv}, the open axioms recalled in Proposition~\ref{prop:open_axioms}, and the wall-crossing formula recalled in Theorem~\ref{wall_crossing}.
Moreover, the relative quantum cohomology ring of the pair $(X,L)$ recalled in Section~\ref{ssec:rqc} encodes the invariants $\ogw$ along with the genus zero closed Gromov-Witten invariants of $X$.

The genus zero closed Gromov-Witten invariants of $X$ are given by multilinear maps
\[
\gw_\beta : H^*(X;\Q)^{\otimes l} \to \Q
\]
for $\beta \in H_2(X;\Z)$ and $l \geq 0.$ For $A_1,\ldots, A_l \in H^*(X;\Q),$ the invariant $\gw_\beta(A_1,\ldots,A_l)$ counts $J$-holomorphic maps $S^2 \to X$ representing the class $\beta$ and passing through chosen cycles Poincar\'e dual to $A_1,\ldots,A_l.$ We recall the basic properties of these invariants in Section~\ref{sec:background}.

\subsection{Statement of results}
Let $L_\triangle$ denote the Chiang Lagrangian in $\cp^3$. In particular, $L_\tr$ is a rational homology sphere, so the open Gromov-Witten invariants $\ogw_{\beta,k}$ are defined. The definition of $L_\tr$ is recalled in Section~\ref{section3}.
Write $\Gamma_0 = [1] \in H^0(L_\triangle;\R) = \hH^0(\x,L_\triangle;\R),$ and for $j = 1,2,3,$ write
\[
\Gamma_j=[\omega^j]\in H^{2j}(\cp^3,L_\triangle;\R) = \hH^{2j}(\cp^3,L_\triangle;\R).
\]
Write $\Delta_j=[\omega^j] \in H^*(\cp^3;\R)$. Take $P_\mathbb{R}$ to be the unique left-inverse of $y$ such that $\ker(P_\mathbb{R})=\mathrm{span}\{\Gamma_2\}$. In light of Lemma~\ref{lm:relhom} below, we identify $H_2(\cp^3,L_\triangle;\Z) \simeq \Z.$ Our main results are the following.

\begin{thm}
\label{Theorem 1}
The open WDVV equations imply the following relations for the invariants $\ogw_{\beta, k}$ of $(\mathbb{C}P^3, L_\triangle).$
\begin{enumerate}
\item
\label{recusion a}
For $k \geq 1, l \geq 1 $, and $I = \{ j_2,\ldots,j_l\}$,
\begin{align*}
    &\ogw_{\beta, k} (\Gamma_{j_1}, \Gamma_{j_2},\ldots,\Gamma_{j_l})= \\
    &\qquad = -\sum_{\substack{\varpi(\hat{\beta}) + \beta_1 = \beta \\ \beta_1 \not= \beta \\ I_1 \sqcup I_2 = I}} \sum_{i = 0}^{3}
    \gw_{\hat{\beta}}(\Delta_{j_1 -1}, \Delta_1, \Delta_{I_1}, \Delta_i)
    \ogw_{\beta_1, k} (\Gamma_{3-i}, \Gamma_{I_2})\\
    &\qquad \quad +\sum_{\substack{\beta_1 +\beta_2 = \beta \\ k_1 + k_2 = k - 1 \\
    I_1 \sqcup I_2 = I}}\binom{k-1}{k_1} \bigg( \ogw_{\beta_1, k_1} (\Gamma_{j_1 - 1}, \Gamma_1, \Gamma_{I_1}) \ogw_{\beta_2, k_2+2} (\Gamma_{I_2})\\
    & \ \ \ \ \ \ \ \ \ \ \ \ \ \ \ \ \ \ \qquad \quad -\ogw_{\beta_1, k_1 + 1}(\Gamma_{j_1 -1}, \Gamma_{I_1})\ogw_{\beta_2, k_2 + 1} (\Gamma_1, \Gamma_{I_2}) \bigg).
\end{align*}
\item
\label{recursion b}
For $k \geq 2, l \geq 0$, and $I = \{j_1,\ldots,j_l\}$,
\begin{align*}
       &\ogw_{\beta, k} (\Gamma_{j_1},\ldots,\Gamma_{j_l}) \ogw_{2, 0} (\Gamma_2, \Gamma_2) = \\
       & \qquad = \sum_{\substack{ \varpi(\hat{\beta}) + \beta_1 = \beta+2 \\I_1 \sqcup I_2 = I}}
       \sum_{i = 0}^{3} \gw_{\hat{\beta}}(\Delta_2, \Delta_2, \Delta_{I_1}, \Delta_i)
       \ogw_{\beta_1, k-1}(\Gamma_{3-i}, \Gamma_{I_2})\\
       &\qquad\quad + \sum_{\substack{\beta_1 + \beta_2 = \beta +2 \\ k_1 + k_2 = k-2 \\
       I_1 \sqcup I_2 = I} } \binom{k-2}{k_1}
       \ogw_{\beta_1, k_1 + 1} (\Gamma_2, \Gamma_{I_1})
       \ogw_{\beta_2, k_2 + 1} ( \Gamma_2, \Gamma_{I_2})\\
       & \qquad\quad -\sum_{\substack{\beta_1 + \beta_2 = \beta + 2 \\ k_1 + k_2 = k-2 \\
       I_1 \sqcup I_2 = I \\ (\beta_1, k_1) \neq (\beta, k-2) }} \binom{k-2}{k_1}
       \ogw_{\beta_1, k_1 + 2}(\Gamma_{I_1})
       \ogw_{\beta_2, k_2}(\Gamma_2, \Gamma_2, \Gamma_{I_2}).
\end{align*}
\item
\label{recursion c}
For $l \geq 2$, and $I = \{j_3,\ldots,j_l\}$,
\begin{align*}
       &\ogw_{\beta, k} (\Gamma_{j_1},\ldots,\Gamma_{j_l})= \\
       &\qquad =\ogw_{\beta, k} (\Gamma_{j_1-1},\Gamma_{j_2+1},\Gamma_{j_3},\ldots,\Gamma_{j_l})\\
       &\qquad\quad + \sum_{\substack{ \varpi(\hat{\beta}) + \beta_1 = \beta
       \\\beta_1 \neq \beta
       \\I_1 \sqcup I_2 = I}}
       \sum_{i = 0}^{3}
       \bigg(\gw_{\hat{\beta}}(\Delta_1, \Delta_{j_2}, \Delta_{I_1}, \Delta_i)
       \ogw_{\beta_1, k}(\Gamma_{3-i}, \Gamma_{j_1-1}, \Gamma_{I_2})\\
       &\qquad \quad \ \ \ \ \ \ \ \ \ \ \ \ \ \ \ \ \ \  -\gw_{\hat{\beta}}(\Delta_1, \Delta_{j_1-1}, \Delta_{I_1}, \Delta_i)
       \ogw_{\beta_1, k}(\Gamma_{3-i}, \Gamma_{j_2}, \Gamma_{I_2})\bigg)\\
       &\qquad\quad +\sum_{\substack{\beta_1 + \beta_2 = \beta \\ k_1 + k_2 = k \\
       I_1 \sqcup I_2 = I }}
        \binom{k}{k_1} \bigg(
       \ogw_{\beta_1, k_1 }(\Gamma_1, \Gamma_{j_1-1}, \Gamma_{I_1})
       \ogw_{\beta_2, k_2+1}(\Gamma_{j_2}, \Gamma_{I_2})\\
       &\qquad \quad \ \ \ \ \ \ \ \ \ \ \ \ \ \ \ \ \ \ -\ogw_{\beta_1, k_1}(\Gamma_1, \Gamma_{j_2}, \Gamma_{I_1})\ogw_{\beta_2, k_2+1}(\Gamma_{j_1-1}, \Gamma_{I_2})\bigg).
\end{align*}
\end{enumerate}
\end{thm}
\begin{thm}\label{thm:iv}
Consider the choice of spin structure and orientation on $L_\tr$ given in Section~\ref{subsection_orientation_spin}. Then
\[
\ogw_{1,1}=-3, \qquad \ogw_{1,0}(\Gamma_2)=\frac{1}{4},  \qquad \ogw_{2,0}(\Gamma_3)=1.
\]
It follows from the open WDVV equations that $\ogw_{2,0}(\Gamma_2, \Gamma_2)=\frac{35}{64}.$
\end{thm}
We call the invariants $\ogw_{1,1}, \ogw_{1,0}(\Gamma_2)$ and $\ogw_{2,0}(\Gamma)$ \emph{basic invariants}. Theorem~\ref{theorem_lm} below shows that these invariants are given by straightforward counts of $J$-holomorphic disks without correction terms. Thus, we can compute the basic invariants using the theory of axial disks developed by Evans-Lekili~\cite{EvansLekili} and Smith~\cite{Smith,smith2020monotone} for computing the Floer cohomology of the Chiang Lagrangian and other homogeneous Lagrangian submanifolds. Since the Chiang Lagrangian is not fixed by an anti-symplectic involution, the techniques of real algebraic geometry used in other computations of open Gromov-Witten invariants~\cite{chen-zinger,Horev_Solomon,kai,RelativeQuantumCohomology} are not available.

\begin{cor}
\label{Corollary 1.2}
The genus zero open Gromov-Witten invariants of $(\mathbb{C}P^3, L_\triangle)$ are entirely determined by the open WDVV equations, the axioms of $\ogw$, the wall-crossing formula Theorem \ref{wall_crossing}, the genus zero closed Gromov-Witten invariants of $\mathbb{C}P^3$, and the values computed in Theorem~\ref{thm:iv}.
\end{cor}
Samples values for the invariants $\ogw$ computed using Theorems~\ref{Theorem 1} and~\ref{thm:iv} are given in Tables~\ref{table:boundary} and~\ref{table:interior}.

\begin{table}[ht]
\begin{center}
\renewcommand{\arraystretch}{1.4}
\begin{tabular}{|c|c|c||c|c|}
\hline
    $\beta=k$ & $\ogw_{\beta,k}$ exact value & round value &   $\beta=k$ & $\ogw_{\beta,k}$ round value \\
    \hline
    $1$  & $-3$ & $-3.00$ &  $17$ & $-1.38\cdot 10^7$ \\
    \hline
    $2$  & $-\frac{5}{4}$ & $-1.25$ &  $18$ & $-7.68\cdot 10^7$ \\
    \hline
    $3$  & $\frac{7}{16}$ & $0.44$ &  $19$ &  $3.17\cdot 10^8$ \\
    \hline
    $4$  & $\frac{3}{4}$ & $0.75$ &  $20$ &  $6.78\cdot 10^9$ \\
    \hline
    $5$  & $\frac{105}{256}$ & $0.41$ &  $21$  & $2.24\cdot 10^{10}$ \\
    \hline
    $6$  & $-\frac{85}{64}$ & $-1.33$  & $22$ &  $-3.73\cdot 10^{11}$ \\
    \hline
    $7$  & $-\frac{16005}{4096}$ & $-3.91$ & $23$ &  $-4.65\cdot 10^{12}$ \\
    \hline
    $8$  & $0$ & $0$  & $24$ &  $-8.08\cdot 10^{10}$ \\
    \hline
    $9$  & $\frac{2123349}{65536}$ & $32.40$ & $25$ &  $4.84\cdot 10^{14}$ \\
    \hline
    $10$  & $\frac{91035}{1024}$ & $88.90$   & $26$  & $4.05\cdot 10^{15}$ \\
     \hline
    $11$  & $-\frac{201485745}{1048576}$ & $-192.15$ &  $27$ &  $-2.48\cdot 10^{16}$ \\
    \hline
    $12$  & $-\frac{9045}{4}$ & $-2261.25$  & $28$ &  $-7.68\cdot 10^{17}$ \\
     \hline
    $13$  & $-\frac{72025175295}{16777216}$ & $-4293.03$ &  $29$ &  $-3.62\cdot 10^{18}$ \\
     \hline
    $14$  & $\frac{695299995}{16384}$ & $42437.74$ &  $30$ &  $8.50\cdot 10^{19}$ \\
    \hline
    $15$  & $\frac{87325406388675}{268435456}$ & $325312.49$ &  $31$ &  $1.47\cdot 10^{21}$ \\
    \hline
    $16$  & $4860$ &  $4860.00$  & $32$ &  $1.82\cdot 10^{19}$ \\
     \hline
\end{tabular}
\caption{Values of invariants with boundary constraints only.}
\label{table:boundary}
\end{center}
\end{table}

\begin{table}[ht]
\begin{center}
\renewcommand{\arraystretch}{1.4}
\begin{tabular}{l|c c c c c c c c}
\backslashbox{$l_3$}{$\beta$}& $1$ & $2$ & $3$ & $4$ & $5$ & $6$ & $7$ & $8$\\
\hline
$0$ & $\frac{1}{4}$ & $\frac{35}{64}$ & $\frac{507}{1024}$ & $\frac{723}{1024}$ & $\frac{427725}{262144}$ & $\frac{1180259}{262144}$ & $\frac{839314095}{67108864}$ & $\frac{39117}{1024}$\\
$1$ & $0$ & $1$ & $\frac{3}{8}$ & $\frac{11}{32}$ & $\frac{1251}{2048}$ & $\frac{5003}{4096}$ & $\frac{1481235}{524288}$ & $\frac{15033}{2048}$ \\
$2$ & $0$ & $0$ & $0$ & $\frac{1}{4}$ &$\frac{5}{16}$ & $\frac{23}{64}$ & $\frac{2943}{4096}$ & $\frac{25}{16}$\\
$3$ & $0$ & $0$ & $0$ & $0$ & $0$ & $0$ & $\frac{7}{32}$ & $\frac{3}{8}$ \\
\end{tabular}
\caption{Values of $\ogw_{\beta,0}(\Gamma_2^{\otimes l_2},\Gamma_3^{\otimes l_3})$. The value of $l_2$ is determined by $\beta$ and $l_3$ by the open degree axiom.}
\label{table:interior}
\end{center}
\end{table}

\subsubsection{Monotonicity versus periodicity}
In other computations~\cite{Itenberg_logarithmic, itenberg2015welschinger, itenberg2018relative, quintic_3-fold,RelativeQuantumCohomology,Horev_Solomon,tehrani,chen-zinger}, the absolute values of $\ogw_{\beta,k}$ or invariants of a similar flavor are monotonically increasing in $\beta.$ In the case of $(\x,L_\tr)$, we see that these values are not monotonically increasing. Interestingly, it appears that the only persistent violation of monotonicity occcurs periodically: $|\ogw_{\beta-1,\beta-1}|>|\ogw_{\beta,\beta}|$ for $\beta\in 8\Z$ at least up to $\beta = 32.$

Another unusual periodicity is visible in the signs of the invariants presented in Table~\ref{table:boundary}, which repeat with period $16.$ Moreover, the sign is inverted when $\beta$ is increased by~$8.$

\subsubsection{Bounding cochain corrections and denominators}\label{sssec:bcc&d}
In \cite{EvansLekili}, Evans-Lekili proved that there are exactly two $J$-holomorphic disks representing the class $\beta = 2 \in H_2(\x,L_\tr;\Z)$ passing through a specific choice of two points in $L_\tr$. This does not coincide with our computation  $\ogw_{2,2}=-\frac{5}{4}$. The discrepancy reflects correction terms arising from a bounding cochain in the Fukaya $A_\infty$ algebra of $L_\tr.$ The correction terms need not be whole numbers. In fact, all invariants with only boundary constraints that are not whole numbers indicate the presence of non-trivial correction terms. Table~\ref{table:boundary} provides several examples. When interior constraints are present, open Gromov-Witten invariants may be non-integral even in the absence of correction terms~\cite{Horev_Solomon}.

As far as we have calculated, for the Chiang Lagrangian $L_\tr$, the denominators appearing in the invariants $\ogw_{\beta,k}$ with only boundary constraints are always powers of $4.$ In the presence of interior constraints, powers of $2$ appear as well. Examining the proof of the existence of bounding cochains in~\cite{solomon2016point}, we can see where denominators of $4$ arise. Namely, the proof uses a $1$-cochain $h$ with real coefficients with coboundary equal to a $2$-cocycle $g$ representing the generator of $H^2(L_\tr;\Z) \simeq \Z/4\Z.$ Since $[4g] = 0 \in H^2(L_\tr;\Z),$ there exists an integral $1$-cochain $\tilde h$ with coboundary equal to $4g$. So, we can take $h = \frac{\tilde h}4.$ This is the source of the powers of $4$ in the denominators of the invariants $\ogw_{\beta,k}.$

In contrast, the invariants $\ogw_{\beta,k}$ with only boundary constraints for the Lagrangian submanifold $\R P^{n} \subset \C P^{n}$ with $n$ odd are shown in~\cite{RelativeQuantumCohomology} to be whole numbers. For $n = 3,$ these invariants are shown
in~\cite{solomon2016point} to be straightfoward disk counts, but for $n \geq 5$ they presumably include correction terms arising from bounding cochains. Since $H^{2i}(\R P^n;\Z) \simeq \Z/2$ for $0 < i < n/2,$
by the reasoning of the preceding paragraph, one might expect to see powers of $2$ in the denominators of the invariants $\ogw_{\beta,k}.$ However, since $\R P^n$ is fixed by an anti-symplectic
involution, i.e. complex conjugation, holomorphic disks come in pairs. So, in the proof of the existence of bounding cochains in~\cite{solomon2016point}, one needs only a $(2i-1)$-cochain $h$ with coboundary equal to \emph{twice} a $2i$-cocycle $g$ representing the generator of $H^{2i}(\R P^n;\Z)$. But $[2g] = 0 \in H^{2i}(\R P^n;\Z)$, so $h$ can be chosen to be integral. This explains the absence of denominators. A similar argument should apply for other Lagrangian submanifolds $L$ fixed by an anti-symplectic involution. If all torsion in $H^{2i}(L;\Z)$ has order $2$, then the invariants $\ogw_{\beta,k}$ should be integral. More generally, if $H^{2i}(L;\Z)$ contains torsion of even order, the rate of growth of the power of $2$ in the denominator of $\ogw_{\beta,k}$ as a function of $\omega(\beta)$ should be slower when $L$ is fixed by an anti-symplectic involution than otherwise. In particular, the behavior of denominators in the invariants $\ogw_{\beta,k}$ could potentially be used to prove that a Lagrangian submanifold is not fixed by an anti-symplectic involution.

\subsubsection{Mixed disk and sphere invariants}

As mentioned above, the invariants $\ogw_{\beta,k}$ for $\beta \in \Ima\varpi$ and $k = 0$ count $J$-holomorphic disks and $J$-holomorphic spheres together. In other computations~\cite{RelativeQuantumCohomology,kai}, these mixed disk and sphere invariants vanish for a natural choice of $P_\R.$ Another family of invariants combining counts of $J$-holomorphic maps of different topology has also been shown to vanish in certain examples~\cite{tehrani}.

We show that the the mixed disk and sphere invariants of $(\cp^3,L_\tr)$ do not vanish.
As shown in Lemma~\ref{lm:relhom} below, the map $\varpi : H_2(\x;\Z) \to H_2(\x,L_\tr;\Z)\simeq \Z$ is given by multiplication by $4.$ So, the non-vanishing of mixed disk and sphere invariants is visible in Table~\ref{table:interior} for $\beta= 4,8.$ In fact, Corollary~\ref{cor:map_p} below asserts that this non-vanishing persists for any choice of $P_\R.$

Let $\rho: \widehat{H}^{4}(\x,L_\tr;\R)\rightarrow H^{4}(\x;\R)$ denote the natural map. Let $P_\R,P_\R'$ be two choices of left inverse maps of $y$ with associated invariants $\ogw_{\beta,k}$ and $\ogw\,'\!\!_{\beta,k}$ respectively. The long exact sequence of the pair $(\cp^3,L_\tr)$ implies there exists a map
\[\mathfrak{p}_\R:H^{4}(\x;\R)\rightarrow H^3(L_\tr;\R)\]
such that $\mathfrak{p}_\R\circ\rho=P_\R-P_\R'.$ We use Poincar\'e duality to identify $H^3(L_\tr;\R)\simeq \R.$
\begin{prop}\label{prop_map_p}
For $\beta \in \Ima \varpi,$ we have
\[\ogw_{\beta,0}(\Gamma_{i_1},\ldots,\Gamma_{i_l})-\ogw\,'\!\!_{\beta,0}(\Gamma_{i_1},\ldots,\Gamma_{i_l})=\frac{\beta}{4}\gw_{\frac{\beta}{4}}(\Delta_{i_1},\ldots,\Delta_{i_l})\mathfrak{p}_\R(\Delta_2).\]
\end{prop}
\begin{cor}\label{cor:map_p}
There is no map $P_\R:\widehat{H}^{4}(\cp^3,L_\tr;\R)\rightarrow H^3(L_\tr;\R)$ such that $\ogw_{\beta,k}$ vanishes for every $\beta\in \Ima\varpi$ and $k=0.$
\end{cor}
Proposition~\ref{prop_map_p} is a special case of Proposition~\ref{prop:general_map_p}, which gives the analogous statement for $X$ a general symplectic manifold and $L$ a connected spin Lagrangian submanifold with $H^*(L;\R) \simeq H^*(S^n;\R)$ and $[L] = 0 \in H_n(X;\R).$ In fact, Proposition~\ref{prop:general_map_p} remains valid in all situations where the invariants $\ogw$ are defined in~\cite{RelativeQuantumCohomology}.

\subsubsection{Relative quantum cohomology}
It follows from Corollary~\ref{Corollary 1.2} that we can compute the big relative quantum cohomology of $(\cp^3,L_\tr).$ However, the resulting ring does not appear to have a tractable presentation by generators and relations. On the other hand, the small relative quantum cohomology of $(\cp^3,L_\tr)$ is more accessible.
\begin{thm}\label{relativequantumpro}
The small relative quantum cohomology of $(\x,L_\triangle)$ is given by
\[QH^*(\x,L_\triangle)\simeq\R[[q^{1/4}]][x,y]/I,\]
with \[I=( x^4-q-\frac{1}{2}q^{1/2}y-\frac{3}{64}q^{1/4}y,\ y^2+\frac{5}{4}q^{1/2}y,\  xy-\frac{3}{4}q^{1/4}y ).\]
\end{thm}
Computations of small relative quantum cohomology for a variety of pairs $(X,L)$ are given in~\cite{kai}. However, all Lagrangian submanifolds considered there are fixed by an anti-symplectic involution. The definitions of big and small quantum cohomology are recalled in Section~\ref{ssec:rqc}.

\subsubsection{Invariants without corrections}

The proof of Theorem~\ref{thm:iv} depends on the following general result, which says when the invariants $\ogw_{\beta,k}$ can be computed without taking into account correction terms from bounding cochains. For simplicity, we continue as in Section~\ref{ssec:background} with $L$ a connected spin Lagrangian with $H^*(L;\R) \simeq H^*(S^n;\R).$ A similar result holds in other settings where the invariants $\ogw$ can be defined.
For $\beta\in H_2(X,L;\Z)$, define
\begin{equation}\label{P_beta_definition}
    P_\beta=\{\Tilde{\beta}\in H_2(X,L;\Z)| \   \omega(\beta)>\omega(\Tilde{\beta})> 0 \}.
\end{equation}
Let $\sigma_k = (k-1)!$ for $k \in \Z_{>0}$ and let $\sigma_0 = 1.$
\begin{thm}\label{theorem_lm}
Let $k \in \Z_{\geq 0}$ and $\beta \in H_2(X,L;\Z)$ satisfy either $k > 0$ or $\beta\not\in\Ima \varpi$.
Let $A_{j}\in\widehat{H}^*(X,L;\mathbb{R})$. Let $\bar{b}\in A^n(L;\R)$ such that $\pd([\bar{b}])=pt$, and  let $a_{j}$ be a representative of $A_{j}$.
Suppose that for every
\[
\Tilde{\beta}\in P_\beta, \qquad 0\le j\le k,\qquad I\subset \{1,\ldots,l\},
\]
one of the following two conditions is satisfied:
\begin{enumerate}
\item $1-\mu(\Tilde{\beta}) +j(n-1) + \sum_{i\in I} (|A_{i}|-2)<0,$ or
\item $1-\mu(\Tilde{\beta}) +j(n-1) + \sum_{i\in I} (|A_{i}|-2)\ge n-1.$
\end{enumerate}
Then,
\[
\ogw_{\beta,k}(A_{1},\ldots,A_{l})= (-1)^n \sigma_k\int_{\mathcal{M}_{k,l}(\beta)} \bigwedge_{j=1}^{l} evi^{\beta*}_j  a_{j} \bigwedge_{j=1}^{k}evb^{\beta*}_j\bar{b}.
\]

\end{thm}
Theorem~\ref{theorem_lm} is a special case of Theorem~\ref{lm}. The basic invariants $\ogw_{1,1}, \ogw_{1,0}(\Gamma_2)$ and $\ogw_{2,0}(\Gamma_3)$ of Theorem~\ref{thm:iv} are exactly the invariants to which Theorem~\ref{theorem_lm} applies for $L$ the Chiang Lagrangian.

\subsubsection{Orientation, spin structure and signs}\label{sssec:oss}
We prove general formulas governing the dependence of the invariants $\ogw_{\beta,k}$ on the spin structure and orientation of $L$, which appear as the spin axiom~\ref{ax_spin} and the orientation axiom~\ref{ax_orientation} of Proposition~\ref{prop:open_axioms}. See also Lemmas~\ref{lm_spin} and~\ref{lm_orientation}. We use the spin and orientation axioms to deduce vanishing results for certain open Gromov-Witten invariants, which are formulated in Corollaries~\ref{cor:projvan} and~\ref{cor:quadricvan}.

A significant part of the present paper is devoted to computing the signs of the three basic invariants of Theorem~\ref{thm:iv}. The spin and orientation axioms show that by changing the orientation and spin structure on $L_\tr$, one can adjust the signs of any two of the basic invariants arbitrarily. Moreover, changing the orientation and spin structure only affects the signs of invariants but not their absolute values.

On the other hand, it can be seen from the proof of Lemma~\ref{lm:ogw2022} that if the sign of one of the three basic invariants is changed while the two others are held fixed, then the absolute values of non-basic invariants change. Qualitatively, the denominators of invariants cease to be powers of two.

\subsubsection{Arithmetic structure of the open WDVV equations}\label{sssec:arith}
From the perspective of Theorem~\ref{Theorem 1}, it is surprising that only powers of $2$ appear in the denominators of the invariants $\ogw$ for the Chiang Lagrangian $L_\tr.$ Indeed, to use relation~\ref{recursion b} for recursive computations, it is necessary to divide by $\ogw_{2,0}(\Gamma_2,\Gamma_2) = \frac{35}{64}.$ A priori, this should introduce factors of $5$ and $7$ in denominators, but unexpected cancellations occur so that only factors of $2$ remain.

To test the rigidity of these cancellations, we searched for rational numbers $v \neq 1/4$ that upon substitution for the value of $\ogw_{1,0}(\Gamma_2)$ in the recursive scheme of Corollary~\ref{Corollary 1.2} yield $\ogw_{k,k}$ with denominator a power of $2$ for $k = 1,\ldots,M.$ As $M$ increases, the magnitudes of the possible numerators and denominators of $v$ increase. For $M = 4,$ the $v$ with smallest numerator and denominator that we found is
\[
v=-\frac{46,912,496,118,431}{17,592,186,044,416}\approx -2.66.
\]
For this choice of $v,$ the value for $\ogw_{5,5}$ given by the recursive scheme has denominator $7.$

Therefore, we reach the following conclusion. Suppose it were possible to formalize the heuristic argument of Section~\ref{sssec:bcc&d} to prove geometrically that the invariants $\ogw_{\beta,k}$ can only have powers of $2$ for denominators. Then, it seems reasonable to expect that the basic invariant $\ogw_{1,0}(\Gamma_2)$ could be deduced from the other two basic invariants, the open WDVV equations, the axioms of $\ogw$, and the closed Gromov-Witten invariants of $\cp^3.$

If it were necessary to compute only two basic invariants geometrically, it would not be necessary to compute their signs. Indeed, these signs could be adjusted arbitrarily by changing the orientation and spin structure on $L_\tr$ as noted in Section~\ref{sssec:oss}. Thus, the computations of this paper could be simplified significantly. In fact, this strategy can be generalized for an arbitrary pair $(X,L)$ as outlined below.

\subsection{Directions for future research}\label{ssec:directions}
Remarkably, Theorem~\ref{theorem_lm} applies exactly for those invariants that are needed as initial values for the recursions given by Theorem~\ref{Theorem 1}. It is an interesting question to determine to what extent this phenomenon persists more generally.

We expect the techniques of the present paper to extend to the other Platonic Lagrangians studied by Smith~\cite{Smith} and more general homogeneous Lagrangians~ \cite{bedulli2006homogeneous,gasparim,KonstantinovThesis}

Motivated by the discussion of Section~\ref{sssec:arith}, we formulate the following two hypotheses. These hypotheses should hold for a large class of Lagrangian submanifolds $L \subset X$, but some restrictions are necessary.
\begin{enumerate}[label=(\Roman*)]
\item~\label{hp:denom}
The invariants $\ogw_{\beta,k}$ are rational of the form $\frac{m}{p_1^{a_1}\cdots p_N^{a_N}}$ where $p_1,\ldots,p_N,$ are the primes that occur as orders of torsion elements in~$H^*(L;\Z).$
\item~\label{hp:arWDVV}
Let $\cA$ be a minimal set of the invariants $\ogw_{\beta,k}$ from which all other invariants can be deduced using the open WDVV equations, the axioms of $\ogw$, the wall-crossing formula, and the closed GW invariants of $X$. Let $\cB \subset \cA$ be defined in the same way as $\cA$ except that we are also allowed to use hypothesis~\ref{hp:denom} to deduce the other invariants. Then $\cB$ is a proper subset of $\cA.$
\end{enumerate}
Hypothesis~\ref{hp:denom} follows from the heuristic argument of Section~\ref{sssec:bcc&d} applied to $L$ of general topology. It should admit a geometric proof.
Hypothesis~\ref{hp:arWDVV} is a statement about the arithmetic structure of the open WDVV equation. In the case of the Chiang Lagrangian, we can take $\cA$ to be the basic invariants $\ogw_{1,1}, \ogw_{1,0}(\Gamma_2)$ and $\ogw_{2,0}(\Gamma_3)$. Then, the calculations presented in Section~\ref{sssec:arith} indicate that the invariants $\ogw_{1,1}$ and $\ogw_{2,0}(\Gamma_3)$ suffice for $\cB.$ On the other hand, in the case $(X,L) = (\C P^n, \R P^n)$ one can see that hypothesis~\ref{hp:arWDVV} does not hold as follows. It is shown in~\cite[Corollary 1.9]{RelativeQuantumCohomology} that one can take $\cA$ to consist of a single invariant. It follows from either the spin axiom~\ref{ax_spin} or the orientation axiom~\ref{ax_orientation} of Proposition~\ref{prop:open_axioms} that by changing the spin structure or orientation on $L$, the sign of this single invariant can be changed without affecting the absolute value of the others, in particular, without violating hypothesis~\ref{hp:denom}. Thus $\cB$ cannot be smaller than $\cA.$ Another case where hypothesis~\ref{hp:arWDVV} does not hold is when $X$ is the quadric hypersurface in $\C P^{n+1}$ given by $\sum_{i = 0}^n z_i^2 - z_{n+1}^2 = 0$ and $L$ is the real locus. This is explained in Remark~\ref{rem:quadricnohparWDVV} based on results of~\cite{kai}, which are related to the orientation axiom.

We observe the following common feature of the forgoing exceptions to hypothesis~\ref{hp:arWDVV}. In both cases, one can take $\cA$ to contain a single invariant, while there are one or more geometric degree of freedom affecting the values of the invariants: in the case $(X,L) = (\C P^n,\R P^n)$ there are choices of spin structure and orientation, and in the case of the quadric, there is the choice of orientation. The presence of more geometric degrees of freedom than the size of $\cA$ forces a measure of flexibility in the open WDVV equations. In contrast, for the the Chiang Lagrangian, $\cA$ contains three invariants, while there are only two geometric degrees of freedom, the choices of spin structure and orientation. So, the equations can be sufficently rigid for hypothesis~\ref{hp:arWDVV} to hold.

\subsection{Outline}
In Section~\ref{sec:background} we recall general definitions and results on closed and open Gromov-Witten invariants that will be useful in the other sections. This includes the closed and open axioms, the closed and the open WDVV equations, the wall-crossing formula, and the definitions of small and big relative quantum cohomology.

In Section~\ref{section3} we recall from~\cite{Smith} a family of Lagrangian homology spheres in Fano threefolds called Platonic Lagrangians. We focus mainly on the Chiang Lagrangian $L_\tr$, which is a special case of this construction, and recall some of its basic properties.

Section~\ref{sec4} contains results related to Fukaya $A_\infty$ algebras, bounding cochains and the superpotential. We recall definitions and results on obstruction theory for bounding cochains from~\cite{solomon2016point}. These are used to prove Theorem~\ref{lm}, which gives a criterion for open Gromov-Witten invariants to coincide with straightforward counts of $J$-holomorphic disks. Theorem~\ref{theorem_lm}, which is needed for Section~\ref{section6}, is obtained as a special case. We also prove Lemmas~\ref{lm_spin} and~\ref{lm_orientation}, from which we deduce the spin axiom~\ref{ax_spin} and the orientation axiom~\ref{ax_orientation} of Proposition~\ref{prop:open_axioms}.

The goal of Section~\ref{section6} is to compute the basic invariants that serve as initial values for the recursive formulas. In Sections~\ref{subsection6.1}-\ref{subsection5.3} we recall definitions and results about the anticanonical divisor, the Maslov class, and axial disks. In Section~\ref{choose_spin_structure} we give the definition of a spin Riemann-Hilbert pair, recall results on canonical orientations, and prove some new ones. Section~\ref{subsection_5.5} contains results concerning Riemann-Hilbert pairs arising from a certain type of axial disk.
We use these results to compute the basic invariants geometrically in Sections~\ref{subsection_ogw11}-\ref{subsection_ogw20}.

Section~\ref{section5} contains the proofs of Theorem~\ref{Theorem 1} and Corollary~\ref{Corollary 1.2}, which give recursive formulas that determine all open Gromov-Witten invariants of $(\cp^3,L_\tr)$ from the three basic invariants.

In Section~\ref{section7} we compute the small relative cohomology of $(\cp^3, L_\tr)$.

In Section~\ref{section8} we recall definitions and results concerning the map $P_\R$ from  \cite{RelativeQuantumCohomology}. These results are used to prove Proposition~\ref{prop:general_map_p}, which shows how the invariants $\ogw$ depend on the map $P_\R.$ Proposition~\ref{prop_map_p} is obtained as a special case. Finally, we prove Corollary~\ref{cor:map_p}.

\subsection{Acknowledgements}
The authors would like to thank J. Evans, O. Kedar, C.-C.~M. Liu, J. Smith and S. Tukachinsky for helpful conversations. The authors were partially funded by ISF grants 569/18 and 1127/22.

\section{Background}\label{sec:background}
In this section we recall definitions and properties of open and closed Gromov-Witten invariants and relative quantum cohomology.

\subsection{Invariants}\label{subsection_ogw}
Let $(X,\omega)$ be a closed symplectic manifold.
For $\beta \in H_2(X;\mathbb{Z})$ and $l \geq 0,$ let
\[
GW_\beta : H^*(X;\Q)^{\otimes l} \rightarrow \Q
\]
denote the standard closed genus zero Gromov-Witten invariant. This invariant counts $J$-holomorphic maps $S^2 \to X$ representing the class $\beta$ passing through cycles Poincar\'e dual to $l$ cohomology classes on $X.$

Let $L \subset X$ be a closed Lagrangian submanifold. Genus zero open Gromov-Witten invariants are analogous invariants that count $J$-holomorphic maps $(D,\partial D) \to (X,L)$ along with correction terms arising from bounding cochains. To simplify the exposition, we focus on the case that $L$ is connected and spin, with $H^*(L;\mathbb{R})\simeq H^*(S^n; \mathbb{R})$ and $[L] = 0 \in H_n(X;\R)$. These conditions hold for the Chiang Lagrangian.
Similar results are known~\cite{solomon2016point} in the case that $L$ is only relatively spin or $[L] \neq 0 \in H_n(X;\R).$

Consider the subcomplex of differential forms on $X$ consisting of those with trivial integral on $L,$
\[\widehat{A}^*(X,L;\R):=\left\{\eta\in A^*(X;\R)\Big| \int_L \eta|_L=0\right\}.\]
For an $\R$-algebra $\Upsilon$, write \[\widehat{A}^*(X,L;\Upsilon):= \widehat{A}^*(X,L;\R)\otimes \Upsilon,\quad\widehat{H}^*(X,L;\Upsilon):= H^*(\widehat{A}^*(X,L;\Upsilon), d).\] Observe that \[\widehat{H}^*(X,L;\Upsilon) \simeq H^0(L;\Upsilon)\oplus H^{>0}(X,L;\Upsilon).\]
For $\beta \in H_2(X,L;\mathbb{Z})$ and $k,l \geq 0,$ let
\[
\oogw_{\beta,k} : \widehat{H}^*(X,L;\R)^{\otimes l}\rightarrow \mathbb{R}
\]
denote the open Gromov-Witten invariants of~\cite{solomon2016point}. The definition is recalled in Section~\ref{subsection4.2}.
For $\beta \in H_2(X,L;\mathbb{Z})$ and $k,l \geq 0,$ let
\[ \ogw_{\beta,k} : \widehat{H}^*(X,L;\R)^{\otimes l}\rightarrow \mathbb{R} \]
denote the enhanced invariants defined in~\cite{RelativeQuantumCohomology}. The definition is recalled in~\ref{section8}. These invariants combine counts of $J$-holomorphic disks and $J$-holomorphic spheres. With the exception of the case that $\beta$ belongs to the image of the natural map
\[\varpi : H_2(X;\mathbb{Z}) \rightarrow H_2(X,L;\mathbb{Z})\]
and $k = 0,$ we have $\ogw _{\beta,k} = \oogw_{\beta,k}.$ In the exceptional case, $\oogw_{\beta,k} = 0,$ but $\ogw_{\beta,k}$ need not vanish.
By the assumption $[L] = 0 \in H_n(X;\R),$ there is a short exact sequence
\begin{equation}\label{ses}
    0 \rightarrow H^n(L;\R) \overset{y}{\longrightarrow} {H}^{n+1}(X,L;\R) \overset{\rho}{\longrightarrow}  H^{n+1}(X;\R) \rightarrow 0.
\end{equation}
For $k=0$ and $\beta\in\mathrm{Im \varpi}$, the definition of the enhanced invariants $\ogw_{\beta,k}$ depends on the choice of a linear map
\[P_\mathbb{R} :{H}^{n+1}(X,L;\R) \rightarrow  H^{n}(L;\R)\simeq \R \]
that is a left-inverse to the map $y$ in the short exact sequence. A choice of $P_\mathbb{R}$ is equivalent to a splitting of~\eqref{ses}. We extend $P_\R$ to a map $P_\R: \widehat{H}^*(X,L;\R) \to H^n(L;\R)$ by setting it to zero outside $\widehat{H}^{n+1}(X,L;\R)$.
\subsection{Axioms}
The closed Gromov-Witten invariants satisfy the following properties~ \cite{Behrend,Fukaya_Ono, KontsevichManin,Li_Tian, mcduffsalamon2012, McDuffSalamon_smallbook,  Ruan_Tian_1994, Ruan_Tian_1995}.
\begin{prop}
\label{prop:closed_axioms}
Let $\beta\in H_2(X,\mathbb{Z})$ and let $A_1,\ldots,A_k\in H^*(X;\R)$.
\begin{enumerate}[label=(\arabic*)]
\item
(Effective) If $\omega(\beta)<0$, then $\gw_{\beta}(A_1,\ldots,A_k)=0.$
\item
(Symmetry) For each permutation $\sigma\in S_k$
\[\gw_{\beta}(A_{\sigma(1)},\ldots,A_{\sigma(k)})=(-1)^{s_\sigma(A)}\gw_{\beta}(A_1,\ldots,A_k),\]
where
$s_\sigma(A):=\sum_{\substack{i<j\\ \sigma(i)<\sigma(j)}}|A_i|\cdot|A_j|.$
\item
\label{ax_degree_GW}
(Degree) If $GW_{\beta}(A_1,\ldots,A_k)\ne 0,$ then
$$\sum_{i=1}^k|A_i|-2k+6= 2n+ 2c_1(\beta).$$
\item
(Fundamental Class) If $(\beta,k)\ne (0,3)$ and $k\ge1,$ then
$$ GW_{\beta}(A_1,\ldots,A_{k-1},1)=0.$$
\item\label{closed_divisor_ax}
(Divisor) If $(\beta,k)\ne (0,3),$ $|A_k|=2$ and $k\ge1,$ then
$$\gw_{\beta}(A_1,\ldots,A_k)=\gw_{\beta}(A_1,\ldots,A_{k-1})\cdot \int_\beta A_k.$$
\item\label{zerooo}
(Zero) If $k \ne 3$ then $GW_{0}(A_1,\ldots,A_k)=0$. If $k=3,$ then
$$GW_{0}(A_1,A_2,A_3)= \int_XA_1\smile A_2\smile A_3.$$
\item
(Deformation invariance) The invariants $\gw_\beta$ remain constant under deformations of the symplectic form $\omega$.
\end{enumerate}
\end{prop}
The enhanced open invariants $\ogw$ satisfy the following properties.
\begin{prop}
\label{prop:open_axioms}
Let $\beta \in H_2(X,L;\Z)$, let $k \in \Z_{\geq 0}$ and let $A_1,\ldots,A_l\in \widehat{H}^*(X,L;\mathbb{R})$.
\begin{enumerate}[label=(\arabic*)]
\item \label{ax_eff}
(Effective) If $\omega(\beta)<0,$ then $\ogw_{\beta,k}(A_1,\ldots,A_k)=0.$
\item \label{ax_symmetry}
(Symmetry) For each permutation $\sigma\in S_l$
\[\ogw_{\beta,k}(A_{\sigma(1)},\ldots,A_{\sigma(l)})=(-1)^{s_\sigma(A)}\ogw_{\beta}(A_1,\ldots,A_l),\]
where
$s_\sigma(A):=\sum_{\substack{i<j\\ \sigma(i)<\sigma(j)}}|A_i|\cdot|A_j|.$
\item \label{ax_degree}
(Degree) If $\ogw_{\beta,k}(A_1,\ldots,A_l)\ne0,$ then
\[n-3+\mu(\beta)+k+2l=kn+\sum_{j=1}^l | A_j|.\]
\item \label{ax_unit}
(Unit/ Fundamental class)
\[\ogw_{\beta,k}(1,A_1,\ldots,A_{l-1})=
         \left\{\begin{array}{ll}
        -1, &  (\beta,k,l)=(\beta_0,1,1)\\
        P_\mathbb{R}(A_1), & (\beta,k,l)=(\beta_0,0,2)\\
        0, & \mathrm{otherwise}.
        \end{array} \right.
  \]
\item \label{ax_zero}
(Zero)
\[\ogw_{\beta_0,k}(A_1,\ldots,A_l)=
         \left\{\begin{array}{ll}
        -1, &  (k,l)=(1,1)\ and\ A_1=1\\
        P_\mathbb{R}(A_1\smile A_2), & (k,l)=(0,2)\\
        0, &\mathrm{otherwise}.
        \end{array} \right.
  \]
\item \label{ax_divisor}
(Divisor) If $| A_l|=2$ then
\[\ogw_{\beta,k}(A_1,\ldots,A_l)= \int_\beta A_l\cdot\ogw_{\beta,k}(A_1,\ldots,A_{l-1}).\]
\item\label{ax_spin}
(Spin) Changing the spin structure on $L$ by the action of $\alpha\in H^1(L;\Z/2\Z)$ changes the sign of $\ogw_{\beta,k}(A_1,...,A_l)$ by the sign $(-1)^{\alpha(\partial\beta)}.$
\item\label{ax_orientation}
(Orientation) Changing the orientation of $L$
changes the sign of $\ogw_{\beta,k}(A_1,...,A_l)$ by the sign $(-1)^{k+1}.$
\item \label{ax_def}
(Deformation invariance) The invariants $\ogw_{\beta,k}$ remain constant under deformations of the symplectic form $\omega$ for which $L$ remains Lagrangian.
\end{enumerate}
\end{prop}
\begin{proof}
Axioms~\ref{ax_degree}-\ref{ax_divisor} are given in Proposition 4.19 from \cite{RelativeQuantumCohomology}. The analogs of Axioms~\ref{ax_symmetry} and~\ref{ax_def} for the invariants $\oogw$ are given in Theorem~4 of~\cite{solomon2016point}. The analogs of Axioms~\ref{ax_spin} and~\ref{ax_orientation} for the invariants $\oogw$ are given in Lemmas~\ref{lm_spin} and~\ref{lm_orientation} below respectively. The extension to the invariants $\ogw$ is similar to the proof of Proposition~4.19 in~\cite{RelativeQuantumCohomology}. Axiom~\ref{ax_eff} follows from the fact that the Novikov ring $\Lambda$ consists of power series in $T^\beta$ for $\omega(\beta) \geq 0.$ Indeed, the enhanced superpotential $\overline{\Omega}$ is defined in Section 1.3.3 in~\cite{RelativeQuantumCohomology} as an element of the ring $R_W.$
\end{proof}
\subsection{Bases}\label{subsection_bases}
Let $\Delta_i \in H^*(X;\R)$ for $i = 0,\ldots,N,$ be a basis and let $\Gamma_i \in \ker(P_\mathbb{R}) \subset \hH^*(X,L;\R)$ be the unique classes such that $\rho(\Gamma_i) = \Delta_i$. Let $\Gamma_{N+1} = y(1).$ We also write $\Gamma_\diamond = y(1).$ For convenience, set $\Delta_{N+1} = 0.$

\subsection{Novikov rings}\label{subsection_novikovrings}
Let $\mu : H_2(X,L;\Z) \to \Z$ denote the Maslov index.
Define Novikov coefficients rings
\begin{gather*}
\Lambda:=\{\sum_{i=0}^\infty a_iT^{\beta_i}|a_i\in \mathbb{R}, \beta_i\in H_2(X,L;\Z), \omega(\beta_i)\ge0, \lim_{i\rightarrow \infty}{\omega(\beta_i)}=\infty \},\\
\Lambda_c:=\{\sum_{j=0}^\infty a_jT^{\varpi(\beta_j)}|a_j\in \mathbb{R}, \beta_j\in H_2(X;\Z), \omega(\beta_j)\ge0, \lim_{j\rightarrow \infty}{\omega(\beta_j)}=\infty \}\leqslant \Lambda.
\end{gather*}
Gradings on $\Lambda, \Lambda_c$ are defined by declaring $|T^\beta| = \mu(\beta).$ Let
\begin{gather*}
R_W = \Lambda[[t_0,\ldots,t_{N+1},s]], \\
Q_W = \Lambda_c[[t_0,\ldots,t_{N+1}]], \qquad Q_U = \Lambda_c[[t_0,\ldots,t_N]].
\end{gather*}
Gradings on $R_W,Q_W,Q_U,$ are defined by declaring
$|t_i| = 2-|\Gamma_i|,\, |s| = 1-n.$

\subsection{Potentials}\label{potential}
The closed Gromov-Witten potential is given by
\[
\Phi(t_0,\ldots,t_N) =
\sum_{\substack{\beta \in H_2(X;\mathbb{Z}) \\ r_i \geq 0}} \frac{T^{\varpi(\beta)} t_N^{r_N} \cdots t_0^{r_0}}{r_N! \cdots r_0!}GW_\beta(\Delta_0^{\otimes r_0} \otimes \cdots \otimes \Delta_N^{\otimes r_N}) \in Q_U.\]
The enhanced superpotential, which encodes the enhanced open Gromov-Witten invariants $\ogw_{\beta,k},$ is given by
\begin{equation}\label{eq:esup}
\overline{\Omega}(s,t_0,\ldots,t_{N+1}) = \sum_{\substack{\beta\in H_2(X,L;\mathbb{Z})\\k\ge 0\\r_i\ge 0}}\frac{T^\beta s^k t_{N+1}^{r_{N+1}} \cdots t_0^{r_0}}{k! r_{N+1}! \cdots r_0!} \ogw_{\beta,k}(\Gamma_0^{\otimes r_0}\otimes \cdots \otimes \Gamma_{N+1}^{\otimes r_{N+1}}) \in R_W.
\end{equation}

\subsection{WDVV equations.}\label{subsection_g}
Let
$$
g_{ij} = \int_X \Delta_i \smile \Delta_j, \qquad i, j = 0,\ldots,N,
$$
and let $g^{ij}$ denote the inverse matrix. The following WDVV equations hold for the closed Gromov-Witten potential~\cite{Behrend,Fukaya_Ono,KontsevichManin,Li_Tian,mcduffsalamon2012, McDuffSalamon_smallbook,Ruan_Tian_1994, Ruan_Tian_1995,Witten_2d}.

\begin{thm}[Closed WDVV equations] For all quadruples of integers $i,j,k,l\in \{0,\ldots, n\}$, the function $\Phi=\Phi^{\mathbb{C}P^n}$ satisfies:
\[\sum_{\mu, \nu=0}^N \partial_{t_i} \partial_{t_j}\partial_{t_\nu}\Phi\cdot g^{\nu \mu}\cdot \partial_{t_\mu} \partial_{t_k} \partial_{t_l}\Phi= \sum_{\mu, \nu=0}^N \partial_{t_j} \partial_{t_k}\partial_{t_\nu}\Phi\cdot g^{\nu \mu}\cdot \partial_{t_\mu} \partial_{t_i} \partial_{t_l}\Phi \]
\end{thm}
The following is Corollary 1.6 from \cite{RelativeQuantumCohomology}.
\begin{thm}[Open WDVV equations]\label{thm:open_wdvv} For $u, v, w  = 0,\ldots, N+1,$ we have the open WDVV equations

\begin{multline}
 \label{eq15}
    \sum_{l,m = 0}^N \partial_{t_u} \partial_{t_l} \overline{\Omega} \cdot g^{lm} \cdot  \partial_{t_m} \partial_{t_w} \partial_{t_v} \Phi - \partial_{t_u} \partial_s \overline{\Omega} \cdot \partial_{t_w} \partial_{t_v} \overline{\Omega} = \\
     =\sum_{l,m = 0}^N   \partial_{t_u} \partial_{t_w} \partial_{t_l} \Phi \cdot g^{lm} \cdot \partial_{t_m} \partial_{t_v} \overline{\Omega}  - \partial_{t_u} \partial_{t_w} \overline{\Omega} \cdot \partial_{t_v} \partial_s \overline{\Omega},
\end{multline}
\begin{equation}\label{eq16}
\sum_{l,m = 0}^N \partial_s \partial_{t_l} \overline{\Omega} \cdot g^{lm} \cdot \partial_{t_m} \partial_{t_w} \partial_{t_v} \Phi - \partial_s^2 \overline{\Omega} \cdot \partial_{t_w} \partial_{t_v} \overline{\Omega} = -\partial_s \partial_{t_w} \overline{\Omega}\cdot \partial_{t_v} \partial_s \overline{\Omega}.
\end{equation}
\end{thm}
\subsection{Wall-crossing}\label{subsection_wallcrossing}
Define
\[\Gamma_{\diamond}=y(1) \in \widehat{H}^*(X,L;\R).\]
The following is Theorem 6 from \cite{RelativeQuantumCohomology}.
\begin{thm}[Wall crossing]
\label{wall_crossing} Suppose $[L]=0$, then the invariants $\ogw_{\beta,k}$ satisfy
$$\ogw_{\beta, k+1}(\eta_1, \ldots, \eta_l)=-\ogw_{\beta,k}(\Gamma_\diamond, \eta_1, \ldots, \eta_l)$$
\end{thm}

\subsection{Relative quantum cohomology}\label{ssec:rqc}
Abbreviate
\[\Delta= \sum_{i=0}^N t_i\Delta_i,\quad \Gamma= \sum_{i=0}^{N+1} t_i\Gamma_i.\]
The underlying module of the big relative quantum cohomology is
\[
QH^*_{\mathrm{big}}(X,L):= \widehat{H}^*(X,L;\Lambda[[t_0,\ldots,t_{N+1}]]).
\]
The big relative quantum product
\[\mem^{\mathrm{big}}: QH^*_{\mathrm{big}}(X,L)\otimes QH^*_{\mathrm{big}}(X,L)\rightarrow QH^*_{\mathrm{big}}(X,L)\]
is given by
\begin{multline}\label{big_rel_qunatum_cohomology}
\mem^{\mathrm{big}}(\Gamma_v, \Gamma_u)=\sum_{\substack{ 0\le l,m\le N \\ \beta\in H_2(X;\Z) \\ p\ge 0}}
\frac{T^{\varpi{(\beta)}}}{p!}\gw_\beta(\Delta^{\otimes p}\otimes \Delta_v \otimes \Delta_u \otimes \Delta_l)\cdot g^{lm}\cdot \Gamma_m +\\
+\sum_{\substack{\beta\in H_2(X,L;\Z)\\ p\ge0}} \frac{T^{\beta}}{p!}\ogw_{\beta,0}(\Gamma^{\otimes p}\otimes \Gamma_v\otimes \Gamma_u)\cdot \Gamma_\diamond.
\end{multline}
In Theorem 7 of~\cite{RelativeQuantumCohomology}, it is shown that the product $\mem^{\mathrm{big}}$ is associative and graded commutative.
Though in \cite{RelativeQuantumCohomology} the definition of $\mem^{\mathrm{big}}$ is based on a chain-level construction, the above explicit formula follows from Lemma 5.11 in~\cite{RelativeQuantumCohomology}.

The underlying module of the small relative quantum cohomology is
\[
QH^*(X,L) = \widehat{H}^*(X,L;\Lambda).
\]
The small relative quantum product
\[
\mem: QH^*(X,L) \otimes QH^*(X,L) \to QH^*(X,L)
\]
is given by specializing the big relative quantum product to $t_0 = \cdots =t_{N+1} = 0.$ So, it too is associated and graded commutative. Explicitly, $\mem$ is given by
\begin{multline}\label{relqua}
\mem(\Gamma_u, \Gamma_v)=\sum_{\substack{0\le l,m\le N\\  \beta\in H_2(X;\Z)}}T^{\varpi(\beta)}\cdot\gw_\beta(\Delta_u, \Delta_v,\Delta_l)\cdot g^{lm}\cdot \Gamma_m+\\
+\sum_{\beta\in H_2(X,L;\Z)}T^\beta \cdot\ogw_{\beta,0}(\Gamma_u, \Gamma_v)\cdot\Gamma_\diamond.
\end{multline}

\section{The Chiang Lagrangian}\label{section3}

In this section we describe the construction of the Chiang Lagrangian along with an orientation and spin structure. Then, we prove some lemmas concerning its topology.

\subsection{Geometric construction}\label{subsection3.1}
The following is largely based on~\cite{Smith}.
Consider the fundamental representation $V$ of $\sln(2, \mathbb{C}).$ Projectifying, we obtain an action of $\sln(2, \mathbb{C})$ on $\cp^1 = \mathbb{P}(V).$ Taking the $d$-fold symmetric product, we obtain an action of $\sln(2, \mathbb{C})$ on $\sym^d \cp^1 = \cp^d.$ Equivalently, we can consider the $d$-fold symmetric power $S^d V$ and projectify to obtain an action of $\sln(2, \mathbb{C})$ on $\mathbb{P}(S^dV) = \cp^d.$

Fix $d\ge3$, and a configuration $C$ of $d$ distinct points in $\cp^1$. Then, the
$\sln(2, \mathbb{C})$-orbit of $C$ in $\sym^d \cp^1 \cong \mathbb{P}S^dV$
is a three-dimensional complex submanifold, of which the
$\su(2)$-orbit is a three-dimensional totally real submanifold. The stabilizer of $C$ in $\sln(2, \mathbb{C})$ is a finite subgroup of $\su(2)$ which we denote by $\Gamma_C.$
In \cite{faber1993linear}, Aluffi and Faber identify those configurations $C$ for which the
$\sln(2, \mathbb{C})$-orbit has smooth closure $X_C$ in $\mathbb{P}S^dV$.
There are four cases: the vertices of an equilateral triangle on a great circle in $\cp^1$, which we denote by $\tr$, and the vertices of a regular tetrahedron, octahedron
and icosahedron in $\cp^1$, which we denote by $T, O$ and $I$
respectively.

In each case the restriction of the $\su(2)$-action to $X_C$ with the Fubini- Study K\"ahler form is Hamiltonian with the moment map $m: \mathbb{P}S^dV\rightarrow \mathfrak{su}(2)^*$ which is defined by
\[\langle m([z]), \xi \rangle=\frac{i}{2}\frac{z^\dag\varphi(\xi)z}{z^\dag z}, \qquad \xi\in \mathfrak{su}(2), \]
where $\varphi:\mathfrak{su}(2)\rightarrow \mathrm{Mat}_{(d+1)\times(d+1)}(\mathbb{C})$ is given by the representation $S^dV.$ The $\su(2)$-orbits of the configurations $\triangle, T, O$ and $I$ all lie in the zero sets of the respective moment maps, so they are Lagrangian. We denote these orbits by $L_C$.

The Chiang Lagrangian is $L_{\triangle} \subset X_{\triangle} = \mathbb{P}S^3V \cong \cp^3$. It is the space of all the equilateral triangles on great circles in $\cp^1.$ Topologically, $L_{\triangle}$ is the quotient
of $\su(2)$ by the binary dihedral subgroup $\Gamma_\tr$ of order 12.

To work with the Chiang Lagrangian, we will need the following additional notations.
Let $W_\tr$ denote
the Zariski open $\sln(2, \mathbb{C})$-orbit in $\x$, and let
$Y_\tr$ denote its complement. $Y_\tr$ consists of those
$3$-point configurations in $\x$ where at least $2$ of the points
coincide. Let $N_\tr \subset Y_\tr$ be the subvariety consisting
of those configurations where all $3$ points coincide.
If $[z_0:\ldots:z_3]$ are standard coordinates on
$\mathbb{P}S^3 V$, then the roots of the polynomial
\[f(T) := \sum z_j (-T)^j \]
correspond (with multiplicity) to the $3$-tuple of
points obtained by viewing $[z]$ as a point of
$\sym^3 \cp^1$. We count $\infty$ as a root with
multiplicity $3 - \deg f$. Consequently, $Y_\tr$ is
defined by the vanishing of the discriminant
$\triangle(f)$ of $f$.

\subsection{Orientation and spin structure}\label{subsection_orientation_spin}
Let $\alpha,\beta,\gamma$
be a basis of $\mathfrak{su}(2)$ corresponding to infinitesimal right-handed rotations about a right-handed set of orthonormal axes.
The infinitesimal group action gives rise to a
frame for $TL_\triangle,$
\[
\alpha\cdot z,\beta\cdot z,\gamma\cdot z,
\]
where $z\in L_\triangle.$ Let $\oo_\triangle$ be the orientation on $L_\triangle$ such that this frame is positively oriented. Let $\s_\triangle$ be the spin structure on $L_\triangle$ such that this frame can be lifted to the associated double cover of the frame bundle.

\subsection{Topological lemmas}\label{subsection3.2}
Equip $\mathbb{P}S^3V \cong \cp^3$ with the Fubini-Study K\"ahler form $\omega$ such that $\int_{\cp^1} \omega=1$.
The following lemma is given in Section 4.3 of \cite{EvansLekili}.
\begin{lm}\label{lm:h_1}
    $H_1(L_\tr;\Z)=\Z /4\Z.$
\end{lm}

\begin{lm}\label{h_2}
$H_2(L_\triangle;\mathbb{Z})=0$.
\end{lm}
\begin{proof} By Poincar\'e duality, it suffices to show $H^1(L;\mathbb{Z})=0$. By the universal coefficients theorem we get the following short exact sequence
\[0\rightarrow \text{Ext}^1(H_{0}(L_\triangle;\mathbb{Z}),\mathbb{Z})\rightarrow H^1(L_\triangle;\mathbb{Z})\rightarrow\text{Hom}(H_1(L_\triangle;\mathbb{Z}),\mathbb{Z})\rightarrow 0.
\]
By Lemma~\ref{lm:h_1}  $H_1(L_\triangle;\mathbb{Z})=\mathbb{Z}/4\mathbb{Z}$. In addition, $L_\triangle$ is connected, so $H_0(L_\triangle;\mathbb{Z})=\mathbb{Z}$.
Hence, we get an exact sequence
$
0\rightarrow H^1(L_\triangle;\mathbb{Z})\rightarrow 0,
$
which yields $H^1(L_\triangle;\mathbb{Z})=0$.
\end{proof}

The following is Lemma~4.4.1 from \cite{EvansLekili}.
\begin{lm}
\label{maslov_index_lemma}
The Chiang Lagrangian $L_\tr$ has minimal Maslov number equal to $2.$
\end{lm}
For  $\zeta\in H_2(\x;\mathbb{Z})$, we denote by $c_1(\zeta)$ the evaluation of the first Chern class of $\x$ on $\zeta$. Let $\xi = [\cp^1]\in H_2(\x;\mathbb{Z})$, so $c_1(\xi)=4$.
\begin{lm}\label{lm:relhom}
The long exact sequence of the pair $(\cp^3,L_\tr)$ gives the short exact sequence
\[
0\rightarrow H_2(\x;\mathbb{Z})\overset{\varpi}{\rightarrow}H_2(\x,L_\triangle;\mathbb{Z})\rightarrow H_1(L_\triangle;\mathbb{Z})\rightarrow 0,
\]
with $H_2(\x,L_\triangle;\mathbb{Z}) \simeq\mathbb{Z} \simeq H_2(\cp^3;\Z)$, and the map $\varpi$ is given by multiplication by 4.
\end{lm}
\begin{proof}
Consider the long exact sequence
\[\ldots\rightarrow H_2(L_\triangle;\mathbb{Z})\rightarrow H_2(\x;\mathbb{Z}) \overset{\varpi}{\rightarrow} H_2(\x,L_\triangle;\mathbb{Z})\rightarrow H_1(L_\triangle;\mathbb{Z})\rightarrow H_1(\x;\mathbb{Z})\rightarrow\ldots ,\]
where
\[
H_1(\x;\mathbb{Z})=0, \qquad H_2(\x;\mathbb{Z})\simeq \mathbb{Z},
\]
and Lemma~\ref{lm:h_1} gives $H_1(L_\triangle;\mathbb{Z})\simeq\mathbb{Z}/4\mathbb{Z}.$
By Lemma~\ref{h_2} we have $H_2(L_\triangle;\mathbb{Z})=0$, so we get the short exact sequence
\begin{equation}\label{eq:ses4}
0\rightarrow \mathbb{Z} \overset{\varpi}{\rightarrow} H_2(\x,L_\triangle;\mathbb{Z}) \rightarrow \mathbb{Z}/4\mathbb{Z}\rightarrow0.
\end{equation}
We have two possibilities:
\begin{enumerate}
    \item\label{lm:relhom_case_a}
    The short exact sequence~\eqref{eq:ses4} is split, so $H_2(\x,L_\triangle;\mathbb{Z}) \simeq \mathbb{Z}\oplus\mathbb{Z}/4\mathbb{Z}$. Thus,~\eqref{eq:ses4} is isomorphic to
    \[
    0\rightarrow \mathbb{Z} \overset{\varpi}{\rightarrow} \mathbb{Z}\oplus\mathbb{Z}/4\mathbb{Z} \rightarrow \mathbb{Z}/4\mathbb{Z}\rightarrow0,
    \]
    where
    \[
    \varpi(x)= (x,0).
    \]
    \item\label{lm:relhom_case_b}
    The short exact sequence~\eqref{eq:ses4} is not split, so $H_2(\x,L_\triangle;\mathbb{Z})=\mathbb{Z}$. Thus,~\eqref{eq:ses4} is isomorphic to
    \[ 0\rightarrow \mathbb{Z} \overset{\varpi}{\rightarrow} \mathbb{Z} \rightarrow \mathbb{Z}/4\mathbb{Z}\rightarrow0,\]
    where
    \[  \varpi(x)= 4x.\]
\end{enumerate}
In case~\ref{lm:relhom_case_a} the Maslov index is given by
\[
\mu: \mathbb{Z}\oplus\mathbb{Z}/4\mathbb{Z}\rightarrow \mathbb{Z},\]
\[\mu (a,b)= ka, \qquad k\in\mathbb{Z}.
\]
Since $\mu(\varpi(\xi))=2c_1(\xi) = 8$, it follows that $k=8$ in contradiction to Lemma~\ref{maslov_index_lemma}. Therefore, we must be in case~\ref{lm:relhom_case_b}.
\end{proof}

In light of Lemma~\ref{lm:relhom}, let $\tilde{\xi}\in H_2(\x,L_\triangle;\mathbb{Z})$ be the generator such that $\varpi(\xi)=4\tilde{\xi}$.
Recall
$\Gamma_1=[\omega]\in H^2(\x,L_\triangle;\mathbb{Z})$.
\begin{lm}
\label{divisor_for_chiang}
We have
\[
\int_{\tilde{\xi}} \Gamma_1 = \frac{1}{4}.
\]
\end{lm}
\begin{proof}
We have
\[
4\int_{\tilde{\xi}}\Gamma_1= \int_{\varpi(\xi)}\omega = \int_\xi\omega=1.
\]
\end{proof}

\section{Bounding cochains and open Gromov-Witten invariants}\label{sec4}
In this section we recall the definition of the Fukaya $A_\infty$ algebra of a Lagrangian submanifold following~\cite{Fukaya_cyclic_symmetry,Fukaya_counting_discs,FOOO_lagr_floer_theory,FOOO_toric_I,FOOO_toric_II,FOOO,A_infinity:Jake_Sara}. We recall the notion of bounding cochains from~\cite{FOOO} and the invariants $\oogw_{\beta,k}$ of~\cite{solomon2016point} mentioned in Section~\ref{subsection_ogw}. Additionally, we recall results concerning bounding cochains from \cite{solomon2016point}. We use these results to prove Theorem~\ref{lm}, from which Theorem~\ref{theorem_lm} follows as a special case. These theorems give a class of open Gromov-Witten invariants which coincide with straightforward counts of $J$-holomorphic disks without corrections from bounding chains. Theorem~\ref{theorem_lm} plays an essential role in the computation of the basic invariants of Theorem~\ref{thm:iv}. We also prove Lemmas~\ref{lm_spin} and~\ref{lm_orientation} concerning the affect of change of spin structure and orientation on the invariants $\oogw.$ These lemmas are used to prove parts~\ref{ax_spin} and~\ref{ax_orientation} of Proposition~\ref{prop:open_axioms}.

Throughout this section $(X, \omega)$ is a symplectic manifold of real dimension $2n$ and $L\subset X$ is a connected Lagrangian submanifold with relative spin structure $\mathfrak{s}$. The notion of a relative spin structure appeared in \cite{FOOO}. See also \cite{Wehrheim_Woodward}. In particular, a relative spin structure determines an orientation on $L$. Let $J$ be an $\omega$-tame almost complex structure on $X$.

\subsection{Moduli spaces}\label{ssec:moduli spaces}
Denote by $\mathcal{M}_{k+1,l}(\beta)$ the moduli space of $J$-holomorphic genus zero open stable maps $u:(D, \partial D) \rightarrow (X,L)$ of degree $\beta$ with one boundary component, $k+1$ boundary marked points and $l$ interior marked points. The boundary points are labeled according to their cyclic order. We denote elements in $\mathcal{M}_{k+1,l}(\beta)$ by $[u;z_0,\ldots,z_k,w_1,\ldots,w_l]$, where $z_i\in \partial D$ and $w_i\in \Int D$.
Denote by
\[evb^\beta_j: \mathcal{M}_{k+1,l}(\beta) \rightarrow L,\quad j=0,\ldots,k, \]
\[evi^\beta_j: \mathcal{M}_{k+1,l}(\beta) \rightarrow X,\quad j=1,\ldots,l,\]
the evaluation maps associated to boundary marked points and to the interior marked points respectively.

Denote by $\mathcal{M}^S_{k,l}(\beta)$ the moduli space of genus zero $J$-holomorphic open stable maps $u:(D, \partial D) \rightarrow (X,L)$ of degree $\beta$ with one boundary component, $k$ unordered boundary points, and $l$ interior marked points. It comes with evaluation maps as in the case of $\mathcal{M}_{k,l}(\beta)$.

We assume that all $J$-holomorphic genus zero open stable maps with one boundary component are regular and the moduli spaces $\mathcal{M}_{k+1,l}(\beta),\mathcal{M}^S_{k,l}(\beta)$ are smooth orbifolds with corners. In addition, we assume that the evaluation maps $evb^\beta_0$ are proper submersions. These assumptions should hold for $(X,L) = (\cp^3,L_\tr)$ and $J$ the standard complex structure on $\cp^3$ by Remark~1.6 of~\cite{A_infinity:Jake_Sara} in light of the transitive action of $\sln(4,\C)$ on $\C P^3$ and the transitive action of the subgroup $\su(2) \subset \sln(4,\C)$ on $L_\tr$. The arguments of the present section extend to arbitrary targets $(X,\omega,L)$ and arbitrary $\omega$-tame almost complex structures $J$ given the virtual fundamental class techniques of \cite{Fukaya_cyclic_symmetry, Fukaya_counting_discs, FOOO, FOOO_toric_I, FOOO_toric_II, FOOO_lagr_floer_theory,FOOO_Kuranish_II, FOOO_Antisymplectic, FOOO_Spectral,FOOO_Kuranishi_book,Fukaya_Ono}. Alternatively, it should be possible to use the polyfold theory of~\cite{Hofer_Wysocki_Zehnder_29,Hofer_Wysocki_Zehnder_30,Hofer_Wysocki_Zehnder_31,Hofer_Wysocki_Zehnder_32,Wehrheim_Li}. The relative spin structure $\mathfrak{s}$ determines an orientation on $\mathcal{M}_{k+1,l}(\beta),\mathcal{M}^S_{k,l}(\beta)$ as in Chapter~8 of~\cite{FOOO}.

\subsection{The Fukaya \texorpdfstring{$A_\infty$}{A-infinity} algebra}\label{subsection4.2}
Denote by $A^*(X;\R)$ the ring of differential forms on $X$ with coefficients in $\mathbb{R}$. For $m>0$, denote by $\widetilde{A}^m(X,L;\R)$ differential $m$-forms on $X$ that vanish on $L$, and denote by $\widetilde{A}^0(X,L;\R)$ functions on $X$ that are constant on $L$.
Define
$$\widetilde{H}^*(X,L;\mathbb{R})\coloneqq H^*(\widetilde{A}^*(X,L;\mathbb{R}), d).$$
By Lemma 5.14 from\cite{solomon2016point}, if $H^*(L;\R)\simeq H^*(S^n;\R)$ then $\hH^*(X,L;\mathbb{R})\simeq \widetilde{H}^*(X,L;\mathbb{R})$.
Denote by $\beta_0$ the zero element of $H_2(X,L;\Z)$.

Recall the definition of the Novikov ring $\Lambda$ from Section~\ref{subsection_novikovrings}. Let $s,t_0,\ldots,t_M$, be formal variables with degrees in $\mathbb{Z}$. Set
\[R:= \Lambda[[s,t_0,\ldots,t_M]],\quad Q:=\mathbb{R}[[t_0,\ldots,t_M]] \subset R,\]
thought of as differential graded algebras with trivial differential.
Define a valuation
\[\nu: R\rightarrow \mathbb{R}_{\ge0},\]
by
\[
\nu\left(\sum_{j=0}^\infty a_jT^{\beta_j}s^{k_j}\prod_{a=0}^Mt_a^{l_{aj}}\right)
= \inf_{\substack{j\\a_j\ne 0}} \left(\omega(\beta_j)+k_j+\sum_{a=0}^M l_{a_j}\right).
\]
Whenever a tensor product (resp. direct sum) of modules with valuation is written, we mean the completed tensor product  (resp. direct sum).
Set
\[
C:= A^*(L)\otimes R, \quad D:= \widetilde{A}^*(X,L)\otimes Q.
\]
In particular, the gradings on $C$ and $D$ take into account the degrees of $T^\beta, s,t_j$, and the degrees of the differential forms.
The valuation $\nu$ induces valuations on $\Lambda,Q,C,D$ and their tensor products, which we also denote by $\nu$.
Let
\[\Lambda^+=\{\alpha\in\Lambda | \ \nu(\alpha)>0  \}.\]
Let
\[
R^+:= R\Lambda^+ \triangleleft R,\quad \mathcal{I}_R:= \{ \alpha\in R| \ \nu(\alpha)>0\}\triangleleft R, \quad \mathcal{I}_Q:= \{\alpha\in Q| \ \nu(\alpha)>0 \}\triangleleft Q.
\]

Let $\gamma\in \mathcal{I}_QD$ be a closed form with $|\gamma|=2$. For example, given closed differential forms $\gamma_j\in \widetilde{A}^*(X,L;\R)$ for $j=0,\ldots,M,$ take $t_j$ of degree $2-|\gamma_j|$ and $\gamma:=\sum_{j=0}^Mt_j\gamma_j$.
Define maps
\[
\mathfrak{m}_k^{\gamma,\beta}:C^{\otimes k} \rightarrow C
\]
by
\[ \mathfrak{m}^{\gamma,\beta_0}_1(\alpha)=d\alpha,\]
and for $k\ge 0$ when $(k,\beta)\ne(1,\beta_0)$, by
\begin{equation}\label{def:m_k}
  \mathfrak{m}^{\gamma,\beta}_k(\alpha_1,\ldots,\alpha_k):=(-1)^{\sum_{j=1}^kj(|\alpha_j|+1)+1}
\sum_{\substack{l\ge0}}\frac{1}{l!}{evb_0^\beta}_* (\bigwedge_{j=1}^l(evi_j^\beta)^*\gamma\wedge
\bigwedge_{j=1}^k (evb_j^\beta)^*\alpha_j
).
\end{equation}
The push-forward $evb^\beta_{0_*}$ is given by integration over the fiber, which is well-defined because $evb^\beta_0$ is a proper submersion.
Define also \[\mathfrak{m}_k^{\gamma}:C^{\otimes k} \rightarrow C\]
by
\[\mathfrak{m}^{\gamma}_k:= \sum_{\substack{\beta\in H_2(X,L;\Z)}}T^\beta\mathfrak{m}_k^{\gamma,\beta}.  \]
Furthermore, define
$$\mathfrak{m}^\gamma_{-1} \coloneqq \sum_{\substack{\beta\in H_2(X,L;\Z), \\ l\ge 0}} \frac{1}{l!} T^\beta \int_{\mathcal{M}_{0,l}(\beta)} \bigwedge_{j=1}^l (evi^\beta_j)^*\gamma.$$
The following is Proposition 2.6 from \cite{A_infinity:Jake_Sara}.
\begin{prop}
 The operations $\{\mathfrak{m}_k^\gamma\}_{k\ge0}$ define an $A_\infty$ structure on $C$. That is, for any $\alpha_1,\ldots,\alpha_{k+1}\in C$ and $k\in\Z_{\ge0},$
 \[
 \sum_{\substack{k_1+k_2=k+1\\ 1\le i \le k_1}}(-1)^{\sum_{j=1}^{i-1}(|\alpha_j|+1)}\mathfrak{m}_{k_1}(\alpha_1,\ldots,\alpha_{i-1},\mathfrak{m}_{k_2}(\alpha_i,\ldots,\alpha_{i+k_2-1}),\alpha_{i+k_2},\ldots,\alpha_k)=0.
 \]
\end{prop}
The following is Proposition 3.1 from~\cite{A_infinity:Jake_Sara}.
\begin{prop}
The operations $\mathfrak{m}_k$ are $R$-multilinear. Namely, for any $\alpha_1,\ldots,\alpha_{k+1}\in C$ and $a\in R$,
\begin{equation}\label{eqqq}
\mathfrak{m}^\gamma_k(\alpha_1,\ldots,\alpha_{i-1},a\cdot\alpha_i,\ldots,\alpha_k)=(-1)^{|a|\cdot(i+\Sigma_{j=1}^{i-1}|\alpha_j|)}a\cdot\mathfrak{m}^\gamma_k(\alpha_1,\ldots,\alpha_k)+\delta_{1,k}\cdot da\cdot \alpha_1.
\end{equation}
\end{prop}

\begin{prop}\label{prop:energyzero}
For $k\ge0$ and $\omega(\beta)=0$,
\[
\mathfrak{m}^{\gamma,\beta}_k(\alpha_1,\ldots,\alpha_k)=
\begin{cases}
d\alpha_1, & (\beta,k)=(\beta_0,1),\\
(-1)^{|\alpha_1|}\alpha_1\wedge\alpha_2, & (\beta,k)=(\beta_0,2),\\
-\gamma_1|_L, & (\beta,k)=(\beta_0, 0),\\
0, & \text{otherwise}.
\end{cases}
\]
\end{prop}
\begin{proof}
If $\beta\ne \beta_0$ and $\omega(\beta)=0$, then $\mathcal{M}_{k+1,l}=\emptyset$. If $\beta=\beta_0$, see Proposition 3.8 in \cite{A_infinity:Jake_Sara}.
\end{proof}

Let $\langle \ , \ \rangle$ denote the Poincar\'e pairing
\begin{equation}\label{def:Poinacre_pairing}
    \langle \xi, \eta \rangle \coloneqq (-1)^{|\eta|}\int_L\xi \wedge \eta.
\end{equation}
The following is Proposition 3.1 from \cite{A_infinity:Jake_Sara}.
\begin{prop}
    Let $a\in R$ and $\alpha_1,\ldots,\alpha_{k+1}\in C$. The Poincar\'e pairing satisfies the following
\begin{equation}\label{relation1}
\langle a\cdot\alpha_1,\alpha_2\rangle=a \langle \alpha_1,\alpha_2\rangle, \qquad \langle \alpha_1, a\cdot\alpha_2\rangle=(-1)^{|a|\cdot(1+|\alpha_1|)}a\cdot\langle \alpha_1,\alpha_2\rangle,
\end{equation}

\end{prop}
The following is Proposition 3.3 from \cite{A_infinity:Jake_Sara}.
\begin{prop}\label{prop:linear_a_infty}
    For $\alpha_1,\ldots,\alpha_{k+1}\in C$,
    \begin{equation}\label{relation2}
\langle\mathfrak{m}_k^\gamma(\alpha_1,\ldots,\alpha_k), \alpha_{k+1}\rangle=(-1)^{(|\alpha_{k+1}|+1)\sum_{j+1}^k(|\alpha_j|+1)}\langle\mathfrak{m}^\gamma_k(\alpha_{k+1},\alpha_1,\ldots,\alpha_{k-1}),\alpha_k\rangle.
\end{equation}
\end{prop}

\subsection{Bounding pairs, the superpotential and invariants}
\begin{dfn}\label{dfn_bd_pair}
A \textbf{bounding pair} with respect to $J$ is a pair $(\gamma,b)$ where $\gamma\in \mathcal{I}_Q D$ is closed with $|\gamma|=2$ and $b\in \mathcal{I}_R C$ with
$
|b|=1,
$
such that
\[
\sum_{k\ge 0}\mathfrak{m}_k^\gamma(b^{\otimes k})=c\cdot 1, \qquad c\in \mathcal{I}_R,\;|c|=2.
\]
In this situation, $b$ is called a \textbf{bounding cochain} for $\mathfrak{m}^\gamma$.
\end{dfn}
The definition of bounding pair appeared in \cite{solomon2016point}, and the definition of bounding cochain appeared in \cite{FOOO}.

Let $\gamma\in \mathcal{I}_Q D $, $b\in  \mathcal{I}_R C$. The standard superpotential~\cite{Fukaya_counting_discs} is given by
\[
\widehat{\Omega}(\gamma, b):=\widehat{\Omega}_J(\gamma, b)\coloneqq (-1)^n(\sum_{k\ge 0} \frac{1}{k+1} \langle \mathfrak{m}^\gamma_k(b^{\otimes k}), b\rangle + \mathfrak{m}^\gamma_{-1}).
\]
Intuitively, $\widehat{\Omega}$ counts $J$-holomorphic disks with constraints $\gamma$ in the interior and $b$ on the boundary.
Modification is necessary in order to avoid $J$-holomorphic disks the boundary of which can degenerate to a point, forming a $J$-holomorphic sphere.
We say that a monomial element of $R$ is \textbf{of type $\mathcal{D}$} if it has the form $a\, T^{\beta}s^0t_0^{j_0}\cdots t_N^{j_N}$ with $a\in \R$ and $\beta\in \Ima(\varpi)$.
Following~\cite{solomon2016point}, in the present paper, the superpotential is defined by
\[
\Omega(\gamma,b):=\Omega_J(\gamma,b) : = \widehat{\Omega}_J(\gamma,b)-\text{ all monomials of type }\mathcal{D}\text{ in }\widehat{\Omega}_J.
\]

Definition 3.12 from \cite{solomon2016point} gives a notion of gauge equivalence between a bounding pair $(\gamma,b)$ with respect to $J$ and a bounding pair $(\gamma',b')$ with respect to another almost complex structure~$J'.$ Let $\sim$ denote the resulting equivalence relation.
For a graded module $M$, we denote by $(M)_j$ the degree $j$ part of the module.

Define a map
\[\varrho:\{\text{bounding pairs}\}/\sim \ \longrightarrow (\mathcal{I}_Q\widetilde{H}^*(X,L;Q))_2\oplus (\mathcal{I}_R)_{1-n},\]
by
\[\varrho([\gamma,b]):=\left([\gamma],\int_Lb\right). \]
By Lemma 3.16 in \cite{solomon2016point} $\varrho$ is well defined.
The following is Theorem 1 from\cite{solomon2016point}.
\begin{thm}\label{thm1}
If $(\gamma,b)\sim(\gamma',b')$, then $\Omega_J(\gamma,b)=\Omega_{J'}(\gamma',b')$.
\end{thm}
The following is Theorem 2 from \cite{solomon2016point}.
\begin{thm} \label{thm2}
Assume $H^*(L;\mathbb{R})=H^*(S^n;\mathbb{R})$. Then
$\varrho$ is bijective.
\end{thm}
Fix $\Gamma_0,\ldots,\Gamma_M$ a basis of $\widetilde{H}^*(X,L;\mathbb{R})$, set $|t_j|=2-|\Gamma_j|$, and take
\[ \Gamma:=\sum_{j=0}^Mt_j\Gamma_j\in   (\mathcal{I}_Q\widetilde{H}^*(X,L;\mathbb{R}))_2.\]
By Theorem~\ref{thm2}, choose a bounding pair $(\gamma,b)$ such that
\[ \varrho([\gamma,b])=(\Gamma,s).\]
By Theorem~\ref{thm1}, the superpotential $\Omega=\Omega(\gamma,b)$ is independent of the choice of $(\gamma,b)$.
\begin{dfn}\label{def:ogw}
The \textbf{open Gromov-Witten invariants} of $(X,L)$,
\[ \text{OGW}_{\beta,k}: \widetilde{H}^*(X,L;\mathbb{R})^{\otimes l}\rightarrow \mathbb{R},\]
are defined by setting
\[
\text{OGW}_{\beta,k}(\Gamma_{i_1},\ldots,\Gamma_{i_l}):= \text{ the coefficient of }T^{\beta}\text{ in }
\partial_{t_{i_1}}\cdots\partial_{t_{i_l}}\partial_s^k\Omega|_{s=0,t_j=0}
\]
and extending linearly to general input.
\end{dfn}
The following is Proposition 5.2 from \cite{solomon2016point}.
\begin{prop}\label{cor:ogwbasis}
The invariants $\mathrm{OGW}_{\beta,k}$ are independent of the choice of a basis.
\end{prop}

\subsection{Changes of spin structure and orientation}
The following Lemmas describe the dependence
of the invariants $\oogw_{\beta,k}$ on the spin structure and the orientation of $L.$

\begin{lm}\label{lm_spin}
Changing the spin structure on $L$ by the action of $\alpha\in H^1(L;\Z/2\Z)$ changes $\oogw_{\beta,k}(\cdots)$ by multiplication by $(-1)^{\alpha(\partial\beta)}.$
\end{lm}

\begin{proof}
In the following proof, we add a superscript $\s$ to the $A_\infty$ operations, superpotential and open Gromov-Witten invariants, to indicate the spin structure used to construct them. By Lemma~2.10 of~\cite{jakethesis}, the orientations of $\M_{k+1,l}(\beta)$ induced by $\s$ and $\alpha \cdot \s$ differ by a sign of $(-1)^{\alpha(\partial \beta)}$. So,
\[
\m_k^{\gamma,\beta,\s} = (-1)^{\alpha(\partial \beta)} \m_k^{\gamma,\beta,\alpha\cdot \s}.
\]
Let $\f_R : R \to R$ be the $\R$ algebra automorphism that sends $T^\beta$ to  $(-1)^{\alpha(\partial\beta)}T^\beta$ and acts trivially on the other formal variables. Then, the $R$ module automorphism
\[
\mathfrak{f} = \mathfrak{f}_R \otimes \id_{A^*(L)}:C\rightarrow C
\]
is a strict $A_\infty$ homomorphism between the $A_\infty$ structures $\m_k^\s$ and $\m_k^{\alpha\cdot \s}$ and preserves the pairing $\langle\cdot,\cdot \rangle$. So, if $b$ is a bounding cochain for $\m_k^\s$ with $\int_L b = s$, then $\mathfrak{f}(b)$ is a bounding cochain for $\m_k^{\alpha\cdot\s}$ with $\int_L \mathfrak{f}(b) = s$ and $\mathfrak{f}(\Omega^\s(\gamma,b)) = \Omega^{\alpha\cdot\s}(\gamma,\f(b)).$
By Definition~\ref{def:ogw}, we obtain
\begin{align*}
\oogw^{\alpha\cdot\s}_{\beta,k}(\Gamma_{i_1},\ldots,\Gamma_{i_l}) & = [T^\beta]\partial_{t_{i_1}}\cdots\partial_{t_{i_l}}\partial_s^k \Omega^{\alpha\cdot\s}(\gamma,\mathfrak{f}(b)))|_{s=0,t_j=0} \\
&=[T^{\beta}]\partial_{t_{i_1}}\cdots\partial_{t_{i_l}}\partial_s^k\mathfrak{f}(\Omega^\s(\gamma,b))|_{s=0,t_j=0} \\
&=(-1)^{\alpha(\partial\beta)} [T^\beta]\partial_{t_{i_1}}\cdots\partial_{t_{i_l}}\partial_s^k(\Omega^\s(\gamma,b))|_{s=0,t_j=0} \\
&=(-1)^{\alpha(\partial\beta)}  \oogw^{\s}_{\beta,k}(\Gamma_{i_1},\ldots,\Gamma_{i_l}).
\end{align*}
\end{proof}

\begin{lm}\label{lm_orientation}
Changing the orientation of $L$
changes $\oogw_{\beta,k}(\Gamma_{i_1},\ldots,\Gamma_{i_l})$ by multiplication by $(-1)^{k+1}.$
\end{lm}
\begin{proof}
In the following proof, we add a superscript $\oo$ to the $A_\infty$ operations, superpotential, Poincar\'e pairing, $\int_L$ and open Gromov-Witten invariants, to indicate the orientation used to construct them.
By Lemma 2.9 from \cite{jakethesis}, changing the orientation of $L$ reverses the orientation of $\mathcal{M}_{k+1,l}(\beta)$.
Since the relative orientation is not changed, it follows that
\[\m^{\gamma,\beta,\oo}_k=\m^{\gamma,\beta,-\oo}_k.\]
Let $\f_R : R \to R$ be the $\R$ algebra automorphism that sends $s$ to  $-s$ and acts trivially on the other formal variables.
Then, the $R$ module automorphism
\[
\mathfrak{f} = \mathfrak{f}_R \otimes \id_{A^*(L)}:C\rightarrow C
\]
is a strict $A_\infty$ homomorphism between the $A_\infty$ structures $\m_k^\oo$ and $\m_k^{-\oo}$ that satisfies
\[
\langle\cdot,\cdot \rangle^\oo=-\langle\cdot,\cdot \rangle^{-\oo}.
\]
Similarly, $\int_L^\oo = - \int_L^{-\oo}.$ So, if $b$ is a bounding cochain for $\m_k^\oo$ with $\int_L^\oo b=s$, then $\f(b)$ is a bounding cochain for $\m_k^{-\oo}$ with $\int_L^{-\oo} \f(b)=s$ and $\f(\Omega^\oo(\gamma,b))=-\Omega^{-\oo}(\gamma,\f(b))$. By Definition~\ref{def:ogw}, we obtain
\begin{align*}
\oogw^{-\oo}_{\beta,k}(\Gamma_{i_1},\ldots,\Gamma_{i_l})&:= [T^\beta]\partial_{t_{i_1}}\cdots\partial_{t_{i_l}}\partial_s^k(\Omega^{-\oo}(\gamma,\f(b)))|_{s=0,t_j=0}\\
&=- [T^{\beta}]
\partial_{t_{i_1}}\cdots\partial_{t_{i_l}}\partial_s^k\mathfrak{f}(\Omega^\oo(\gamma,b))|_{s=0,t_j=0}\\
&=(-1)^{k+1}[T^\beta]\partial_{t_{i_1}}\cdots\partial_{t_{i_l}}\partial_s^k(\Omega^{\oo}(\gamma,b))|_{s=0,t_j=0}\\
&=(-1)^{k+1}\oogw^{\oo}_{\beta,k}(\Gamma_{i_1},\ldots,\Gamma_{i_l}).
\end{align*}

\end{proof}

We apply the preceding lemmas to prove vanishing results for the invariants $\ogw$ of certain Lagrangian submanifolds $L \subset X.$ The case $n = 3$ and $k$ even of the following corollary has been proved previously in~\cite[Proposition 1.3]{Chen-Zinger-recursion} by a different method, which goes back to an observation of Mikhalkin~\cite[Remark 2.4]{Wel05b}.
\begin{cor}\label{cor:projvan}
Consider $(X,L) = (\C P^n,\R P^n)$ with $n$ odd. Let $\alpha \in H^1(\R P^n;\Z)$ be the generator.
The invariants $\ogw_{\beta,k}(\cdots)$ vanish when $k + \alpha(\partial \beta)$ is even.
\end{cor}
\begin{proof}
We identify $H_2(\C P^n,\R P^n;\Z)$ with $\Z$ by the isomorphism taking the generator with positive symplectic area to $1 \in \Z.$ It is shown in~\cite[Corollary 1.9]{RelativeQuantumCohomology} that all the invariants $\ogw_{\beta,k}(\cdots)$ are determined by the open WDVV equations, the axioms of $\ogw$, the wall-crossing formula Theorem~\ref{wall_crossing}, the closed Gromov-Witten invariants of $\C P^n,$ and the value of $\ogw_{1,2}.$ If we simultaneously change the spin structure on $\R P^n$ by $\alpha$ and reverse the orientation, then Lemma~\ref{lm_spin} and Lemma~\ref{lm_orientation} imply that $\ogw_{\beta,k}(\cdots)$ changes by multiplication by $(-1)^{k + 1 + \alpha(\partial\beta)}.$ In particular $\ogw_{1,2}$ remains unchanged. But since all the invariants $\ogw_{\beta,k}(\cdots)$ are determined from $\ogw_{1,2}$ in a way that does not depend on the spin structure or the orientation of $\R P^n,$ they must also remain unchanged. So, when $k + 1+\alpha(\partial\beta)$ is odd, $\ogw_{\beta,k}(\cdots)$ must vanish.
\end{proof}
The following was obtained by a different method in~\cite[Lemma 6.8]{kai}.
\begin{cor}\label{cor:quadricvan}
Take $X$ to be the quadric hypersurface in $\C P^{n+1}$ given by $\sum_{i = 0}^n z_i^2 - z_{n+1}^2 = 0$ and $L$ to be the real locus. Then $\ogw_{\beta,k}(\cdots)$ vanishes when $k$ is even.
\end{cor}
\begin{proof}
We identify $H_2(X,L;\Z)$ with $\Z$ by the isomorphism taking the generator with positive symplectic area to $1 \in \Z.$ It is shown in~\cite[Theorem 7]{kai} that all the invariants $\ogw_{\beta,k}(\cdots)$ are determined by the open WDVV equations, the axioms of $\ogw$, the wall-crossing formula Theorem~\ref{wall_crossing}, the closed Gromov-Witten invariants of $\C P^n,$ and the value of $\ogw_{1,3}.$ If we reverse the orientation of $L$, then Lemma~\ref{lm_orientation} asserts that $\ogw_{\beta,k}(\cdots)$ changes by multiplication by $(-1)^{k + 1}.$ In particular $\ogw_{1,3}$ remains unchanged. But since all the invariants $\ogw_{\beta,k}(\cdots)$ are determined from $\ogw_{1,3}$ in a way that does not depend on the orientation of $L,$ they must also remain unchanged. So, when $k + 1$ is odd, $\ogw_{\beta,k}(\cdots)$ must vanish.
\end{proof}

\begin{rem}\label{rem:quadricnohparWDVV}
Let $(X,L)$ be as in Corollary~\ref{cor:quadricvan}. It is shown in~\cite[Prop. 6.9]{kai} that the vanishing result of Corollary~\ref{cor:quadricvan} implies that changing the sign of the invariant $\ogw_{1,3},$ from which all the other invariants $\ogw$ are determined by a recursion, does not affect the absolute value of other invariants. Consequently, hypothesis~\ref{hp:arWDVV} of Section~\ref{ssec:directions} cannot hold for this choice of $(X,L).$ The proof of Corollary~\ref{cor:quadricvan} is based on the tension between the fact that all invariants are determined recursively from a single invariant, and the freedom in the choice of the orientation of $L$, which affects the values of invariants through Lemma~\ref{lm_orientation}. Thus, it seems natural to restrict hypothesis~\ref{hp:arWDVV} to $(X,L)$ where such tension is not present. That is, the number of geometric degrees of freedom in the definition of invariants should not exceed the number of invariants from which all others are determined by recursions that do not depend on these degrees of freedom.
\end{rem}

\subsection{Obstruction theory}
The goal of the present section is the proof of Lemma~\ref{lm:1}, which refines the existence result for bounding cochains given in Proposition~3.4 of~\cite{solomon2016point}. For this purpose, we recall relevant parts of the obstruction theoretic construction of bounding cochains given there. The basic idea goes back to~\cite{FOOO}.

Let
\[T(C):= \bigoplus_{k\ge0} C^{\otimes k},\]
denote the completed tensor algebra.
For $x\in \mathcal{I}_RC$, abbreviate
\[e^x= 1\oplus x \oplus (x\otimes x)\oplus (x\otimes x \otimes x)\otimes \ldots\in T(C) .\]
Moreover, define
\[\mathfrak{m}^\gamma:T(C)\rightarrow C,\]
by
\[\mathfrak{m}^\gamma(\otimes_{k\ge 0}\eta_k)=\sum_{k\ge0}\mathfrak{m}^\gamma_k(\eta_k).\]
Any element $\alpha\in C$, can be written as
\begin{equation}\label{equation}
   \alpha=\sum_{i=1}^\infty\lambda_i\alpha_i, \qquad \alpha_i\in A^*(L),\quad \lambda_i=T^{{\beta}_i}s^{k_i}\prod_{a=0}^M t_a^{l_{a_i}},\quad \lim_{i}\nu(\lambda_i)=\infty.
\end{equation}
Define a filtration $F^E$ on $R$, $C$, by
\[\lambda\in F^EC\Longleftrightarrow \nu(\lambda)> E.\]
\begin{dfn}
A multiplicative submonoid $G\subset R$ is \textbf{sababa} if it can be written as a list  \begin{equation}\label{sababa}
    G=\{\pm\lambda_0=\pm T^{\beta_0}, \pm\lambda_1,\pm\lambda_2,\ldots\}
\end{equation}
such that $i<j\Rightarrow \nu(\lambda_i)\le \nu(\lambda_j).$
\end{dfn}
For $j=1,\ldots,m,$ and elements $\alpha_j=\sum_i\lambda_{ij}\alpha_{ij}\in C$ decomposed as in~\eqref{equation}, denote by $G(\alpha_1,\ldots,\alpha_m)$ the multiplicative monoid generated by $\{\pm T^\beta\,|\,\beta\in \Pi_\infty\}$, $\{t_j\}_{j=0}^N$, and  $\{\lambda_{ij}\}_{i,j}$. The following is Lemma 3.3 from \cite{solomon2016point}.
\begin{lm}
For $\alpha_1,\ldots,\alpha_m\in \mathcal{I}_R$, the monoid $G(\alpha_1,\ldots,\alpha_m)$ is sababa.
\end{lm}
For $\alpha_1,\ldots,\alpha_m\in \mathcal{I}_R$ write the image of $G=G(\alpha_1,\ldots,\alpha_m)$ under $\nu$ as the sequence $\{E_0^G=0, E^G_1, E^G_2,\ldots\}$ with $E_i^G<E^G_{i+1}$. Let $\kappa^G_i\in \Z_{\ge0}$ be the largest index such that $\nu(\lambda_{\kappa_i^G})=E^G_i$. In the future we omit $G$ from the notation and simply write $E_i,\kappa_i$, since $G$ will be fixed in each instance and no confusion should occur.

Given $\alpha \in C$, a differential form with coefficients in $R$, we denote by $(\alpha)_j \in C$ the summand of differential form of degree $j$ in $\alpha$. Let $\Upsilon'$ be an $\R-$vector space, let $\Upsilon''=R, Q,$ or $\Lambda,$ and let $\Upsilon=\Upsilon'\otimes\Upsilon''$.
For $x\in\Upsilon$ and a monomial $\lambda\in\Upsilon'',$ denote by $[\lambda](x)\in\Upsilon'$ the coefficient of $\lambda$ in $x$.

Fix a sababa multiplicative monoid $G=\{\lambda_j\}_{j=0}^\infty\subset R$ ordered as in~\eqref{sababa}.
Let $l\ge0.$ Suppose we have $b_{(l)}\in C$ with $|b_{(l)}|=1$, and
\[\mathfrak{m}^{\gamma}(e^{b_{(l)}})\equiv c_{(l)}\cdot 1\pmod{F^{E_{l}}C},\qquad c_{(l)}\in (\mathcal{I}_R)_2.\]
Define the obstruction cochains $o_j\in A^*(L)$ for $j=\kappa_l+1,\ldots,\kappa_{l+1}$ to be
\[o_j:= [\lambda_j](\mathfrak{m}^\gamma(e^{b_{(l)}})).\]
Lemmas~\ref{lm:3.5}-\ref{lm:3.10} are Lemmas 3.5-3.10 from \cite{solomon2016point}.
\begin{lm}\label{lm:3.5}
$|o_j|=2-|\lambda_j|$.
\end{lm}
\begin{lm}\label{lm:3.6}
$do_j=0$.
\end{lm}
\begin{lm}\label{lm:3.7}
If $|\lambda_j|=2,$ then $o_j=c_j\cdot 1$ for some $c_j\in\mathbb{R}.$ If $|\lambda_j|\ne 2,$ then $o_j\in A^{>0}(L;\R)$.
\end{lm}
\begin{lm}\label{lm:3.8}
If $|\lambda_j| = 2-n$ and $(db_{(l)})_n = 0,$ then $o_j = 0$.
\end{lm}
\begin{lm}\label{lm:3.9}
Suppose for all $j\in\{\kappa_l+1,\ldots,\kappa_{l+1}\}$ such that $|\lambda_j|\ne 2$, there exist $b_j\in A^{1-|\lambda_j|}(L;\R)$ such that $(-1)^{|\lambda_j|}db_j=-o_j.$ Then
\[
b_{(l+1)}:=b_{(l)}+\sum_{\substack{\kappa_l+1\le j\le \kappa_{l+1}\\ |\lambda_j|\ne 2}}\lambda_jb_j
\]
satisfies
\[
\mathfrak{m}^{\gamma}(e^{b_{(l+1)}})\equiv c_{(l+1)}\cdot 1\pmod{F^{E_{l+1}}C},\qquad c_{(l+1)}\in (\mathcal{I}_R)_2.
\]
\end{lm}
\begin{lm}\label{lm:3.10}
Let $\zeta\in \mathcal{I}_R C$. Then
$\mathfrak{m}^{\gamma}(e^\zeta)\equiv 0\pmod{F^{E_0}C}.$
\end{lm}
\begin{lm}\label{lm:1}
Let $b_0\in C$ such that \[|b_0|=1,\qquad (db_0)_n=0, \qquad \mathfrak{m}^\gamma({e^{b_0}})\equiv c\cdot 1 \pmod{R^+C}, \quad c\in (\mathcal{I}_R)_2.\] Then there exists a bounding cochain $b$, such that \begin{enumerate}
    \item $b\equiv b_0 \pmod{R^+}$,
    \item $\int_L b=\int_L b_0$,
    \item $(b)_j=(b_0)_j$ where $j\in \{n-1,n\}$.
\end{enumerate}
\end{lm}
\begin{proof} We follow the proof of Proposition 3.4 from \cite{solomon2016point}. Define $b_{(0)}:=b_0$. By Lemma~\ref{lm:3.10}, the cochain $b_{(0)}$ satisfies
\[ \mathfrak{m}^\gamma(e^{b_{(0)}})\equiv 0=c_{(0)}\cdot 1\pmod{F^{E_0}C},\qquad c_{(0)}=0.\]
Moreover, $|b_{(0)}|=1$, $\int_L b_{(0)}=\int_L b_0$, $(db_{(0)})_n=0$, and $(b_{(0)})_j=(b_0)_j$ where $j\in \{n-1,n\}$. Observe also that $F^{E_0}C \subset R^+C$.
Proceed by induction.  Suppose we have $b_{(l)}\in C$ such that $|b_{(l)}|=1$, $b_{(l)}\equiv b_0 \pmod{R^+C}$, and
\begin{gather*}
(db_{(l)})_n =0,\qquad\int_L b_{(l)}=\int_L b_0,\\
(b_{(l)})_j=(b_0)_j \quad j\in \{n-1,n\},\qquad \mathfrak{m}^\gamma(e^{b_{(l)}})\equiv c_{(l)}\cdot 1\pmod{F^{E_l}C},\quad c_{(l)}\in (\mathcal{I}_R)_2.
\end{gather*}
By Lemma~\ref{lm:3.6} we have $do_j=0$, and by Lemma~\ref{lm:3.5} we have $o_j\in A^{2-|\lambda_j|}(L;\R).$ In order to apply Lemma~\ref{lm:3.9}, we need to find forms $b_j\in A^{1-|\lambda_j|}(L;\R)$ such that $(-1)^{|\lambda_j|}db_j=-o_j$ for all $j\in\{\kappa_l+1,\ldots,\kappa_{l+1}\}$ such that $|\lambda_j|\ne 2.$
Since $F^{E_l}C \subset R^+C,$ we have
\[
\mathfrak{m}^\gamma(e^{b_{(l)}})\equiv c_{(l)}\cdot 1 \pmod{R^+C}.
\]
It follows that $o_j=0$ when $\lambda_j\not\in R^+$, so we choose $b_j=0$. Hence, we may assume $\lambda_j\in R^+$.
If $|\lambda_j|= 2-n,$ since $(db_{(l)})_n =0$,  Lemma~\ref{lm:3.8} gives $o_j = 0,$ so we choose $b_j = 0.$ If $2-n < |\lambda_j|<~2,$ then $0 < |o_j| < n.$ The assumption $H^*(L;\R)\simeq H^*(S^n;\R)$ implies $[o_j] = 0 \in H^*(L;\mathbb{R}),$ so we choose $b_j$ such that $(-1)^{|\lambda_j|}db_j = -o_j.$ For other possible values of $|\lambda_j|,$ degree considerations imply $o_j = 0,$ so we choose $b_j =0.$
By Lemma~\ref{lm:3.9}, $b_{(l+1)}:=b_{(l)}+\sum_{\substack{\kappa_l+1\le j\le \kappa_{l+1}\\ |\lambda_j|\ne 2}}\lambda_jb_j$
satisfies
\[
\mathfrak{m}^\gamma(e^{b_{(l+1)}})\equiv c_{(l+1)}\cdot 1\pmod{F^{E_{l+1}}C},\qquad c_{(l+1)}\in (\mathcal{I}_R)_2.
\]
Since $b_j=0$ when $\lambda_j\not\in R^+$, it follows that $b_{(l+1)}\equiv b_{(l)}\equiv b_0 \pmod{R^+C}$.
Since $b_j = 0$ when $|\lambda_j| = 2-n$, it follows that $(b_{(l+1)})_{n-1}=(b_{(l)})_{n-1}=(b_{(0)})_{n-1}$, and that $(db_{(l+1)})_n = (db_{(l)})_n=0$. In addition, $b_j = 0$ when $|\lambda_j| = 1-n$. Thus, $(b_{(l+1)})_{n}=(b_{(l)})_{n}=(b_{(0)})_{n},$ and $\int_L b_{(l+1)}=\int_L b_{(l)}=\int_L b_0$.
The inductive process gives rise to a convergent sequence $\{b_{(l)}\}_{l=0}^\infty$ where $b_{(l)}$ is bounding modulo $F^{E_l}C$.
Taking the limit as $l$ goes to infinity, we obtain
\[
b:=\lim_l b_{(l)},\quad |b|=1,\quad b\equiv b_0 \pmod{R^+C},  \quad \int_L b = \int_L b_0, \quad  \]
\[(b)_j=(b_0)_j \quad j\in \{n-1,n\},\qquad \mathfrak{m}^\gamma(e^b)= c\cdot 1,\quad c=\lim_lc_{(l)}\in (\mathcal{I}_R)_2.
\]
\end{proof}

\subsection{Straightforward counts}
In the following, we will be using the following notation conventions.
For $k\ge1$ denote by $[k]$ the set
$[k]:= \{1,\ldots,k \}.$
For $N>0$, we denote $t\coloneqq (t_1,\ldots,t_N)$. Similarly, we denote by $\alpha$ an $N$-tuple where all the components are taken from the set $\mathbb{Z}_{\ge0}$, i.e
\[\alpha:= (\alpha_1,\ldots,\alpha_N),\quad  \alpha_i\in \mathbb{Z}_{\ge0}. \]
We denote $t^\alpha \coloneqq t_1^{\alpha_1}\cdots t_N^{\alpha_N}$. Recall the definition of $P_\beta$ from \eqref{P_beta_definition}.

\begin{thm}\label{lm}
Let $A_{j}\in\widetilde{H}^*(X,L;\mathbb{R})$. Let $\bar{b}\in A^n(L;\R)$ such that $\pd([\bar{b}])=pt$, and  let $a_{j}$ be a representative of $A_{j}$. Let $\sigma_k = (k-1)!$ for $k \in \Z_{>0}$ and let $\sigma_0 = 1.$ Then,
\begin{multline*}
\oogw_{\beta,k}(A_{1},\ldots,A_{l})= (-1)^n \sigma_k\int_{\mathcal{M}_{k,l}(\beta)} \bigwedge_{j=1}^{l} evi^{\beta*}_j  a_{j} \bigwedge_{j=1}^{k}evb^{\beta*}_j\bar{b} = \\
=(-1)^n\int_{\mathcal{M}^S_{k,l}(\beta)} \bigwedge_{j=1}^{l} evi^{\beta*}_j  a_{j} \bigwedge_{j=1}^{k}evb^{\beta*}_j\bar{b}
\end{multline*}
if for every
\[
\Tilde{\beta}\in P_\beta, \qquad 0\le j\le k,\qquad I\subset[l],
\]
one of the following two conditions is satisfied:
\begin{enumerate}
\item \label{condition_a}
$1-\mu(\Tilde{\beta}) +j(n-1) + \sum_{i\in I} (|A_{i}|-2)<0,$ or
\item \label{condition_b}
$1-\mu(\Tilde{\beta}) +j(n-1) + \sum_{i\in I} (|A_{i}|-2)\ge n-1.$
\end{enumerate}
\end{thm}

\begin{proof}
Let $\Gamma_0,\ldots,\Gamma_M$ be a basis of $\widetilde{H}^*(X,L;\mathbb{R})$ with $\Gamma_0 = 1$. From multilinearity of $\oogw$ and Proposition~\ref{cor:ogwbasis}, it suffices to prove the lemma for $\mathrm{OGW}_{\beta,k}(\Gamma_0^{\otimes{r_0}}\otimes\cdots\otimes\Gamma_M^{\otimes{r_M}})$.
By the unit axiom \ref{ax_unit}, we can assume that $r_0=0$, and by the zero axiom \ref{ax_zero} we can assume that $\beta\ne \beta_0$. We have $\gamma= \sum_{i=1}^M t_i \gamma_i$, where $\gamma_i$ is a representative of $\Gamma_i\in\widetilde{H}^*(X,L;\mathbb{R})$.
Let $b_{0,1,0}\in A^n(L;\R)$ be a representative of the Poincar\'e dual of a point.  Since by Proposition~\ref{prop:energyzero} we have
\[\mathfrak{m}^\gamma(e^{b_{0,1,0}})\equiv db_{0,1,0}-b_{0,1,0}\wedge b_{0,1,0}+\gamma|_L\equiv 0 \pmod{R^+C},\]
it follows by Lemma~\ref{lm:1} that there exists a bounding cochain $b\in C$ such that $b\equiv s\cdot b_{0,1,0} \pmod{R^+C}$, and
\begin{equation}\label{equation_thm_bc}
    (b)_n=b_{0,1,0},\qquad (b)_{n-1}=0.
\end{equation}
Write $b=s\cdot b_{0,1,0}+\sum T^{\tilde{\beta}}s^j t^\alpha b_{\tilde{\beta},j,\alpha}$, where
\[
\tilde{\beta}\in H_2(X,L;\Z),\quad \omega(\tilde{\beta})>0, \quad \alpha=(\alpha_1,\ldots,\alpha_M)\in \mathbb{Z}_{\ge0}^M, \quad j\in \mathbb{Z}_{\ge0},\quad b_{\tilde{\beta},j,\alpha}\in A^*(L;\R).
\]
Since
\[
|b| = 1,\qquad |s|= 1-n, \qquad |T^{\tilde \beta}| = \mu(\tilde\beta), \qquad |t_i|=2-|\gamma_i|,
\]
it follows that
\[
|b_{\Tilde{\beta},j, \alpha}|= 1-\mu(\Tilde{\beta}) -j(1-n) - \sum_{i=1}^M \alpha_i(2-|\gamma_i|).
\]
Hence, if conditions~\ref{condition_a} and \ref{condition_b} are satisfied, then by  \eqref{equation_thm_bc} we must have $b_{\tilde{\beta},j,\alpha}=0$ when
\[
\tilde{\beta}\in P_\beta,\qquad 0 \leq \alpha_i \leq r_i,\qquad 0\le j\le k.
\]
Define
\[
J:=R^+ + ( s^{k+1}, t_0, t_1^{r_1+1},\ldots,t_M^{r_M+1} )\triangleleft R,
\]
and denote
\[b' = s\cdot b_{0,1,0}+ \sum_{\substack{\omega(\tilde{\beta})=\omega(\beta)\\ 0\le j\le k\\ 0\le\alpha_i\le r_i}} T^{\tilde{\beta}}s^j t^\alpha b_{\tilde{\beta},j,\alpha}.\]
Since $b'\equiv b \pmod{JC}$, it follows that
\[[T^\beta](\partial_{t_1}^{r_1}\ldots \partial_{t_M}^{r_M}\partial^k_s\Omega(\gamma,b)|_{s=0,t_j=0})=[T^\beta](\partial_{t_1}^{r_1}\ldots \partial_{t_M}^{r_M}\partial^k_s\Omega(\gamma,b')|_{s=0,t_j=0}).\]
Hence, it suffices to consider $\Omega(\gamma,b')$ instead of $\Omega(\gamma,b)$ in the computation of $\mathrm{OGW}_{\beta,k}(\Gamma_1^{\otimes{r_1}}\otimes\cdots\otimes\Gamma_M^{\otimes{r_M}})$. Since
\[db_{0,1,0}=0, \qquad b_{0,1,0}\wedge b_{0,1,0}=0, \qquad \gamma|_L=0,  \]
by Proposition~\ref{prop:energyzero}, for every $m\ge0,$ $0\le j\le k$ and  $\tilde{\beta}\in H_2(X,L;\Z)$ such that $\omega(\tilde{\beta})=0$, we get
\[\langle \mathfrak{m}^{\gamma,\tilde{\beta}}_m(b_{0,1,0}^{\otimes m}), b_{\beta-\tilde{\beta},j,\alpha}\rangle=0.\]
So, by Proposition~\ref{prop:linear_a_infty} for every $1\le i\le m$, we get
\[\langle\mathfrak{m}^{\gamma,\tilde{\beta}}_m(b_{0,1,0}^{\otimes i-1}\otimes b_{\beta-\tilde{\beta},j,\alpha}\otimes b_{0,1,0}^{\otimes {m-i}}),b_{0,1,0}\rangle=0.\] Hence,
\begin{multline}\label{eq:deromega}
     [T^\beta](\partial_{t_1}^{r_1}\ldots \partial_{t_M}^{r_M}\partial^k_s(\langle\sum_{m\ge0} \mathfrak{m}^{\gamma}_m(b'^{\otimes{m}}), b'\rangle)) = \\ =(\partial_{t_1}^{r_1}\ldots \partial_{t_M}^{r_M}\partial^k_s( \langle \mathfrak{m}^{\gamma,\beta}_{k-1}(s\cdot b_{0,1,0},\ldots,s\cdot b_{0,1,0}),s\cdot b_{0,1,0}\rangle)).
\end{multline}
Denote $l=\sum_{i=1}^Mr_i.$ For $k=0$ we get
 \begin{align*}
     \text{OGW}_{\beta,0}(\Gamma_1^{\otimes{r_1}}\otimes\cdots\otimes\Gamma_M^{\otimes{r_M}})&= [T^\beta](\partial_{t_1}^{r_1}\ldots \partial_{t_M}^{r_M}\partial^k_s\Omega|_{s=0,t_j=0})\\
     &=(-1)^n [T^\beta](\partial_{t_1}^{r_1}\ldots \partial_{t_M}^{r_M}\mathfrak{m}^\gamma_{-1})\\
     &= (-1)^n\int_{\mathcal{M}_{0,l}(\beta)} \bigwedge_{j=1}^{r_1} evi^{\beta*}_j  \gamma_1\cdots\bigwedge_{j=1}^{r_M} evi^{\beta*}_j  \gamma_M,
\end{align*}
where $\mathcal{M}_{0,l}(\beta)=\mathcal{M}^S_{0,l}(\beta).$
For $k\ne0$ we get
\begin{align*}
\oogw_{\beta,k}(\Gamma_1^{\otimes{r_1}}\otimes&\cdots\otimes\Gamma_M^{\otimes{r_M}}) =  \\
& = [T^\beta](\partial_{t_1}^{r_1}\ldots \partial_{t_M}^{r_M}\partial^k_s\Omega|_{s=0,t_j=0})\\
&\!\overset{\eqref{eq:deromega}}=(-1)^n \frac{1}{k}\cdot [T^\beta](\partial_{t_1}^{r_1}\ldots \partial_{t_M}^{r_M}\partial_s^k\langle \mathfrak{m}^\gamma_{k-1} ((s\cdot b_{0,1,0})^{\otimes {k-1}}), s\cdot b_{0,1,0} \rangle)\\
&\!\!\!\!\!\overset{\eqref{eqqq}+\eqref{relation1}}=(-1)^{1+(1-n)\frac{k(k-1)}{2}}\cdot(k-1)!\cdot [T^\beta](\partial_{t_1}^{r_1}\ldots \partial_{t_M}^{r_M}\langle \mathfrak{m}^\gamma_{k-1} ( b_{0,1,0}^{\otimes {k-1}}), b_{0,1,0} \rangle)\\
&\overset{\eqref{def:m_k}}= (k-1)!\cdot \langle evb^{\beta}_{0*}\bigwedge_{j=1}^{l} evi^{\beta*}_j  \gamma_{a_j} \bigwedge_{j=1}^{k-1}evb^{\beta*}_j(b_{0,1,0})), b_{0,1,0} \rangle\\
&\overset{\eqref{def:Poinacre_pairing}}= (-1)^n(k-1)!\int_{\mathcal{M}_{k,l}(\beta)} \bigwedge_{j=1}^{l} evi^{\beta*}_j  \gamma_{a_j} \bigwedge_{j=1}^{k}evb^{\beta*}_j(b_{0,1,0}).
\end{align*}
The diffeomorphism of $\mathcal{M}^S_{k,l}(\beta)$ corresponding to relabeling boundary marked points by a permutation $\sigma \in S_k$ preserves or reverses orientation depending on $sgn(\sigma)$. Let $\sigma\in S_k$ be a permutation. Denote by $\mathcal{M}_{\sigma(k),l}(\beta)$ the moduli space obtained from $\mathcal{M}_{k,l}(\beta)$ by relabeling boundary marked points by $\sigma$. So,
\begin{align*}
    \int_{\mathcal{M}_{k,l}(\beta)} \bigwedge_{j=1}^{l} evi^{\beta*}_j\gamma_{a_j} \bigwedge_{j=1}^{k}evb^{\beta*}_j(b_{0,1,0})&= sgn(\sigma)\int_{\mathcal{M}_{\sigma(k),l}(\beta)} \bigwedge_{j=1}^{l} evi^{\beta*}_j  \gamma_{a_j} \bigwedge_{j=1}^{k}evb^{\beta*}_{\sigma(j)}(b_{0,1,0})\\
    &=sgn(\sigma)^2\int_{\mathcal{M}_{\sigma(k),l}(\beta)} \bigwedge_{j=1}^{l} evi^{\beta*}_j  \gamma_{a_j} \bigwedge_{j=1}^{k}evb^{\beta*}_j(b_{0,1,0})\\
    &=\int_{\mathcal{M}_{\sigma(k),l}(\beta)}\bigwedge_{j=1}^{l} evi^{\beta*}_j  \gamma_{a_j} \bigwedge_{j=1}^{k}evb^{\beta*}_j(b_{0,1,0}).
\end{align*}
Therefore, since $\mathcal{M}_{k,l}^S(\beta) = \coprod_{\sigma \in S_k} \mathcal{M}_{\sigma(k),l}(\beta),$
it follows that
\[
\oogw_{\beta,k}(\Gamma_1^{\otimes{r_1}}\otimes\cdots\otimes\Gamma_M^{\otimes{r_M}})= (-1)^n \int_{\mathcal{M}^S_{k,l}(\beta)} \bigwedge_{j=1}^{l} evi^{\beta*}_j  \gamma_{a_j} \bigwedge_{j=1}^{k}evb^{\beta*}_j(b_{0,1,0}).
\]

\end{proof}
\begin{proof}[Proof of Theorem~\ref{theorem_lm}]
By Lemma~5.14 of~\cite{solomon2016point}, if $H^*(L;\R)\simeq H^*(S^n;\R)$, then $\widehat{H}^*(X,L;\R)\simeq\widetilde{H}^*(X,L;\R)$. Recall that if $k\ne0$ or $\beta\not\in \Ima \varpi$, then $\ogw_{\beta,k}=\oogw_{\beta,k}$.
Thus, when $k\ne0$ or $\beta\not\in \Ima \varpi$ Theorem~\ref{lm} yields Theorem~\ref{theorem_lm}.
\end{proof}
\begin{rem}\label{rem}
In the case of the Chiang Lagrangian, by Lemma~\ref{lm:relhom} $H_2(X,L_\triangle;\mathbb{Z})=\mathbb{Z}$ and $\varpi$ is given by multiplication by $4$. Hence, Theorem~\ref{theorem_lm} holds for $\ogw_{\beta,0}(\cdots)$ when $\beta\not\in 4\mathbb{Z}$.
\end{rem}

\section{Computation of basic invariants}\label{section6}
In order to use the recursive formula to compute the open Gromov Witten invariants
of the Chiang Lagrangian $L_\tr,$ we need the initial values for the recursive
formula, which are given by the following theorems. Fix the orientation $\oo_\triangle$ and the spin structure $\s_\triangle$ on $L_\tr$ as defined in Section~\ref{subsection_orientation_spin}.

\begin{thm}\label{thm:ogw11}
$\ogw_{1, 1} = -3$.
\end{thm}
\begin{thm} \label{thm:ogw102}
$\ogw_{1, 0}(\Gamma_2) =\frac{1}{4}$.
\end{thm}

\begin{thm} \label{thm:ogw203}
$\ogw_{2, 0}(\Gamma_3)  = 1$.
\end{thm}
The proofs appear below.

\subsection{Intersections with the compactification divisor}\label{subsection6.1} Recall from Section~\ref{subsection3.1} the definition of the Chiang Lagrangian $L_\tr \subset X_\tr \simeq \C P^3 \simeq \sym^3\cp^1.$ Recall also the definitions of the discriminant locus $Y_\tr \subset X_\tr$ and the subvariety $N_\tr \subset Y_\tr.$
\begin{lm}\label{divisor}
$\cp^3$ is Fano with anticanonical divisor $\y$.
\end{lm}
\begin{proof}
The proof is part of Section 3.4 in \cite{Smith}.
\end{proof}
The following is Lemma 3.1 from \cite{auroux2007mirror}.
\begin{lm} \label{lm:intersect}
If $L$ is special Lagrangian in the complement
of an anticanonical divisor $Y$ in a compact K\"{a}hler
manifold $X$, then the Maslov index of a disk
$u: (D, \partial D) \rightarrow (X, L)$ is given by twice the
algebraic intersection number $[u] \cdot [Y]$.
\end{lm}
The following is Lemma 3.7 from~\cite{Smith}.
\begin{lm}\label{lemma3.7}
A clean intersection of a holomorphic disk $u$ with $\n$ contributes at least 2 to the intersection number $[u]\cdot[\y]$. A non-clean intersection contributes at least 3.
\end{lm}

\subsection{Axial disks}\label{subsection_axialdisks}
The following is Definition 2.1 from~\cite{Smith}.
\begin{dfn}
If $X$ is a complex manifold carrying an action of a compact Lie group $K$ by holomorphic automorphisms, and $L$ is a totally real submanifold which is an orbit of the $K$-action, then we say $(X,L)$ is $K$-\textbf{homogeneous}.
\end{dfn}
\begin{ex}
The pair $(\cp^3,L_\triangle)$ is $\su(2)$-homogeneous as explained in Section~\ref{subsection3.1}.
\end{ex}

The following is Definition 2.3 from~\cite{Smith}.
\begin{dfn}
Let $(X, L)$ be $K$-homogeneous. If $u: (D, \partial D) \rightarrow (X, L)$ is a holomorphic disk, and there
exists a smooth group homomorphism $R: \R
\rightarrow K$ such that (possibly after reparametrizing
$u$) we have $u(e^{i\theta} z) = R(\theta)u(z)$
for all $z \in D$ and all $\theta \in \mathbb{R}$, then we say $u$ is \textbf{axial}. Let $\xi \in \mathfrak{k} = \operatorname{Lie}(K).$ We say that $u$ is \textbf{axial of type} $\xi$ if $\dot R(0)$ belongs to the orbit of $\xi$ under the adjoint action of $K$.
\end{dfn}

Thinking of $\cp^1$ as $\C \cup \{\infty\}$, we identify $\cp^1 \simeq  S^2$ by stereographic projection from the north pole. Thus, the complex structure on the $2$-sphere is given by left-handed rotation around the outward normal by an angle of $\pi/2.$

For concreteness, we choose $\tr$ to be the triangle with vertices
\[c_1=(0,0,1),\quad c_2=(\frac{\sqrt{3}}{2},0,-\frac{1}{2}),\quad c_3=(-\frac{\sqrt{3}}{2},0,-\frac{1}{2}).\]
By abuse of notation, we also denote by $\tr$ the point $[c_1,c_2,c_3]\in\sym^3\cp^1.$

 \begin{figure}[ht]
     \centering
     \includegraphics[scale=0.28]{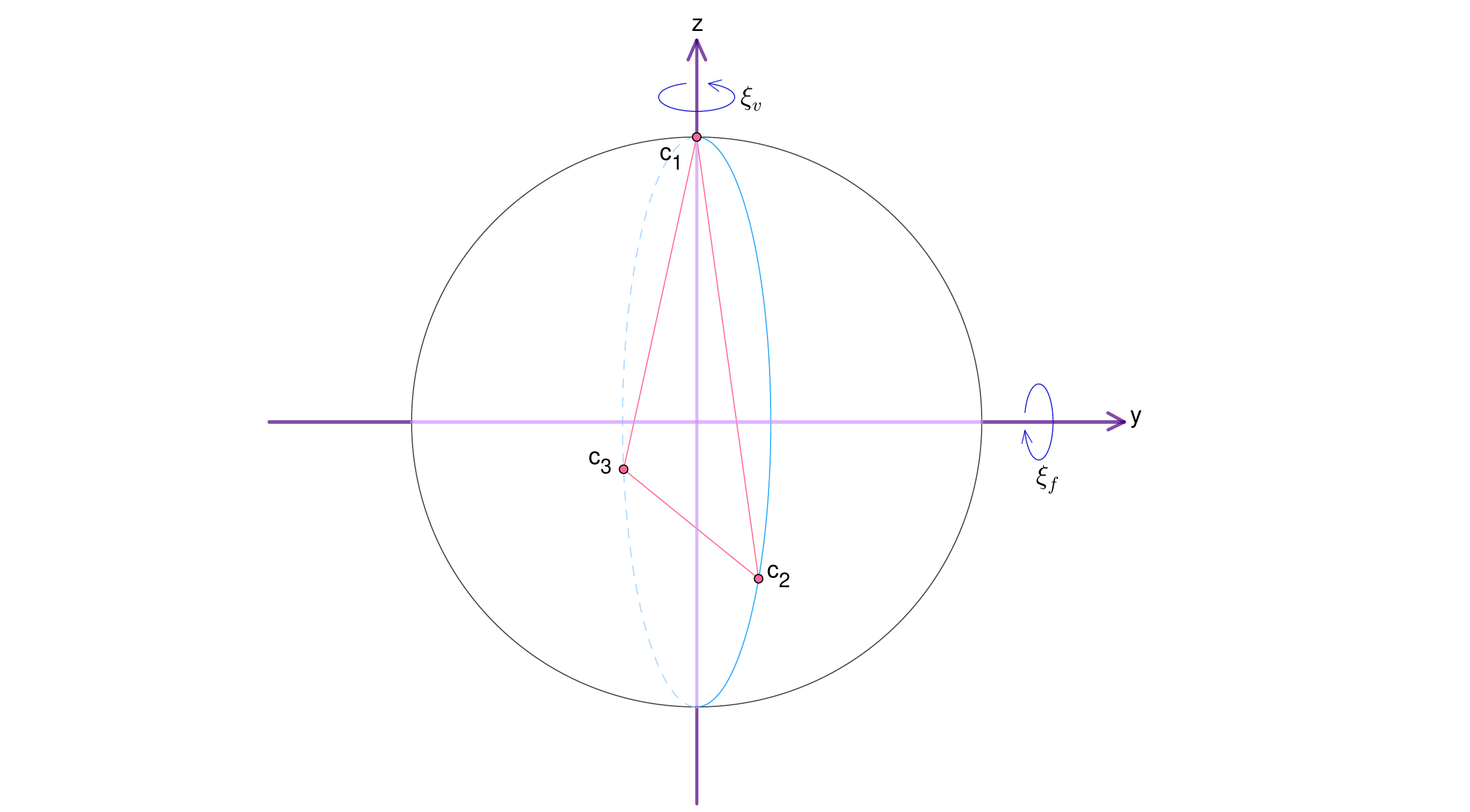}
     \caption{The choice of $\xi_v, \xi_f$ for the configuration $\tr$.}
     \label{fig:1}
 \end{figure}

Let $\xi_v\in\mathfrak{su}(2)$ be the infinitesimal right-handed rotation about the $z$ axis and let $\xi_f\in\mathfrak{su}(2)$ be the infinitesimal right-handed rotation about the $y$ axis scaled so that for $\xi = \xi_v,\xi_f,$ we have
$\{t \in \mathbb{R}: e^{2\pi \xi t}\cdot \tr = \tr\} =\mathbb{Z}$. Thus $\xi_v,\xi_f,$
are generators of rotations about the vertex $c_1$ and the center of the face of the triangle
respectively, as shown in Figure~\ref{fig:1}.

\begin{ex}[Axial Maslov 2 disk]\label{ex}
Consider the homomorphism
\[
R:\R\rightarrow \su(2)
\]
given by
\[
R(\theta)=e^{\theta\xi_v}.
\]
Let $u: (D,\partial D) \to (\cp^3,L_\triangle)$ be given by
$u(z)=e^{-i\xi_v\log z}\cdot \tr$. See below for more explicit formulas. Then $u$ satisfies $u(e^{i\theta}z)=R(\theta)u(z),$ so it is an axial disk of type $\xi_v$.

Recall we denote by $c_1,c_2,c_3$ the vertices of the triangle $\tr$.
The stereographic projection from the north pole
\[
p: S^2\to \C\cup \{\infty\},
\]
is given by
\[p(x,y,z)=
         \left\{\begin{array}{ll}
        \frac{x+iy}{1-z}, &  z\ne1\\
        \infty, & z=1 .
        \end{array} \right.
\]
We have
\[
\overline{c}_1:=p(c_1)=\infty, \quad \overline{c}_2:=p(c_2)=\frac{1}{\sqrt{3}}, \quad
\overline{c}_3:=p(c_3)=\frac{-1}{\sqrt{3}}.
\]
Since $\xi_v$ is the generator of a rotation about the vertex $c_1$, the flow of $\xi_v$ is given by
\[
\varphi^t_{\xi_v}(w):=e^{\xi_v t}\cdot w=e^{\chi i t}w, \qquad w\in \C\cup \{\infty\}.
\]
for some $\chi \in \R.$
Since $\xi_v$ is normalized so that
\[
\Z = \{t\in\R: e^{2\pi\xi_v t}\cdot\tr=\tr \} = \{t \in \R : \varphi^{2\pi t}_{\xi_v}(\bar c_2) = \bar c_2 \text{ or } \bar c_3\} = \{t\in \R : 2\pi\chi t \in \pi \Z\},
\]
we have $\chi = \frac{1}{2}$ and
\[
\varphi^t_{\xi_v}(w) = e^{\frac{1}{2} i t}w, \qquad w\in \C\cup \{\infty\}.
\]
Hence, for $z\in D$ we get
\begin{align*}
u(z)&= e^{-i\xi_v\log z}\cdot \tr\\
&=\big[ \ [1:\varphi_{\xi_v}^{-i\log z}(\overline{c}_1)], \ [1:\varphi_{\xi_v}^{-i\log z}(\overline{c}_2)], \ [1: \varphi_{\xi_v}^{-i\log z}(\overline{c}_3)]\  \big]\\
&=\big[ \ [0:1], \  [1:\sqrt{\frac{z}{3}}] , \  [1:-\sqrt{\frac{z}{3}}] \ \big]\in \sym^3(\cp^1).
\end{align*}
Hence, we can write $u(z)= [0:1:0: -\frac{z}{3}]\in \cp^3.$

We can describe $u$ geometrically as follows. The boundary of $u$ is obtained by the action of $R(\theta)$ on $\tr$. Call an isosceles triangle narrow if the congruent sides are longer than the base. A point in the interior of $u$ is described by a narrow isosceles triangle on a great circle where the north pole is the apex. For example, the triangle $c_1d_2d_3$ in Figure~\ref{fig:2} represents a point in the interior of $u$. Note that as $z\rightarrow 0$ the vertices $c_2$ and $c_3$ move toward the south pole, so $u$ intersects $\y$ in a single point. So, $[u]\cdot [\y]\ge 1$. The proof of Lemma 3.8 in \cite{Smith} shows that equality holds. Hence, by Lemma~\ref{lm:intersect} $u$ is a Maslov 2 disk.
\end{ex}
\begin{figure}[ht]
     \centering
     \includegraphics[scale=0.28]{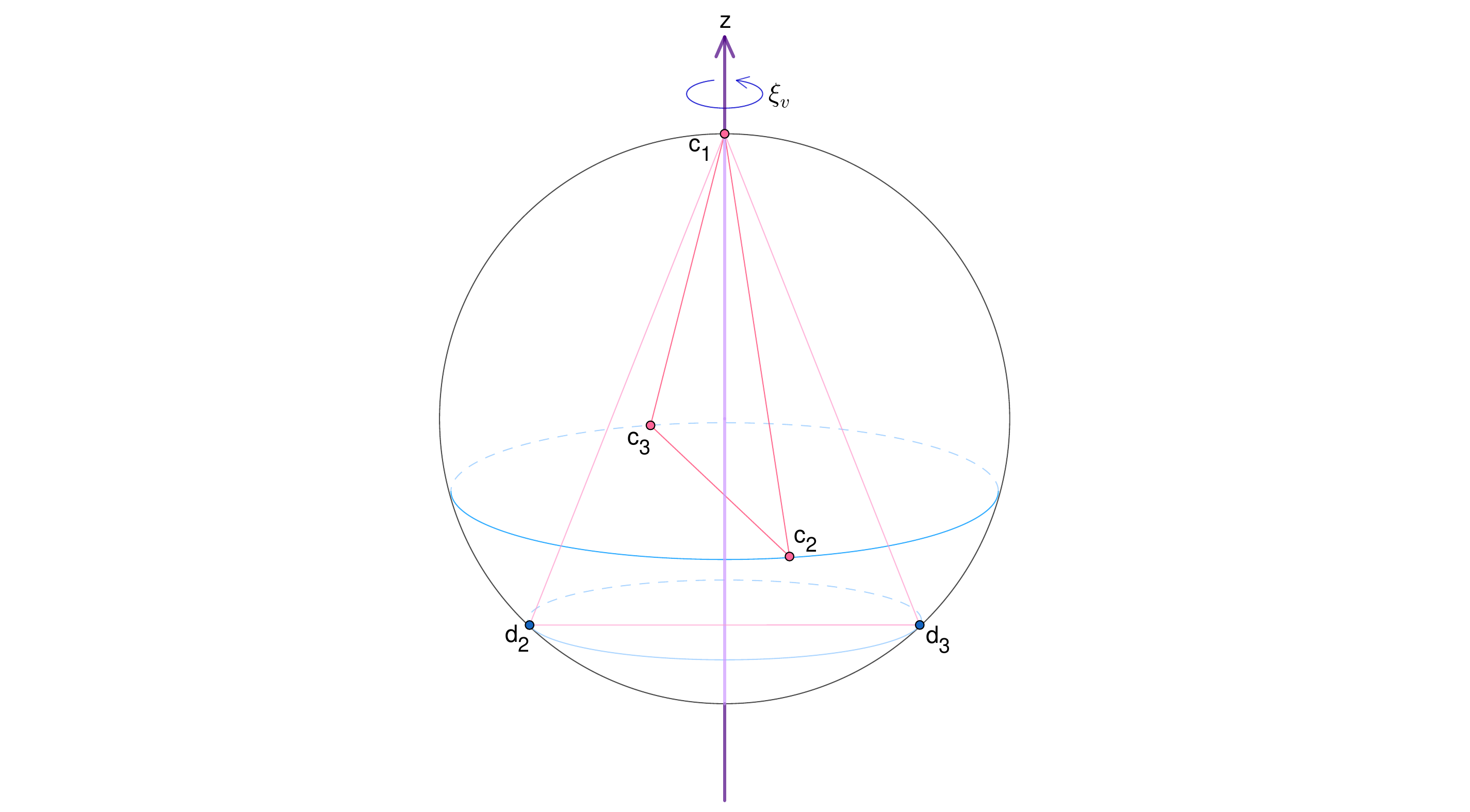}
     \caption{A Maslov 2 disk passing through $\tr$.}
     \label{fig:2}
 \end{figure}
\begin{ex}[Axial Maslov 4 disk]\label{ex2}
Consider the homomorphism
\[
R:\R\rightarrow \su(2),
\]
given by
\[
R(\theta)=e^{\theta\xi_f}.
\]
The disk $u:(D,\partial D) \to (\cp^3,L_\tr)$ given by $u(z)=e^{-i\xi_f\log z}\cdot \tr$ satisfies $u(e^{i\theta}z)=R(\theta)u(z),$ so it is axial of type $\xi_f$.

We can describe $u$ geometrically as follows. The boundary of $u$ is obtained by the action of $R(\theta)$ on $\tr$. The interior of $u$ is given by all the equilateral triangles on northern lines of latitude if we take $r:=(0,-1,0)$ for the north pole. For example, the triangle $s_1s_2s_3$ in Figure~\ref{fig:5} represents an interior point of $u.$ Since $u$ intersects $N_\tr$ at the unique point $[r,r,r]$, it follows by Lemma~\ref{lemma3.7} that $[u]\cdot[\y] \geq 2.$ The proof of Lemma 3.8 in~\cite{Smith} shows that equality holds. Hence, by Lemma~\ref{lm:intersect} $u$ is a Maslov 4 disk.
\end{ex}
\begin{figure}[ht]
     \centering
     \includegraphics[scale=0.28]{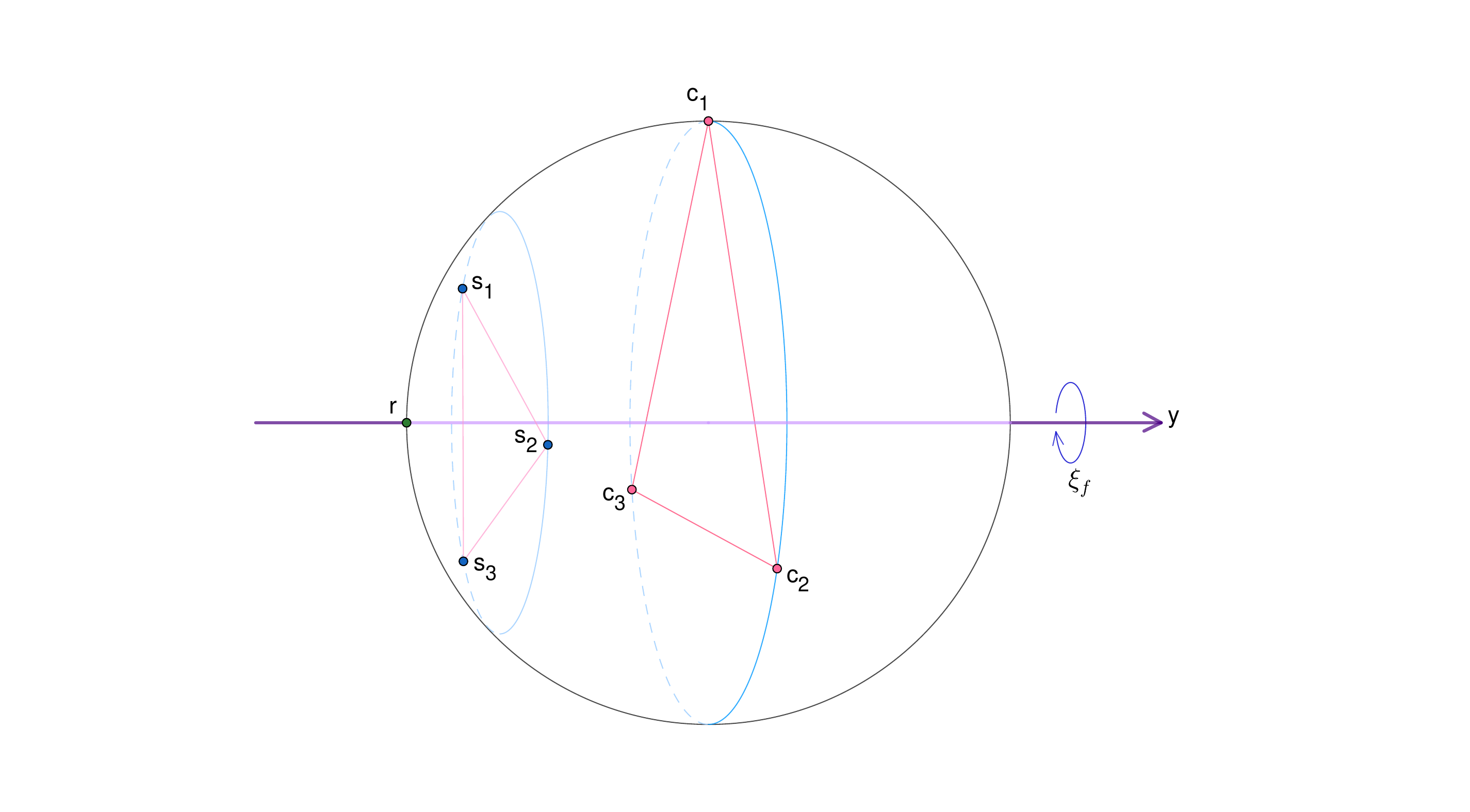}
     \caption{A Maslov 4 disk passing through $\tr$.}
     \label{fig:5}
 \end{figure}

\subsection{Classification of holomorphic disks}\label{subsection5.3}
The following definition is given in Section 4 from~\cite{Smith}.
\begin{dfn}
The intersection points of a holomorphic disk $u:(D,\partial D)\rightarrow(\x, L_\triangle)$ with the compactification divisor $\y$ are called \textbf{poles} of $u$.
\end{dfn}

The following is Definition 4.6 from~\cite{Smith}. Recall from Section~\ref{subsection3.1} the definition of $\Gamma_\tr.$
\begin{dfn}
A \textbf{pole germ} is the germ (at 0) of a holomorphic map $u,$ from an open neighborhood of $0$ in $\C$ to $\cp^3,$ such that $u^{-1}(\y)$ contains $0$ as an isolated point. More generally, for a Riemann surface $\Sigma$ and a point $a\in\Sigma,$ one can speak of a pole germ at $a.$ If we don't specify `at $a$' then we are implicitly working at $0$ in $\C.$ We define an equivalence relation on pole germs at $a$ by $u_1\sim u_2$ if and only if there exists a germ of holomorphic map $A,$ from a neighborhood of $a$ in $\Sigma$ to $\mathrm{SL}(2,\C),$ such that $u_2=A\cdot u_1.$ The \textbf{principal part} of a pole germ $u$ is its equivalence class $[u]_a$ under this relation.

We say a pole germ $u$ is of \textbf{type} $\xi\in\mathfrak{su}(2)$ and \textbf{order} $k\in\mathbb{Z}_{\ge1}$ if its principal part is
\[[z\mapsto e^{-ik\xi\log z}\cdot \tr]_0,\] and $\xi$ is scaled so that $\{t\in\R: e^{2\pi\xi t}\in \Gamma_\tr\}=\mathbb{Z}$.
\end{dfn}
\begin{ex}
An axial disk of type $\xi$ has a pole germ of type $\xi$ and order $1$.
\end{ex}
The following is Lemma 4.10 from~\cite{Smith}.
\begin{lm}\label{lm:lmpole1}
A pole germ $u$ is of type $\xi_v$ if and only if $u(0) \in Y_\triangle\backslash N_\triangle$.
\end{lm}
The following is Corollary 4.13 from~\cite{Smith}.
\begin{lm}
\label{lemmaxi}
All Maslov index $2$ holomorphic disks $u: (D,\partial D)\to (\cp^3,L_\tr)$ are, up to reparametrization, of the form
\[
u(z)=A\cdot e^{-i\xi_v \log z}\cdot \tr
\]
for $A\in \su(2)$. In particular they are all axial.
\end{lm}
The following is Corollary 4.14 from~\cite{Smith}.
\begin{lm}\label{lm:lmpole2}
Suppose $u : (D,\partial D) \to (\cp^3,L_\tr)$ is a holomorphic disk of Maslov index 4, then either $u$ has two poles of
type $\xi_v$ and order $1$, one pole of type $\xi_v$ and order
$2$, or one pole of type $\xi_f$ and order $1$.
In the last case, the disk is axial of type $\xi_f$.
\end{lm}
Recall the notation for moduli spaces of holomorphic disks from Section~\ref{ssec:moduli spaces}. The following is Corollary 4.21 from~\cite{Smith}.
\begin{lm}\label{lm:lmpole3}
Let $u : (D,\partial D) \to (\cp^3,L_\tr)$ be an axial disk of type $\xi_f$. Then for any $w \in \Int D$, the evaluation map
\[
evi_1: \M_{0,1}(2) \to \cp^3,
\]
is a
submersion at $[u;w]$.
\end{lm}
\begin{lm}\label{pointslemma}
Let $a,b\in\cp^1$ be two sufficiently close points. There exists a unique point $c\in\cp^1$ such that there exists a holomorphic disk $u:(D, \partial D)\rightarrow (\x,L_\triangle)$ of Maslov index 2 passing through $[a, b, c]\in\sym^3\cp^1 = \cp^3.$ This disk is unique up to reparameterization.
\end{lm}
\begin{proof}
Following Example~\ref{ex}, we call an isosceles triangle narrow if the congruent sides are longer than the base. Since $a,b,$ are sufficiently close, they cannot belong to an equilateral triangle on a great circle, so a point of the form $[a,b,c]$ must be in $\cp^3\setminus L_\tr.$ Thus, if a $J$-holomorphic disk $u : (D,\partial D) \to (\cp^3,L_\tr)$ passes through a point of the form $[a,b,c],$ then $[a,b,c] \in u(\Int D).$
By Lemma~\ref{lemmaxi} all $J$-holomorphic disks $u: (D,\partial D) \to (\cp^3,L_\tr)$ of Maslov index $2$ have the form $u(z)=A\cdot e^{-i\xi_v\log z}\cdot \tr.$ For such $u,$ it follows from the geometric explanation given in Example~\ref{ex} that a point of $u(\Int{D} )$ is given by the vertices of a narrow isosceles triangle with apex $A \cdot (0,0,1).$ So, since $a,b,$ are sufficiently close, if $u$ passes through $[a,b,c]$, then $a,b,c,$ must be the vertices of a narrow isosceles triangle on a great circle with apex $c = A \cdot (0,1,1).$ We claim that such $c$ is unique.
Indeed, the points $a$ and $b$ are sufficiently close, so they determine a unique great circle. In addition, there exists a unique point $c$ on the great circle such that the triangle with vertices $a,b,c,$ is narrow isosceles. The condition $c = A \cdot (0,0,1)$ determines $A$ up to the action of $S^1 = \Stab_{(0,0,1)} \subset \su(2).$ This $S^1$ acts on $u$ by reparameterization.
\end{proof}
\begin{lm}\label{4maslovthroughn}
Let $p\in N_\triangle.$ There exists an axial disk of type $\xi_f$ that passes through $p.$ This disk is unique up to reparameterization.
\end{lm}
\begin{proof}
Let $q\in\cp^1$ such that $p:=[q,q,q].$ Denote $r:=(0,-1,0)$. By definition, the general axial disk of type $\xi_f$ has the form $u(z)=A\cdot e^{-i\xi_f\log z}\cdot \tr.$ It follows from Example~\ref{ex2} that such a disk intersects $N_\tr$ at the unique point $[A\cdot r,A\cdot r,A\cdot r].$ Let $A\in \su(2)$ be a rotation such that $A\cdot r=q.$ Then $u(0) = p.$ On the other hand, the condition $A\cdot r = q$ determines $A$ up to the action of $S^1 = \Stab_{q} \subset \su(2).$ This $S^1$ acts on $u$ by reparameterization.
\end{proof}

\subsection{Riemann-Hilbert pairs}\label{choose_spin_structure}
\begin{dfn}
A \textbf{Riemann-Hilbert pair} consists of holomorphic rank $n$ vector bundle $E\rightarrow D$ over the closed unit disk with a smooth totally real $n$-dimensional subbundle $F\subset E|_{\partial D}$. A \textbf{spin} Riemann-Hilbert pair is a Riemann-Hilbert pair $(E,F)$ along with an orientation and spin structure on $F.$
\end{dfn}
\begin{dfn} Let $(E,F),$ $(E',F')$ be spin Riemann-Hilbert pairs.
An \textbf{isomorphism} of Riemann-Hilbert pairs $(E,F)\rightarrow (E',F')$ consists of
\begin{enumerate}
    \item a biholomorphism $f:D\rightarrow D$, and
    \item a holomorphic isomorphism of bundles $\phi:E\rightarrow E'$ covering $f$ such that $\phi|_{\partial D}$ takes $F$ to $F'$.
\end{enumerate}
If $(E,F)$ and $(E',F')$ are spin Riemann-Hilbert pairs, we say that $(f,\phi)$ is an isomorphism if additionally $\phi|_F : F \to F'$ preserves orientation and spin structure.
We may refer to the pair $(f,\phi)$ by $\phi$ alone.
\end{dfn}

Let $(E,F)$ be a Riemann-Hilbert pair. Since $D$ is contractible, we identify $E$ with the trivial bundle $\underline{\C}^n$ and thus the pair is determined by the subbundle $F$.
A totally real subspace of $\mathbb{C}^n$ is described by $A\cdot\mathbb{R}^n$ where $A\in GL(n,\mathbb{C})$. Moreover, $A' \in GL(n,\mathbb{C})$ gives the same totally real subspace if and only if $A^{-1}A'\in GL(n,\mathbb{R})$. The fibers of a totally real subbundle $F \subset \underline \C^n$ can be written as $F_z=A(z)\cdot\mathbb{R}^n$ where $A(z)\in GL(n,\mathbb{C})$, so such a family of matrices $\{A(z)\}_{z\in \partial D}$ determines a Riemann-Hilbert pair.

Let $F^\kappa \subset \underline{\C}$ denote the totally real subbundle given by $F^\kappa_z= z^{\kappa/2}\R$ and let $F^{\kappa_1,\kappa_2}\subset \underline{\C}^2$ denote the totally real subbundle given by $F^{\kappa_1,\kappa_2}_z= z^{\kappa_1/2}\R\oplus z^{\kappa_2/2}\R.$
The following is Theorem~1 from \cite{oh1995rhproblem}.
\begin{thm}\label{theorem_oh}
Any Riemann-Hilbert pair $(E,F)$ splits as a direct sum of Riemann-Hilbert pairs
\[(E,F)\simeq \oplus_{i=1}^n(\underline{\mathbb{C}},F^{\kappa_i}),\]
where $\kappa_i\in \mathbb{Z}$.
\end{thm}
If the boundary real subbundle $F$ is orientable, which is equivalent to having even total index, by Theorem~\ref{theorem_oh} we can write
\[(E,F)\simeq \oplus_{i=1}^m(\underline{\mathbb{C}},F^{\kappa_i})\oplus_{i=1}^l(\underline{\mathbb{C}}^2,F^{{\kappa_1}_i,{\kappa_2}_i}),\] where $\kappa_i$ and  ${\kappa_1}_i+{\kappa_2}_i$ are even numbers. Hence, each summand is orientable.

For a Riemann-Hilbert pair $(E,F),$ let
\[\Bar{\partial}_{(E,F)}:\Gamma((D,\partial D),(E,F))\rightarrow \Gamma(D, \Omega^{0,1}(E)),
\]
denote the Cauchy-Riemann operator determined by the holomorphic structure on $E$. Let $\det\Bar{\partial}_{(E,F)}$ denote the Fredholm determinant. Abbreviate
\[\bar{\partial}_{\kappa}:=\bar{\partial}_{(\underline{\mathbb{C}},F^{\kappa})}, \quad \bar{\partial}_{\kappa_1,\kappa_2}:=\bar{\partial}_{(\underline{\mathbb{C}}^2,F^{\kappa_1,\kappa_2})}.\]

The following is a special case of Proposition 2.8 from \cite{jakethesis}.

\begin{prop}\label{pr:canor}
Let $(E,F)$ be a spin Riemann-Hilbert pair. Then
$\det\bar{\partial}_{(E,F)}$ admits a canonical orientation. If $\phi:(E,F)\rightarrow (E',F')$ is an isomorphism of spin Riemann-Hilbert pairs, then the induced isomorphism
\[\phi:\det\bar{\partial}_{(E,F)}\rightarrow \det\bar{\partial}'_{(E',F')}\] preserves the canonical orientation. Furthermore, the canonical orientation varies continuously in a family of Cauchy-Riemann operators.
\end{prop}
The following is a consequence of the proof of a special case of Lemma 8.1 from \cite{jakethesis}.
\begin{lm}\label{lm8.1jake}
Suppose
\[
0\rightarrow V' \rightarrow V\rightarrow V''\rightarrow 0
\]
is a short exact sequence of real vector bundles over a base $B.$ Then an orientation on any two of $V, V' , V''$ naturally induces an orientation on the third. A spin structure on any two of $V, V' , V''$ naturally induces a spin structure on the third.

Suppose now that $V',V'',$ are oriented and spin and equip $V = V' \oplus V''$ with the induced orientation and spin structure. If $\xi'_1,\ldots,\xi'_{k'}$ (resp. $\xi''_1,\ldots,\xi''_{k''})$ is an oriented frame of $V'$ (resp. $V''$) that lifts to the associated spin bundle, then $\xi'_1 \oplus 0,\ldots,\xi'_{k'} \oplus 0, 0 \oplus \xi''_1,\ldots,0\oplus \xi''_{k''}$ is an oriented frame of $V' \oplus V''$ that lifts to the associated spin bundle.
\end{lm}
In light of Lemma~\ref{lm8.1jake}, the direct sum of spin Riemann-Hilbert pairs is naturally a spin Riemann-Hilbert pair.
The following can be deduced from the proof of Proposition 8.4 from \cite{jakethesis}.
\begin{lm}\label{orientation_on_ker}
Let $(E,F)$ be a spin Riemann-Hilbert pair. If $(E,F)$ splits as a direct sum of spin Riemann-Hilbert pairs $(E',F')\oplus (E'',F'')$, then the isomorphism
\[\det\Bar{\partial}_{(E,F)}\simeq\det\Bar{\partial}_{(E',F')}\otimes \det\Bar{\partial}_{(E'',F'')} \] is orientation-preserving.
\end{lm}
\begin{rem}\label{rem_spinstructure}
An oriented real line bundle has a canonical spin structure. So, a one-dimensional Riemann-Hilbert pair with an orientation on the real bundle determines a spin Riemann-Hilbert pair. We will use this implicitly below.
\end{rem}

The proof of the following lemma is given as a part of the proof of Theorem C.4.1 in \cite{mcduffsalamon2012}.
\begin{lm}\label{lemmakerbasis}
If $\kappa\ge -1$, then $\ker\bar{\partial}_{\kappa}$ can be written explicitly as
\[\ker\bar{\partial}_{\kappa}=\{\sum_{j=0}^{\kappa}a_jz^j| \ a_j\in\mathbb{C} \text{ and } a_j=\bar{a}_{\kappa-j}\} \]
and $\coker(\bar{\partial}_{\kappa}) = 0.$ In particular, $\dim\ker\bar{\partial}_{\kappa} = \kappa +1.$
\end{lm}
Below, we often use the fact that if $(E,F)$ is a Riemann-Hilbert pair with $\coker \bar{\partial}_{(E,F)} = 0,$ then an orientation of $\det \bar{\partial}_{(E,F)}$ is the same as an orientation of $\ker \bar{\partial}_{(E,F)}.$
\begin{lm}\label{remarkker0}
Evaluation at 1 defines an orientation-preserving isomorphism
\[
f_0:\ker\bar{\partial}_0\rightarrow F^0_1.
\]
\end{lm}
\begin{proof}
This follows from the definition of the orientation on $\det(\partial_0)$ given in Section 2 of~\cite{jakethesis}.
\end{proof}

The following is Lemma A.2.1 from \cite{smith2020monotone}.
\begin{lm}\label{lemmaA21}
Let $(E,F)$ be a rank 1 Riemann-Hilbert pair of Maslov index $2$. Evaluation at $0$ and $1$ defines an isomorphism
\[
f_2:\ker\bar{\partial}_{(E,F)}\rightarrow E_0\oplus F_1.
\]
Moreover, for any choice of orientation on $F,$ this isomorphism is orientation-preserving if the codomain is oriented by the complex structure on $E_0$ and the orientation on $F_1$.
\end{lm}
We choose an orientation on $F^2$ such that the frame $z$ is positively oriented. By Remark~\ref{rem_spinstructure} this choice of orientation determines a spin structure on $F^2.$

\begin{lm}\label{lemmaind2}
The orientation on $\ker\bar{\partial}_2$ is determined by the ordered basis
\[
(z^2+1), \quad i(1-z^2), \quad z.
\]
\end{lm}
\begin{proof}
By Lemma~\ref{lemmakerbasis} we have
\[\ker\bar{\partial}_2=span_{\R}\{(z^2+1),  i(1-z^2), z\}.\]
Consider the map \[f_2:\ker\bar{\partial}_2\rightarrow \underline{\C}_0\oplus F_1^2,\]
from Lemma~\ref{lemmaA21}. Orient the fiber $\underline{\C}_0$ by its complex structure and the fiber $F^2_1 = \R \subset \underline\C_1$ by the basis $1$.  By Lemma~\ref{lemmaA21} the map $f_2$ is orientation-preserving. Since
\[
f_2((z^2+1))=(1,2),\quad f_2(i(1-z^2))=(i,0), \quad f_2(z)=(0,1),
\]
is an oriented basis, it follows that the basis
\[
(z^2+1), \quad i(1-z^2), \quad z
\]
determines the orientation on $\ker\bar{\partial}_2.$
\end{proof}
We choose the orientation on $F^{1,1}$ such that the frame \[\zeta(z)=\begin{pmatrix}
\frac{z+1}{4}\\
\frac{-i(1-z)}{4}
\end{pmatrix}, \quad \eta(z)=\begin{pmatrix}
\frac{i(1-z)}{4}\\
\frac{z+1}{4}
\end{pmatrix}\]
is positively oriented. In addition, we choose the spin structure on $F^{1,1}$ such that this frame can be lifted to the associated double cover of the frame bundle.
\begin{lm}\label{lm11bundle} Evaluation at zero
\[f_{1,1}:\ker\Bar{\partial}_{1,1}\rightarrow \underline{\C}^2_0,\]
defines an isomorphism. If the codomain is oriented by the complex structure on $\underline{\C}^2_0,$ this isomorphism is orientation-preserving.
\end{lm}
\begin{proof}
By Lemma~\ref{lemmakerbasis} the map $f_{1,1}$ defines an isomorphism.
We degenerate $(\underline{\mathbb{C}}^2,F^{1,1})$ to $(\underline{\mathbb{C}}^2,F^{0,2})$ in order to determine the orientation on $\ker\bar{\partial}_{1,1}$.
Consider the family of loops
\[
A_t:\partial D\rightarrow GL(2,\mathbb{C}), \quad t\in[0,1],
\]
given by
\[A_t(z)=\begin{pmatrix}
{\frac{1+t^2z}{2}} & \frac{t(z-1)}{2i}\\
\frac{t(1-z)}{2i} & {\frac{t^2+z}{2}}
\end{pmatrix}.\]
We claim that the family of boundary conditions $F(t)$ given by $F(t)_z = A_t(z) \cdot \R^2$ is a degeneration from $F^{1,1}$ to $F^{0,2}.$
Indeed
\[A_0(z)=\begin{pmatrix}
1 & 0\\
0 & z
\end{pmatrix}\begin{pmatrix}
\frac{1}{2} & 0\\
0 & \frac{1}{2}
\end{pmatrix}, \qquad A_1(z)=\begin{pmatrix}
z^\frac{1}{2} & 0\\
0 & z^\frac{1}{2}
\end{pmatrix}\begin{pmatrix}
{\frac{z^{-\frac{1}{2}}+z^{\frac{1}{2}}}{2}} & \frac{z^{\frac{1}{2}}-z^{-\frac{1}{2}}}{2i}\\
\frac{z^{-\frac{1}{2}}-z^{\frac{1}{2}}}{2i} & {\frac{z^{-\frac{1}{2}}+z^{\frac{1}{2}}}{2}}
\end{pmatrix},\]
where the matrix
\[\begin{pmatrix}
{\frac{z^{-\frac{1}{2}}+z^{\frac{1}{2}}}{2}} & \frac{z^{\frac{1}{2}}-z^{-\frac{1}{2}}}{2i}\\
\frac{z^{-\frac{1}{2}}-z^{\frac{1}{2}}}{2i} & {\frac{z^{-\frac{1}{2}}+z^{\frac{1}{2}}}{2}}
\end{pmatrix}\]
is an invertible real matrix for every
$z\in \partial D.$
The family of frames given by the columns of $A_t$ gives rise to a continuous family of orientations and spin structures on $F(t)$. The orientation on $F(1)$ coincides with the orientation on $F^{1,1}.$ The orientation on $F(0)$ coincides with the direct sum orientation $F^{0,2} = F^0 \oplus F^2.$  The spin structure on $F(1)$ coincides with the spin structure on  $F^{1,1}$ and the spin structure on $F(0)$ coincides with the direct sum spin structure on $F^{0,2}$ as in Lemma~\ref{lm8.1jake}.
Let
\[R(t):=\begin{pmatrix}
1 & 0\\
0 & t
\end{pmatrix}, \quad B(t,z):=\frac{1}{4}\begin{pmatrix}
(3+t^2)+(t^2-1)z & i(1-t^2)+i(t^2-1)z\\
i(\frac{1}{t}-t)+i(t-\frac{1}{t})z & (3t+\frac{1}{t})+(\frac{1}{t}-t)z
\end{pmatrix}.\]
Then $B(t,z)$ is an invertible holomorphic matrix on $D$ for every $t\in(0,1]$, and the matrices satisfy
\[B(t,z)|_{\partial D}A_1(z)R(t)=A_t(z),\quad  t\in(0,1], \ z\in\partial D.\]
Hence, $B(t,z)$ is an isomorphism between $(\underline{\C}^2, F(1))$ and $(\underline{\C}^2, F(t))$ for every $t\in(0,1]$. Denote by $\bar{\partial}(t)$ the family of the Cauchy-Riemann operators obtained from $(\underline{\C}^2, F(t))$ for $t\in(0,1]$. Note that $\ker\bar{\partial}(1)=\ker\bar{\partial}_{1,1}$.
By Lemma~\ref{lemmakerbasis}
the vectors \[\begin{pmatrix}
1+z\\
0
\end{pmatrix},\quad \begin{pmatrix}
i(1-z)\\
0
\end{pmatrix},\quad \begin{pmatrix}
0\\
1+z
\end{pmatrix},\quad \begin{pmatrix}
0\\
i(1-z)
\end{pmatrix}  ,\] form a basis for $\ker\Bar{\partial}_{1,1}.$
So, the following vectors form a basis for $\ker\bar{\partial}(t)$:
\[
v_{1,t}(z):=B(t,z)\begin{pmatrix}
1+z\\
0
\end{pmatrix},
\quad v_{2,t}(z):=B(t,z)\begin{pmatrix}
i(1-z)\\
0
\end{pmatrix},
\]
\[
v_{3,t}(z):=B(t,z)\begin{pmatrix}
0\\
1+z
\end{pmatrix},
\quad v_{4,t}(z):=B(t,z)\begin{pmatrix}
0\\
i(1-z)
\end{pmatrix}.\]
Let
\begin{align*}
w_{1,t}(z)&:=v_{1,t}(z)-v_{4,t}(z)=\begin{pmatrix}
zt^2+1\\
izt-it
\end{pmatrix}, \\
 w_{2,t}(z)&:=v_{1,t}(z)+v_{4,t}(z)=\frac{1}{2}\begin{pmatrix}
z^2(t^2-1)+2z+t^2+1\\
iz^2(t-\frac{1}{t})+i(t+\frac{1}{t})-2izt
\end{pmatrix},\\
w_{3,t}(z)&:=v_{2,t}(z)-v_{3,t}(z)=\frac{1}{2}\begin{pmatrix}
iz^2(1-t^2)+(1+t^2)i-2iz\\
z^2(t-\frac{1}{t})-t-\frac{1}{t}-2zt
\end{pmatrix},
\\
w_{4,t}(z)&:=v_{2,t}(z)+v_{3,t}(z)=\begin{pmatrix}
-iz+i\\
t+\frac{z}{t}
\end{pmatrix}.
\end{align*}
For $i=2,3,4,$ define $u_{i,t}(z):=tw_{i,t}(z)$ when $t>0,$ and $u_{i,0}(z):=\lim_{t\rightarrow 0}tw_{i,t}(z).$
In addition, define $u_{1,t}(z):=w_{1,t}(z).$
Note that the bases
\[
\{v_{1,t}, v_{2,t}, v_{3,t}, v_{4,t}\}, \qquad \{w_{1,t}, w_{2,t}, w_{3,t}, w_{4,t}\}, \qquad \{u_{1,t}, u_{2,t}, u_{3,t}, u_{4,t}\}
\]
determine the same orientation on $\ker\bar{\partial}(t)$ for $t\in(0,1]$.
We have
\begin{align*}
u_{1,0}(z)&= \begin{pmatrix}
1\\
0
\end{pmatrix},\\
u_{2,0}(z)&= \frac{1}{2}\begin{pmatrix}
0\\
i(1-z^2)
\end{pmatrix},\\
u_{3,0}(z)&= -\frac{1}{2}\begin{pmatrix}
0\\
z^2+1
\end{pmatrix},\\
u_{4,0}(z)&= \begin{pmatrix}
0\\
z
\end{pmatrix}.
\end{align*}
By Lemma~\ref{lemmaind2} the basis $\{-u_{3,0},u_{2,0},u_{4,0}\}$ determines the orientation on $\ker\bar{\partial}_2$, and by Lemma~\ref{remarkker0}, the vector $u_{1,0}(z)$ determines the orientation on $\ker\bar{\partial}_0.$
Hence, by Lemma~\ref{orientation_on_ker} the basis \[
\{ u_{1,0},u_{2,0},u_{3,0},u_{4,0}\}
\]
determines the orientation on
 $\ker\bar{\partial}(0)$.
Continuity implies that $\{u_{1,t}, u_{2,t}, u_{3,t}, u_{4,t}\}$ is positively oriented for every $t\in [0,1],$ and so is  $\{v_{1,t}, v_{2,t}, v_{3,t}, v_{4,t}\}$ for $t\in (0,1]$. Hence, the basis
\[v_{1,1}= \begin{pmatrix}
1+z\\
0
\end{pmatrix}, \quad v_{2,1}= \begin{pmatrix}
i(1-z)\\
0
\end{pmatrix}, \quad v_{3,1}= \begin{pmatrix}
0\\
1+z
\end{pmatrix}, \quad v_{4,1}= \begin{pmatrix}
0\\
i(1-z)
\end{pmatrix}.\]
is positively oriented. Therefore, $f_{1,1}$ is orientation-preserving.

\end{proof}

The orientation convention for $\partial D$ is the counter-clockwise orientation. For a point $z\in\partial D$, we identify $T_z\partial D$ with $iz\R,$ so $iz$ is a positively oriented basis. The following definition is the orientation convention for $\mathrm{PSL}(2,\R)$ in~\cite{FOOO}.
\begin{dfn}\label{basis}
Let $z_0, z_1, z_2\in \partial D$ be three distinct points in anticlockwise order. Consider the embedding
\[F:\mathrm{PSL}(2,\R)\rightarrow \partial D\times\partial D\times\partial D,\]
\[g\mapsto(g\cdot z_0, g\cdot z_1, g\cdot z_2).\]
The orientation on $\mathrm{PSL}(2,\R)$ is determined such that $F$ is orientation-reversing.
\end{dfn}
Consider the Lie algebra $\mathfrak{psl}(2,\mathbb{R}).$ Let
\[\eta_1=\begin{pmatrix}
1 & 0\\
0 & -1
\end{pmatrix},
\quad \eta_2= \begin{pmatrix}
0 & 0\\
1 & 0 \end{pmatrix},\quad
\eta_3=  \begin{pmatrix}
0 & 1\\
0 & 0
\end{pmatrix},\] be a basis for this algebra.
\begin{lm}\label{etalemma}
$\{\eta_1,\eta_2,\eta_3\}$ is a positively oriented basis of $\mathfrak{psl}(2,\R).$
\end{lm}
\begin{proof}
Let $\mH$ denote the closed upper half plane and let $\bmH$ denote the compactification by adding a point at infinity.
Since there is a biholomorphism $D\rightarrow \bmH$, we can identify the range of the map $F$ from Definition~\ref{basis} with $\partial\bmH\times \partial\bmH \times\partial\bmH$. Choose
\[z_0=-1,\quad z_1=0,\quad z_2=1.\]
The orientation of
\[
T_{-1}\partial\bmH\times T_0\partial\bmH \times T_1\partial\bmH
\]
is determined by the basis
\[(1,0,0),\quad (0,1,0)\quad (0,0,1).\]
We have
\begin{align*}
    \frac{d}{dt}|_{t=0}\exp(t\eta_1)\cdot z&=\frac{d}{dt}|_{t=0}\begin{pmatrix}
e^t & 0\\
0 & e^{-t}
\end{pmatrix}\cdot z\\
    &=\frac{d}{dt}|_{t=0}\frac{e^tz}{e^{-t}}\\
    &=2z.
\end{align*}
Similarly,
\[
\frac{d}{dt}|_{t=0}\exp(t\eta_2)\cdot z=-z^2, \quad  \frac{d}{dt}|_{t=0}\exp(t\eta_3)\cdot z=1.\]
Hence,
\[F(\eta_1)= (-2,0,2),\quad F(\eta_2)= (-1,0,-1),\quad F(\eta_3)= (1,1,1).\]
Since
\[
\det \begin{pmatrix}
-2 & 0 & 2 \\
-1 & 0 & -1 \\
1 & 1 & 1
\end{pmatrix} < 0,
\]
it follows that $F$ is orientation-reversing. Thus,
by Definition~\ref{basis} the basis $\{\eta_1, \eta_2, \eta_3\}$ is positively oriented.
\end{proof}
\begin{lm}\label{lemma.psl}
The action of $\mathrm{PSL}(2,\R)$ on $D$ induces
an isomorphism
\[\mathcal{D}:\mathfrak{psl}(2,\mathbb{R})\rightarrow \ker\Bar{\partial}_2\]
which is orientation-reversing.
\end{lm}
\begin{proof}
Identify $D$ with the one-point compactified upper half plane $\bmH$ and identify the Riemann-Hilbert pair $(\underline \C,F^2)$ with the Riemann-Hilbert pair $(T\bmH,T\partial\bmH).$ The latter identification preserves the canonical orientation on the determinant line by Proposition~\ref{pr:canor}.
The map $\mathcal{D}$ is given by
\[\mathcal{D}(A)=(z \mapsto \frac{d}{dt}|_{t=0}\exp(tA)\cdot z).\]
Let $ev:=ev_i\oplus ev_0$, where $ev_i$ and $ev_0$ are the evaluation maps at $i$ and $0$ respectively.
By Lemma~\ref{lemmaA21}, $\mathcal{D}$ is orientation reversing if and only if $ev\circ \mathcal{D}$ is orientation reversing.
Consider the points $0,i\in\bmH$.
The basis
\[
(1,0),\quad (i,0)\quad (0,1)
\]
is positively oriented basis of $T_i\bmH\oplus T_0\partial\bmH$.
By the calculation of the previous lemma we have
\[
ev\circ\mathcal{D}(\eta_1)=(2i,0), \quad ev\circ\mathcal{D}(\eta_2)=(1,0), \quad ev\circ\mathcal{D}(\eta_3)=(1,1).
\]
Since
\[
\det \begin{pmatrix}
0 & 2 & 0 \\
1 & 0 & 0 \\
1 & 0 & 1
\end{pmatrix} < 0,
\]
it follows that $\mathcal{D}$ is orientation-reversing.
\end{proof}
 \subsection{Riemann-Hilbert pairs and holomorphic disks}\label{subsection_5.5}

Let $(X,J)$ be an $n$-dimensional complex manifold, and let $L\subset X$ be a smooth totally real $n$-dimensional submanifold. A holomorphic disk $u:(D,\partial D)\rightarrow (X,L)$ gives rise to a holomorphic vector bundle $u^*TX\rightarrow D$ and a smooth totally real subbundle $u|_{\partial D}^*TL\subset u^*TX|_{\partial D}$. Thus, we obtain a Riemann-Hilbert pair $(E_u,F_u)$ associated to $u.$

\subsubsection{A useful example}
\begin{lm}\label{isomorphism_2_0_0}
Let $u:(D,\partial D)\rightarrow (\x,L_\tr)$ be the holomorphic disk given by $u(z)=e^{-i\xi_v\log z}\cdot \tr.$ Let $\theta_1,\theta_2 \in \mathfrak{su}(2)$ be such that $\xi_v,\theta_1,\theta_2$ form a basis of $\mathfrak{su}(2).$ Equip $L_\tr$ with the orientation and spin structure arising from the trivialization of $TL_\tr$ given by the infinitesimal action of $\mathfrak{su}(2)$ and the basis $\xi_v,\theta_1,\theta_2.$ Let
\[
(E,F):=(u^*T\x,u|_{\partial D}^*TL_\tr).
\]
Then, the sections
\begin{equation}\label{eq:holosecs}
\frac{\xi_v}{z}\cdot u,\; \theta_1\cdot u,\;\theta_2\cdot u \in \Gamma(E)
\end{equation}
form a holomorphic frame.
Moreover, there is an isomorphism of spin Riemann-Hilbert pairs
\[
\Psi:(E,F)\rightarrow (\underline{\C}^3,F^2\oplus F^0\oplus F^0)
\]
given by
\begin{equation}\label{eq:iso001}
    \frac{\xi_v}{z}\cdot u\mapsto \begin{pmatrix}
        1\\
        0\\
        0
    \end{pmatrix},\quad\theta_1\cdot u \mapsto \begin{pmatrix}
        0\\
        1\\
        0
    \end{pmatrix},\quad
\theta_2\cdot u\mapsto \begin{pmatrix}
        0\\
        0\\
        1
    \end{pmatrix}.
\end{equation}
 \end{lm}
 \begin{proof}
Recalling Example~\ref{ex}, we see that the kernel of the infinitesimal action of $\mathfrak{sl}(2,\mathbb{C}) $ at $u(0)$ is spanned by $\xi_v$. Given this, Smith \cite[Appendix A.3]{smith2020monotone} shows that the sections~\eqref{eq:holosecs} form a holomorphic frame for~$E.$
We split $(E,F)$ by
\begin{align*}
    E=\left< \frac{\xi_v}{z}\cdot u\right>_\C \oplus \left< \theta_1\cdot u \right>_\C\oplus \left< \theta_1\cdot u \right>_\C,
\end{align*}
\begin{align*}
     F=z\left< \frac{\xi_v}{z}\cdot u\right>_\R \oplus \left< \theta_1\cdot u \right>_\R\oplus \left< \theta_1\cdot u \right>_\R.
\end{align*}
It follows that the formula~\eqref{eq:iso001} gives the desired isomorphism $\Psi.$
Recall the choices of the spin structures and orientations of $F^2$ and $F^0$ given in Section~\ref{choose_spin_structure}.
By Lemma~\ref{lm8.1jake} the frame
\[\begin{pmatrix}
        z\\
        0\\
        0
    \end{pmatrix},\quad
    \begin{pmatrix}
        0\\
        1\\
        0
    \end{pmatrix},\quad
    \begin{pmatrix}
        0\\
        0\\
        1
    \end{pmatrix}.
\]
of $F^2\oplus F^0\oplus F^0$ is positively oriented and can be lifted to the associate double cover of the frame bundle.
Therefore, the isomorphism $\Psi$ preserves orientation and spin structure.
\end{proof}

\subsubsection{Orientation convention for disk moduli spaces}
Let $(X,\omega)$ be a symplectic manifold with $\omega$-tame (integrable) complex structure $J,$ and let $L \subset X$ be a Lagrangian submanifold.
Let $\widetilde{\mathcal{M}}(\beta)$ denote the space of parameterized $J$-holomorphic maps $u: (D^2,\partial D^2) \to (X,L)$ such that $u_*([D^2,\partial D^2]) = \beta$. Each $u \in \widetilde{\mathcal{M}}(\beta)$, determines a Riemann-Hilbert pair $(E_u,F_u) = (u^*TX,u^*TL)$ as explained above. For the following discussion, we assume that the linear Cauchy-Riemann operator $\bar\partial_{(E_u,F_u)} : \Gamma((D^2,\partial D^2),(E_u,F_u)) \to \Gamma(D^2,\Omega^{0,1}(E_u))$ is surjective for every $u$, so $\widetilde{\mathcal{M}}(\beta)$ is a smooth manifold and there is a canonical isomorphism
\[
T_u\widetilde{\mathcal{M}}(\beta) \simeq \ker \bar\partial_{(E_u,F_u)}.
\]
Thus, Proposition~\ref{pr:canor} gives a canonical orientation of $\widetilde{\mathcal{M}}(\beta).$

We discuss now the relevant conventions concerning the orientations of the moduli spaces $\mathcal{M}_{k+1,l}(\beta)$ of unparameterized stable $J$-holomorphic maps with marked points. The following definition is Convention 8.2.1 from~\cite{FOOO}.
\begin{dfn}\label{dfn:qo}
Let $G$ be an oriented Lie group with a smooth, proper, free right action on an oriented manifold $M$. For $p \in M,$ let $\varphi_p : G \to M$ be given by $g \mapsto p\cdot g.$ Let $\pi : M \to M/G$ be the quotient map. Split the short exact sequence
\[
0 \to T_eG \overset{d\varphi_p}\to T_pM \overset{d\pi}\to T_{\pi(p)}(M/G) \to 0
\]
to obtain an isomorphism
\[
T_pM \simeq T_{\pi(p)}(M/G) \oplus T_eG.
\]
The \textbf{quotient orientation} on $M/G$ is determined by the condition that the preceding isomorphism preserves orientation for all $p \in M$.
\end{dfn}
The following definition is based on~\cite[p. 698]{FOOO}.
\begin{dfn}\label{dfn:ormod}
Let $U \subset \widetilde{\mathcal{M}}(\beta) \times (\partial D^2)^{k+1} \times (\Int D^2)^l$ denote the open subset where the marked points are pairwise disjoint and the cyclic ordering on the boundary marked points given by the orientation on $\partial D^2$ induced from the complex orientation of $D^2$ agrees with the order of the labels. Thus, points of $U$ are tuples $(u,z,w)$ where
\[
u \in \widetilde{\mathcal{M}}(\beta), \quad z = (z_0,\ldots,z_k) \in \partial D^2, \quad w = (w_1,\ldots,w_l) \in \Int D^2.
\]
An automorphism of the disk $\psi \in \mathrm{PSL}_2(\R)$ acts on $U$ by
\[
\psi \cdot (u,z,w) = (u\circ \psi, \psi^{-1}(z),\psi^{-1}(w)).
\]
The \textbf{orientation} of $\mathcal{M}_{k+1,l}(\beta)$ is determined by the quotient orientation on the open subset $U/PSL_2(\R) \subset \mathcal{M}_{k+1,l}(\beta).$
\end{dfn}

\subsection{Computation of \texorpdfstring{$\overline{OGW}_{1, 1}$}{degree 1 single boundary constraint invariant}}\label{subsection_ogw11}
Recall from Section~\ref{sec4} that we denote by $\mathcal{M}_{k+1,l}(\beta)$ the moduli space of holomorphic disks $u:(D, \partial D)\rightarrow (\x,L_\triangle)$ representing a class $\beta\in H_2(\x,L_\triangle;\mathbb{Z})$ with $k+1$ boundary points and $l$ interior points. We denote elements in $\mathcal{M}_{k+1,l}(\beta)$ by $[u;z_0,\ldots,z_k,w_1,\ldots,w_l]$, where $z_i\in \partial D$ and $w_i\in \mathrm{int}D$.
Consider the moduli space $\mathcal{M}_{1,0}(1)$, and the evaluation map  \[evb_0:\mathcal{M}_{1,0}(1)\rightarrow L_\triangle,\]
\[evb_0([u;z])=u(z).\]
The following lemma is part of Proposition 4.2 from \cite{Smith}.
\begin{lm}\label{degree3}
The evaluation map $evb_0:\mathcal{M}_{1,0}(1)\rightarrow L_\triangle$ is a covering map of degree $\pm 3$.
\end{lm}

\begin{proof}[Proof of Theorem~\ref{thm:ogw11}]
By Theorem~\ref{theorem_lm} we have
\[\ogw_{1,1}= - \int_{\mathcal{M}_{1,0}(1)}{evb_0}^*\bar{b},\] where $\pd([\bar{b}])=pt$.
Hence, $\ogw_{1,1}$ is determined by the degree of the map $evb_0$, and Lemma~\ref{degree3} gives $\ogw_{1,1}=\mp 3,$ where the sign depends on whether $evb_0$ preserves or reverses orientation.

We show that $evb_0$ preserves orientation.
Choose $\theta_1,\theta_2$ that together with $\xi_v$ form a basis of $\mathfrak{su}(2).$ Equip $L_\tr$ with the orientation and spin structure arising from the trivialization of $TL_\tr$ given by the infinitesimal action of $\mathfrak{su}(2)$ and the basis $\xi_v,\theta_1,\theta_2.$ The orientation of this basis does not affect the calculation below, in accordance with the orientation axiom Proposition~\ref{prop:open_axioms}~ \ref{ax_orientation}.

Let $u:(D,\partial D)\rightarrow (\x,L_\triangle)$ be a holomorphic disk of Maslov index 2. Let
\[
(E,F):=(u^*T\x,u|_{\partial D}^*TL_\tr).
\]
By Lemma~\ref{lemmaxi} we can write
$u(z)=A\cdot e^{-i\xi_v \log z}\cdot \tr$, where $A\in \su(2)$. We may assume that $A$ is the identity since acting by $A^{-1}$ doesn't change the isomorphism class of $(E,F)$.

Let $z\in \partial D$ and let $\zeta \in T_z\partial D$ denote the unit vector in the direction of the orientation. By Definition~\ref{dfn:ormod}, the oriented tangent space of $\mathcal{M}_{1,0}(1)$ at $[u;z]$ is given by
\[T_{[u;z]}\mathcal{M}_{1,0}(1)\simeq(\ker\Bar{\partial}_{(E,F)}\oplus T_z\partial D)/\mathfrak{psl}(2,\mathbb{R}).\]
By Proposition~\ref{pr:canor} and Lemma~\ref{isomorphism_2_0_0}, we have an orientation-preserving isomorphism
\[
\ker\Bar{\partial}_{(E,F)}\oplus T_z\partial D \simeq \ker \bar{\partial}_2 \oplus \ker\bar{\partial}_0 \oplus \ker \bar{\partial}_0 \oplus T_z\partial D.
\]
By Lemma~\ref{lemmakerbasis}, we have an orientation-reversing isomorphism
\[
\ker \bar{\partial}_2 \oplus \ker\bar{\partial}_0 \oplus \ker \bar{\partial}_0 \oplus T_z\partial D \simeq \ker\bar{\partial}_0 \oplus \ker \bar{\partial}_0 \oplus T_z\partial D \oplus \ker \bar{\partial}_2
\]
The linearization of the action of $\mathrm{PSL}_2(\R)$ in Definition~\ref{dfn:ormod} composed with the projection on $\ker \bar \partial_2$ gives the map $\mathcal{D}:\mathfrak{psl}(2,\mathbb{R})\rightarrow \ker\Bar{\partial}_2$ of Lemma~\ref{lemma.psl}. So, by Definition~\ref{dfn:qo} and Lemmas~\ref{lemma.psl} and~\ref{lemmakerbasis}, we have an orientation-preserving isomorphism
\[
\ker\bar{\partial}_0\oplus\ker\Bar{\partial}_0\oplus T_z\partial D \simeq (\ker\Bar{\partial}_{(E,F)}\oplus T_z\partial D)/\mathfrak{psl}(2,\mathbb{R}),
\]
given by
\[
1 \oplus 0 \oplus 0 \mapsto [\theta_1 \cdot u \oplus 0], \qquad 0 \oplus 1 \oplus 0 \mapsto [\theta_2 \cdot u\oplus 0],
\qquad
0 \oplus 0 \oplus \zeta \mapsto [0 \oplus \zeta].
\]
Abbreviate
\[
\Bar{\theta}_i=[\theta_i\cdot u\oplus 0], \quad \Bar{\zeta}=[0\oplus \zeta].
\]
Thus, by Lemmas~\ref{orientation_on_ker} and~\ref{remarkker0}, the orientation on $T_{[u:z]}\mathcal{M}_{1,0}(1)$ is given by the basis $\Bar{\theta}_1, \Bar{\theta}_2, \Bar{\zeta}$.
In order to show that $evb_0$ preserves orientation, it suffices to show that
\[
devb_{0_{[u;z]}}:T_{[u;z]}\mathcal{M}_{1,0}(1)\rightarrow T_{u(z)}L_\triangle
\]
preserves orientation when $z=1$.
The tangent vector $\zeta \in T_1 \partial D$ is represented by the path $t \mapsto e^{it}.$
So,
\[
devb_{0_{[u;1]}}(\Bar{\zeta}) = du_1(\zeta) = \left.\frac{d}{dt}u(e^{it})\right|_{t=0}= \left.\frac{d}{dt} e^{t\xi_v}\right|_{t = 0} = \xi_v\cdot u(1).
\]
Since
\[devb_{0_{[u;1]}}(\Bar{\theta}_1)= \theta_1\cdot u(1),\quad devb_{0_{[u;1]}}(\Bar{\theta}_2)= \theta_2\cdot u(1), \quad devb_{0_{[u;1]}}(\Bar{\zeta})= \xi_v\cdot u(1),\]
it follows that $evb_0$ preserves orientation. Therefore $\ogw_{1,1}=-3$.
\end{proof}

\subsection{Computation of \texorpdfstring{$\overline{OGW}_{1, 0}(\Gamma_2)$}{degree 1 single interior constraint invariant)}}
 By Theorem~\ref{theorem_lm} we have
\begin{equation*}
\ogw_{1, 0}(\Gamma_2)= -\int_{\mathcal{M}_{0,1}(1)}evi_1^*\gamma_2,
\end{equation*}
where $\gamma_2$ is a a differential form representing $\Gamma_2 \in H^4(\x, L_\triangle; \R)$.
By Poincar\'e-Lefschetz duality we have $H^4(\x, L_\triangle; \R) \simeq H_2(\x \backslash L_\triangle;\R)$.
Hence, our strategy for computing $\ogw_{1,0}(\Gamma_2)$
is to find a complex curve in $\cp^3\setminus L_\tr$ representing the Poincar\'e-Lefschetz dual to $\Gamma_2,$ and determine how many holomorphic disks intersect it.

Any complex subvariety $\Upsilon \subset \x \backslash L_\triangle$ of complex
dimension 2 represents $k\pd([\omega])$ where $k \in \mathbb{Z}$ is the degree
of $\Upsilon$, i.e. the number of intersection points with a generic line.
Consider complex subvarieties $\Upsilon_1, \Upsilon_2$
of degree $k_1, k_2$ respectively, in general position. Then
$\Upsilon_1 \cap \Upsilon_2$ represents $k_1k_2\pd([\omega^2])$,
and thus $\pd(\Gamma_2) = [\Upsilon_1 \cap \Upsilon_2] / (k_1k_2)$.

We consider the case where $\Upsilon_1, \Upsilon_2$
are small perturbations of the anticanonical divisor~$Y_\triangle$. For $i = 1, 2$, let $g_i \in \su(2)$ be lifts of rotations by arbitrary small angles $\epsilon_i$
about different axes.
Write
\[
\Upsilon_i = \{[a, g_i(a), b] \in \sym^3\cp^1 \ |\  a,b \in \cp^1 \}.
\]
Note that $\Upsilon_i\subset \cp^3\setminus L_\tr$ since $\epsilon_i$ are small.

\begin{prop}\label{propdeg}
$\deg \Upsilon_i = 4$
\end{prop}

\begin{proof}
Write $a=[a_0:a_1],$ $b=[b_0: b_1]\in\cp^1$. Identifying $\cp^3$ with the projectivization of the space of homogeneous polynomials of degree $3$ in two variables, we can write \[\Upsilon_i = \left\{\left[(a_0 y - a_1 x)\big(((g_i)_{11}a_0+(g_i)_{12}a_1)y - ((g_i)_{21}a_0+(g_i)_{22}a_1) x\big) (b_0y - b_1 x)\right]\,| \,  a,b\in \cp^1\right\},
\]
where $(g_i)_{kl}$ are the components of the matrix $g_i.$
Hence, we have an embedding \[f_i: \cp^1 \times \cp^1 \rightarrow \cp^3,\]
\[([a_0:a_1],[b_0:b_1])\mapsto[f_i^0(a,b):f_i^1(a,b):f_i^2(a,b):f_i^3(a,b)],\]
where $f_i^j$ are homogeneous polynomials of bidegree $(2,1),$ such that $\Upsilon_i=\Ima f_i$.

Let $H_1$ and $H_2$ be hyperplanes in $\cp^3$ given by
\[
H_1 = \{\sum h_j z_j = 0\}, \qquad H_2 =
\{ \sum k_j z_j = 0 \}.
\]
The intersection $H_1\cap H_2$ is a generic line, so we have
\[
\deg\Upsilon_i=\# \Upsilon_i \cap H_1 \cap H_2.
\]
The preimages $f_i^{-1}(H_1)$
and $f_i^{-1}(H_2)$
are the vanishing sets of the polynomials
\[p^h_i(a,b)=\sum_{j=0}^3 h_jf_i^j(a,b),\qquad p^k_i(a,b)=\sum_{j=0}^3 k_jf_i^j(a,b) \]
of bidegree $(2,1)$.
Denote by $\pi_1, \pi_2 : \cp^1 \times \cp^1 \rightarrow \cp^1$ the projection maps.
Let
\[
\sheafo(i,j) = \pi_1^* \sheafo(i) \otimes \pi_2^*\sheafo (j).
\]
So,
\[
c_1(\sheafo(i,j)) = (i,j) \in H^2(\cp^1 \times \cp^1) = \Z \oplus \Z.
\]
The polynomials $p^h_i, p^k_i,$ define sections of the line bundle $\sheafo(2,1).$
So,
\begin{multline*}
\# \Upsilon_i \cap H_1 \cap H_2 = \# f_i^{-1}(H_1 \cap H_2) = \# f_i^{-1}(H_1) \cap f_i^{-1}(H_2) = \\ = \int_{\cp^1 \times \cp^1} c_1(\sheafo(2,1)) \smile c_1(\sheafo(2,1)) = 4.
\end{multline*}

\end{proof}

\begin{proof}[Proof of Theorem \ref{thm:ogw102}]
Let $\gamma_2$ be a a differential form representing $\Gamma_2 \in H^4(\x, L_\triangle; \R)$. We show below that $evi_1$ is transverse to $\Upsilon_1 \cap \Upsilon_2.$
Thus, by Theorem~\ref{theorem_lm} and Poincar\'e duality we have
\begin{equation}\label{eq:ogw102enum}
\ogw_{1, 0}(\Gamma_2)= -\int_{\mathcal{M}_{0,1}(1)}evi_1^*\gamma_2 = -\frac{\# evi_1^{-1}(\Upsilon_1 \cap \Upsilon_2)}{\deg \Upsilon_1 \deg \Upsilon_2}.
\end{equation}
So, we need to count Maslov 2 disks passing through $\Upsilon_1 \cap \Upsilon_2$ with sign given by the orientation of the moduli space $\mathcal{M}_{0,1}(1)$ and the complex orientation of the normal bundle to $\Upsilon_1 \cap \Upsilon_2.$

We represent a point of $\Upsilon_i$ by three dots with an arrow between two of them labeled by $g_i$ as follows.
\[
\begin{tikzcd}
 \bullet \arrow[r, "g_i"] &  \bullet & \bullet
\end{tikzcd}
\]
Hence, a point of $\Upsilon_1 \cap
\Upsilon_2$ is represented by three dots with two arrows between them labeled by
$g_1$ and $g_2$. There are 6 types up to continuous deformations:
\[
\begin{tikzcd}
 \bullet \arrow[r, "g_1"] & \bullet  \arrow[r, "g_2"] & \bullet  \tag{$\Theta_1$}
\end{tikzcd}
\]

\[
\begin{tikzcd}
 \bullet \arrow[r, "g_1"] & \bullet & \arrow[l, swap, "g_2"]  \bullet
\end{tikzcd}
\tag{$\Theta_2$}
\]

\[
\begin{tikzcd}
 \bullet \arrow[r, "g_2"] & \bullet \arrow[r, "g_1"] & \bullet
\end{tikzcd}
\tag{$\Theta_3$}
\]

\[
\begin{tikzcd}
 \bullet & \arrow[l, swap, "g_1"]  \bullet \arrow[r, "g_2"] & \bullet
\end{tikzcd}
\tag{$\Theta_4$}
\]

\[
\begin{tikzcd}
 \bullet \arrow[r, bend left, "g_1"] & \arrow[l, bend left, swap, "g_2"]  \bullet
 & \bullet
\end{tikzcd}
\tag{$\Theta_5$}
\]

\[
\begin{tikzcd}
 \bullet \arrow[r, bend left, "g_1"]  \arrow[r, bend right, swap, "g_2"] &  \bullet
 & \bullet
\end{tikzcd}
\tag{$\Theta_6$}
\]
More explicitly,
\begin{align*}
\Theta_1 &=\{ [a, g_1(a), g_2g_1(a)] \in \sym^3 \cp^1|\ a\in \cp^1\},    \\
\Theta_2 &=\{ [a, g_1(a), g_2^{-1}g_1(a)] \in \sym^3 \cp^1|\ a\in \cp^1  \},   \\
\Theta_3&=\{ [a, g_2(a), g_1g_2(a)] \in \sym^3 \cp^1|\ a\in \cp^1  \},   \\
\Theta_4&=\{ [a, g_1^{-1}(a), g_2g_1^{-1}(a)] \in \sym^3 \cp^1|\ a\in \cp^1\},   \\
\Theta_5&=\{[a, g_1(a), c]\in \sym^3 \cp^1 |\ a,c \in \cp^1, \ g_2g_1(a) =  a\},   \\
\Theta_6 &= \{[a, g_1(a), c] \in \sym^3 \cp^1 | \ a,c \in \cp^1,\ g_2^{-1} g_1(a) = a \}.
\end{align*}
Since each of $\Theta_1,\ldots,\Theta_4$ is the image of a map $\cp^1 \to \cp^3$ of degree 3, it follows that $\deg \Theta_1 = \ldots  =  \deg \Theta_4 =3.$
We claim that each of $\Theta_5, \Theta_6$ is a union of two lines. Indeed,
both $g_2g_1$ and $g_2^{-1}g_1$ have two eigenvectors,
so there are two choices for $a\in\cp^1$.
Each such choice $a$ gives a line $\{[a, g_1(a), c]|\ c\in \cp^1 \}.$ Hence, $\deg \Theta_5 = \deg \Theta_6 = 2$.
Therefore, $\sum_{i=1}^6 \deg \Theta_i=16$.
By Bezout's theorem and Proposition~\ref{propdeg}, we have
\[
\deg \Upsilon_1 \cap \Upsilon_2  =\deg \Upsilon_1 \deg \Upsilon_2
= 16.
\]
Since the degrees coincide, it follows that each
$\Theta_i$ occurs with  multiplicity $1$ in the intersection $\Upsilon_1 \cap \Upsilon_2.$ In particular, the intersection is generically transverse.

By Lemma \ref{lemmaxi} all Maslov 2 disks are axial of the form $u(z)=A\cdot e^{-i\xi_v\log z}\cdot \tr$.
Let $p_1, p_2, p_3 \in \mathbb{C}P^1$ such that $[p_1, p_2, p_3]\in\sym^3 \cp^1$ is in the image of a Maslov index $2$ disk. As in Example~\ref{ex}, there exist $i,j\in\{1,2,3\}$ such that $d(p_i,p_j)\ge e$, where $e$ is the distance between two vertices in an equilateral triangle on a great circle in $\cp^1$. By choosing $\epsilon_i$ sufficiently small, $\Theta_1,\ldots, \Theta_4$ can be brought
arbitrarily close to the locus $N_\triangle$ where all 3 points coincide. This rules out a Maslov index 2 disk passing through them. Hence, the number of Maslov 2 disks through $\Upsilon_1 \cap \Upsilon_2$
is equal to the number of Maslov $2$ disks through $\Theta_5$ and $\Theta_6.$  Since $a$ and $g_1(a)$ are sufficiently close, it follows by Lemma~\ref{pointslemma} that one such disk passes through each of the lines making up $\Theta_5, \Theta_6,$ for a total of four disks.

Next we determine the sign with which each disk passing through $\Theta_5$ and $\Theta_6$ contributes to $\# evi_1^{-1}(\Upsilon_1 \cap \Upsilon_2).$ The treatment of $\Theta_5$ and $\Theta_6$ is parallel, so we focus on $\Theta_6$.
Let $a \in \cp^1$ be one of the two solutions of $g_2^{-1}g_1(a) = a$ as in the above formula for $\Theta_6.$
Denote by $\Theta_6^a$ the corresponding component of $\Theta_6.$
Let $[u;w]\in \mathcal{M}_{0,1}(1)$ with $u(w) \in \Theta_6^a$. Let $(\nu_{\Theta_6})_{u(w)}:= T_{u(w)}\cp^3/ T_{u(w)}\Theta_6$ be the normal space to $\Theta_6$ at $u(w)$ equipped with the complex orientation. The differential ${devi_1}_{[u;w]}$ induces a linear map
\[
\overline{devi_1}_{[u;w]}:T_{[u;w]}\mathcal{M}_{0,1}(1)\to (\nu_{\Theta_6})_{u(w)}.
\]
We show $\overline{devi_1}_{[u;w]}$ is an orientation-reversing isomorphism.

By Lemma~\ref{lemmaxi} we can write $u(z)=A\cdot e^{-i\xi_v\log z}\cdot \tr$.
By applying an appropriate rotation to $\cp^1 = S^2$, which does not affect orientations, we may assume that $a$ and $g_1(a)$ are positioned in the $xz$ plane symmetrically about the south pole as in Figure~\ref{fig:3}.  It follows from Example~\ref{ex} that we may take $A$ to be the identify matrix.
\begin{figure}[ht]
        \centering
        \includegraphics[scale=0.28]{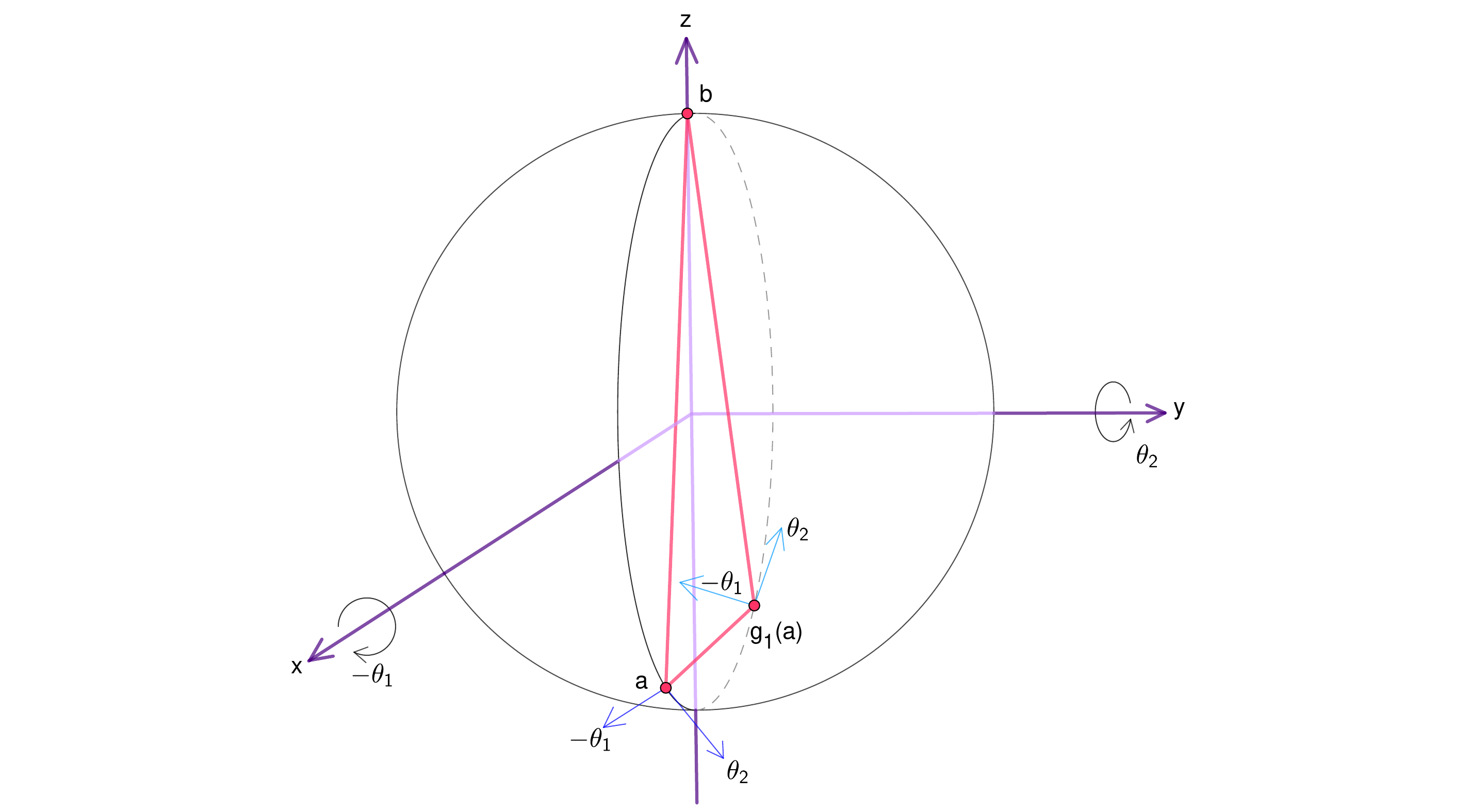}
        \caption{The triangle $u(w).$}
        \label{fig:3}
\end{figure}
We begin by finding an oriented basis for $T_{[u;w]}\mathcal{M}_{0,1}(1)$.
Let
\[
(E,F):=(u^*T\x,u|_{\partial D}^*TL_\triangle).
\]
We choose $\theta_1,\theta_2 \in \mathfrak{su}(2)$ to be lifts of infinitesimal right-handed rotations around the $x$-axis and the $y$-axis respectively. Thus, the trivialization of $TL_\tr$ given by the infinitesimal action of $\mathfrak{su}(2)$ and the basis $\xi_v,\theta_1,\theta_2,$ determines the orientation and the spin structure of
$L_\tr$ given in Section~\ref{subsection_orientation_spin}.
Let $w\in \Int D$. The tangent space of $\mathcal{M}_{0,1}(1)$ at $[u;w]$ is given by
\[
T_{[u;w]}\mathcal{M}_{0,1}(1)\simeq (\ker\Bar{\partial}_{(E,F)}\oplus T_wD)/\mathfrak{psl}(2,\mathbb{R}).
\]
Consider the canonical identification $\C\simeq T_w D$. We denote by $\hat{x}$ and $\hat{y}$ the vectors in $T_w D$ that correspond to $1$ and $i$ in $\C$ respectively.
Hence, the basis $\{\hat{x}, \hat{y}\}$ determines the orientation on $T_w D$ that induced by its complex structure.
By Proposition~\ref{pr:canor} and Lemmas~\ref{orientation_on_ker},~\ref{remarkker0},~\ref{lemma.psl} and~\ref{isomorphism_2_0_0},
we have an orientation-reversing isomorphism
\[
\ker\Bar{\partial}_0\oplus \ker\Bar{\partial}_0\oplus T_wD\simeq(\ker\Bar{\partial}_{(E,F)}\oplus T_wD)/\mathfrak{psl}(2,\mathbb{R}),
\]
given by
\[
1\oplus 0\oplus 0\mapsto[\theta_1\cdot u\oplus 0], \quad 0\oplus 1\oplus 0 \mapsto[\theta_2\cdot u\oplus 0],\quad 0\oplus0\oplus\hat{x}\mapsto [0\oplus \hat{x}], \quad  0\oplus0\oplus \hat{y}\mapsto [0\oplus \hat{y}].
\]
Abbreviate
\[
\bar\theta_i = [\theta_i \cdot u \oplus 0], \qquad \bar x = [0\oplus \hat x], \qquad \bar y = [0\oplus \hat y].
\]
Thus, $\{-\bar\theta_1, \bar\theta_2,\bar{x}, \bar{y}\}$ is an oriented basis for $T_{[u;w]}\mathcal{M}_{0,1}(1)$.

Next, we compute $devi_1$ on $\bar x, \bar y.$
Indeed,
\begin{equation*}
    du_w(\hat{x})=\left.\frac{d}{dt}\right|_{t=0}e^{-i\xi_v\log(w+t)}\cdot \tr = \frac{-i\xi_v}{w}\cdot u(w),
\end{equation*}
and
\begin{equation*}
    du_w(\hat{y})=\left.\frac{d}{dt}\right|_{t=0}e^{-i\xi_v\log(w+it)}\cdot \tr
    = \frac{\xi_v}{w}\cdot u(w).
\end{equation*}
It follows that
\begin{equation}\label{eq:devixy}
{devi_1}_{[u;w]}(\bar{x})=\frac{-i\xi_v}{w}\cdot u(w),\qquad {devi_1}_{[u;w]}(\bar{y})=\frac{\xi_v}{w}\cdot u(w).
\end{equation}

To determine whether $\overline{devi_1}_{[u;w]}$ preserves orientation, it is helpful to work in a holomorphic coordinate chart on $\sym^3(\cp^1)$. Let
\[
\psi : \C^3 \to \sym^3(\cp^1)
\]
be given by
\[
\psi(z_1,z_2,z_3)  = [[1:z_1],[1:z_2],[z_3:1]].
\]
Recall the stereographic projection $p$ from Example~\ref{ex}. Observe that
\[
\psi(p(a),p(g_1(a)),0) = u(w).
\]
Since $a \neq g_1(a)$, we may choose an open $U \subset \C^3$ containing $(p(a),p(g_1(a)),0)$ such that $\hat \psi = \psi|_U$ is a biholomorphism onto its image.
Then
\[
\hat \psi^{-1}(\Theta_6^a) = U \cap(\{(p(a),p(g_1(a)))\}\times \C) = : \hat\Theta_6^a.
\]
Observe that the normal bundle to $\hat\Theta_6^a$ is canonically identified with $\C^2 \times \{0\} \subset \C^3.$
Let $\pi : \C^3 \to \C^2$ be the projection onto the first two factors, and let
\[
\widehat{evi_1} = \pi\circ \hat\psi^{-1}\circ evi_1.
\]
Since $\hat \psi$ is biholomorphic and thus orientation preserving, it suffices to determine whether the map
\[
d\widehat{evi_1}_{[u;w]} :  T_{[u;w]}\mathcal{M}_{0,1}(1)\to \C^2
\]
preserves orientation.

We proceed to compute $d\widehat{evi_1}$ on the oriented basis $\{-\bar \theta_1,\bar \theta_2,\bar x,\bar y\}$ of $T_{[u;w]}\mathcal{M}_{0,1}(1)$.
By Example~\ref{ex}, we have
\begin{equation}
(p(a),p(g_1(a)))= \pi\circ \hat \psi^{-1}\circ u(w) = \left(\sqrt{\frac{w}{3}},-\sqrt{\frac{w}3}\right) \label{eq:psi-1u}
\end{equation}
Since $a$ lies in the $xz$ plane, it follows that $p(a)\in \R$ and consequently $w \in \R_{>0}.$
Recall the map $\varphi^t_{\xi_v}$ from Example~\ref{ex}. For $z \in \C,$ we have
\[
\xi_v\cdot z=\frac{d}{dt}|_{t=0}\varphi^t_{\xi_v}(z)=\frac{i}{2}z.
\]
So, by equations~\eqref{eq:devixy} and~\eqref{eq:psi-1u}, we obtain
\begin{gather*}
d\widehat{evi_1}(\bar y) = \pi\circ d(\hat\psi)^{-1}\left(\frac{\xi_v}{w}\cdot  u(w)\right) =\frac{1}{w}(\frac{i\sqrt{w}}{2\sqrt{3}}, -\frac{i\sqrt{w}}{2\sqrt{3}}), \\
d\widehat{evi_1}(\bar x) = \pi \circ d(\hat\psi)^{-1}\left(\frac{-i\xi_v}{w}\cdot u(w)\right) = \frac{1}{w}(\frac{\sqrt{w}}{2\sqrt{3}}, -\frac{\sqrt{w}}{2\sqrt{3}}).
\end{gather*}
On the other hand, we can see by stereographically projecting Figure~\ref{fig:3} or by direct calculation that
\[
d\widehat{evi_1}(-\bar \theta_1) = -(i,i) + O(\epsilon_1), \qquad d\widehat{evi_1}(\bar \theta_2) = -(1,1) + O(\epsilon_1).
\]
Since
\[
\det \begin{pmatrix}
    0 & -1 & 0 & -1\\
    -1 & 0 & -1 & 0\\
     \frac{1}{2\sqrt{3w}} & 0 & -\frac{1}{2\sqrt{3w}} &0\\
     0 & \frac{1}{2\sqrt{3w}} & 0 &-\frac{1}{2\sqrt{3w}}
\end{pmatrix}=-\frac{1}{3w}<0,
\]
it follows that the basis
\[
d\widehat{evi_1}(\{-\theta_1,\theta_2,\bar x,\bar y\})
\]
is not complex oriented. Thus, $d\widehat{evi_1}$ and consequently also $\overline{devi_1}$ reverse orientation.

Therefore, each of the 4 Maslov 2 disks passing through $\Upsilon_1\cap\Upsilon_2$ contributes to the signed count $\# evi_1^{-1}(\Upsilon_1 \cap \Upsilon_2)$ with  sign $-1$.
By Proposition~\ref{propdeg} and equation~\eqref{eq:ogw102enum},
we obtain
\[
\ogw_{1, 0}(\Gamma_2)=4/16=1/4.
\]
\end{proof}

\subsection{Computation of \texorpdfstring{$\overline{OGW}_{2, 0}(\Gamma_3)$}{degree 2 single interior constraint invariant}}\label{subsection_ogw20}
\begin{proof}[Proof of Theorem \ref{thm:ogw203}]
By Theorem~\ref{theorem_lm} we have
\[\ogw_{2, 0}(\Gamma_3)= -\int_{\mathcal{M}_{0,1}(2)}evi_1^*\gamma_3,\]
where $\gamma_3$ is a representative of $\Gamma_3$. Since $\Gamma_3$ is Poincar\'e dual to a point, it follows that
\[
\ogw_{2, 0}(\Gamma_3) = -\# evi_1^{-1}(p),
\]
where $p \in \C P^3$ is a regular value of $evi_1$.

Let $p \in N_\triangle$. First, we show that $p$ is a regular value of $evi_1$ and compute $\#evi_1^{-1}(p)$ up to sign. We claim that any Maslov 4 disk $u$ that passes through $p$ is an axial disk of type $\xi_f.$
Indeed, by Lemma~\ref{divisor} $\y$ is an
anticanonical divisor, so by Lemma \ref{lm:intersect}
the Maslov index of $u$ is given by twice
the algebraic intersection number $[u] \cdot [\y]$.
Hence, $[u] \cdot [\y] = 2$. Thus, Lemma~\ref{lemma3.7} implies that $u$ intersects $\y$ only at $p$. Hence, by Lemma~\ref{lm:lmpole1}  a pole germ of $u$ is not of type $\xi_v$. So, by Lemma~\ref{lm:lmpole2} $u$ is an axial disk of type $\xi_f$ as claimed. Thus, Lemma~ \ref{lm:lmpole3} guarantees that $p$ is a regular value of $evi_1$. Since by Lemma~\ref{4maslovthroughn} there exists a unique
axial disk of type $\xi_f$ of Maslov index $4$ that passes through $p,$
it follows that $\# evi_1^{-1}(p)=\pm 1$.

It remains to show that in fact $\# evi_1^{-1}(p)= -1$. Let $r:=(0,-1,0),$ and choose
\[
p=[r,r,r]\in N_\Delta \subset \sym^3\cp^1.
\]
Let $u:(D,\partial D)\rightarrow (\x,L_\triangle)$ be the $J$-holomorphic disk of Maslov index 4 that intersects $N_\triangle$ at $p$. As shown above, $u$ is an axial disk of type $\xi_f$, so by Example~\ref{ex2} we have $u(z)= e^{-i\xi_f \log z}\cdot \tr$.
Let
\[
(E,F):=(u^*T\x,u|_{\partial D}^*TL_\triangle).
\]

\begin{figure}[ht]
        \centering
        \includegraphics[scale=0.27]{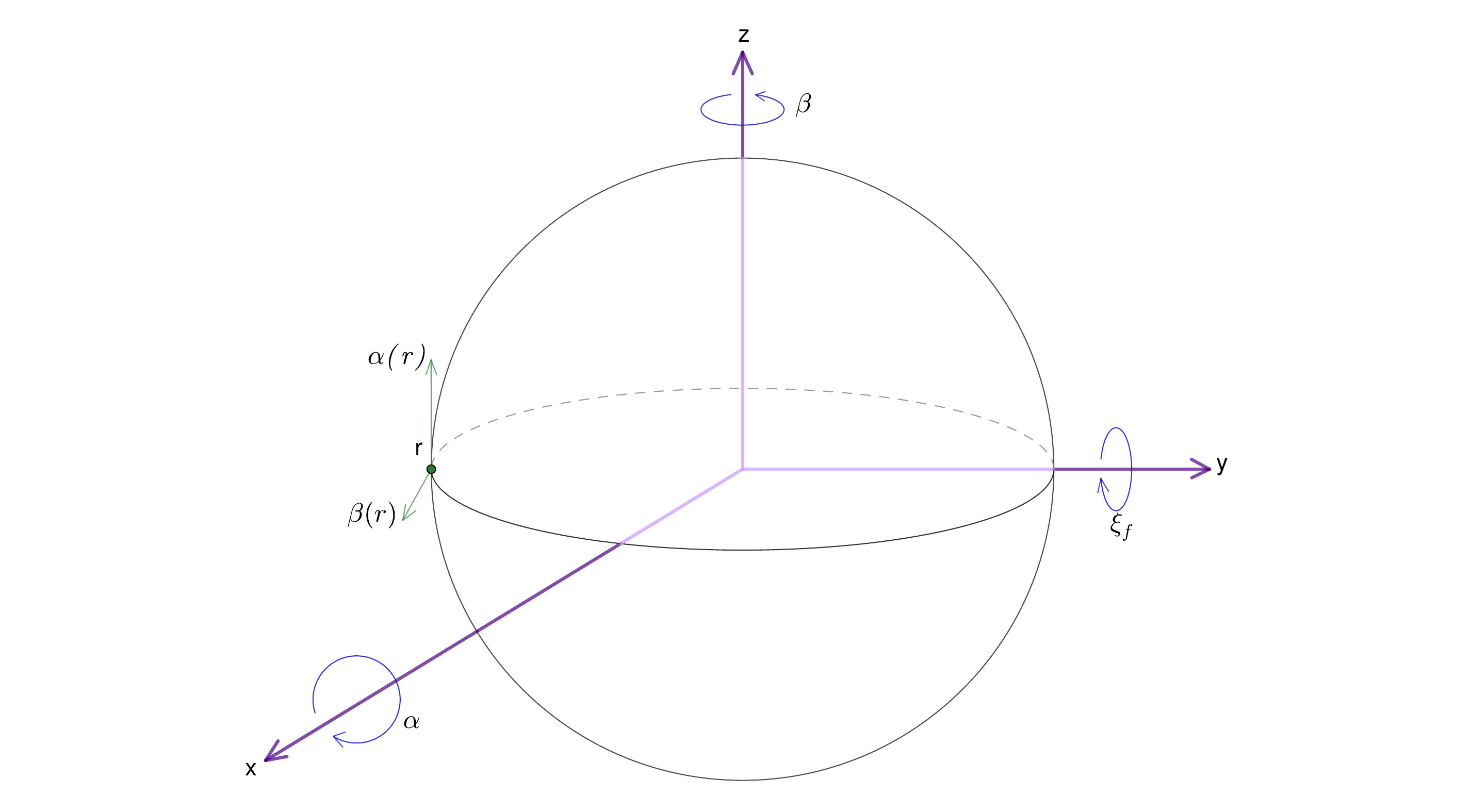}
        \caption{The rotations $\alpha$ and $\beta$ at $p$.}
        \label{fig:6}
    \end{figure}

We construct frames for $E$ and $F$ as follows.
Let $\alpha\in \mathfrak{su}(2)$ be an infinitesimal right-handed rotation about the $-x$ axis, and let $\beta\in \mathfrak{su}(2)$ be an infinitesimal right-handed rotation about the $z$ axis. Recall that $\xi_f$ is an infinitesimal right-handed rotation about the $y$ axis, so $\xi_f, \alpha, \beta$ is a basis of $\mathfrak{su}(2)$ corresponding to infinitesimal right-handed rotations about a right-handed set of orthogonal axes as in the definition of the orientation $\oo_\tr$ and spin structure $\s_\tr$ on $L_\triangle$ given in Section~\ref{subsection_orientation_spin}. Since $u(0)=p,$ and we have taken the complex structure on $\cp^1$ given by left-handed rotation around the outward normal by an angle of $\pi/2,$ it follows that $i\beta\cdot u(0)= -\alpha\cdot u(0).$ See Figure~\ref{fig:6}. Thus the kernel of the infinitesimal action of $\mathfrak{sl}(2,\mathbb{C})$ at $u(0)$ is spanned by $\xi_f$ and by $\alpha+i\beta$.
So, it follows from the argument of Smith \cite[Appendix~A.3]{smith2020monotone} that the following sections give a holomorphic frame of~$E:$
\[
\frac{\xi_f}{z}\cdot u, \qquad\frac{\alpha+i\beta}{z}\cdot u,\qquad
(\alpha-i\beta)\cdot u.
\]
Taking $\C$-linear combinations, we see that
\[
\frac{\xi_f}{z}\cdot u, \qquad \frac{(1+z)\alpha+i(1-z)\beta}{z}\cdot u,\qquad
\frac{(1+z)\beta-i(1-z)\alpha}{z}\cdot u,
\]
is also a holomorphic frame of $E.$
Since $\xi_f, \alpha, \beta$ is a basis of $\mathfrak{su}(2)$, the sections
\begin{equation*}
\xi_f\cdot u,\qquad \alpha\cdot u, \qquad \beta\cdot u,
\end{equation*}
give a frame of for $F.$
This frame is $\oo_\tr$ oriented and can be lifted to the double cover of the frame bundle of $F$ associated to $\s_\tr$ because of how we have chosen $\alpha,\beta$.

Let $\Psi: E \to \underline \C^3$ be the isomorphism given by
\[
    \frac{\xi_f}{z}\cdot u\mapsto \begin{pmatrix}
        1\\
        0\\
        0
    \end{pmatrix}, \quad
    \frac{(1+z)\alpha+i(1-z)\beta}{z}\cdot u\mapsto \begin{pmatrix}
        0\\
        1\\
        0
    \end{pmatrix}, \quad \frac{(1+z)\beta-i(1-z)\alpha}{z}\cdot u\mapsto \begin{pmatrix}
        0\\
        0\\
        1
    \end{pmatrix}.
\]
Since
\[
    \xi_f\cdot u\mapsto \begin{pmatrix}
        z\\
        0\\
        0
    \end{pmatrix},
\quad\alpha\cdot u \mapsto \begin{pmatrix}
        0\\
        \frac{z+1}{4}\\
        \frac{-i(1-z)}{4}
    \end{pmatrix},\quad
\beta\cdot u\mapsto \begin{pmatrix}
        0\\
        \frac{i(1-z)}{4}\\
        \frac{z+1}{4}
    \end{pmatrix},
\]
we see that $\Psi:(E,F) \to (\underline \C^3,F^{2} \oplus F^{1,1}).$
Recall the
choices of spin structures and orientations on $F^2$ and $F^{1,1}$ given in Section~\ref{choose_spin_structure}.
By Lemma~\ref{lm8.1jake} the frame
\[\begin{pmatrix}
        z\\
        0\\
        0
    \end{pmatrix},\quad \begin{pmatrix}
        0\\
        \frac{z+1}{4}\\
        \frac{-i(1-z)}{4}
    \end{pmatrix},\quad \begin{pmatrix}
        0\\
        \frac{i(1-z)}{4}\\
        \frac{z+1}{4}
    \end{pmatrix},\]
of  $F^2\oplus F^{1,1}$ is positively oriented and can be lifted to the spin double cover of the frame bundle.
Hence, the isomorphism $\Psi$ preserves orientation and spin structure.
Thus, by Lemma~\ref{orientation_on_ker}, we obtain an orientation preserving isomorphism
\[
T_{[u;w]}\mathcal{M}_{0,1}(2)\simeq (\ker \Bar\partial_{(E,F)} \oplus T_w D)/\mathfrak{psl}(2,\mathbb{R})\simeq(\ker\Bar{\partial}_2\oplus\ker\Bar{\partial}_{1,1}\oplus T_wD)/\mathfrak{psl}(2,\mathbb{R}),
\]
where $w\in\mathrm{int}D$. By Lemma~\ref{lemma.psl} this gives an orientation-reversing isomorphism
\[
\Psi_*: T_{[u;w]}\mathcal{M}_{0,1}(2)\overset{\sim}{\longrightarrow}\ker\Bar{\partial}_{1,1}\oplus T_wD.
\]

It suffices to show that
\[ {devi_1}_{[u;w]}: T_{[u;w]}\mathcal{M}_{0,1}(2)\rightarrow  T_{u(w)}\cp^3\]
reverses orientation when $w=0$.
Denote by $f_D:T_0D\rightarrow \C$ the canonical isomorphism, and recall that $f_{1,1}: \ker\bar{\partial}_{1,1}\rightarrow \underline{\C}^2_0$ is the evaluation map at zero.
In the following commutative diagram of isomorphisms, the left vertical arrow reverses orientation, and the right vertical arrow preserves orientation.
\[\begin{tikzcd}[ampersand replacement=\&]
	T_{[u:0]}\mathcal{M}_{0,1}(2)  \&  T_{u(0)}\cp^3  \\
	\ker\bar{\partial}_{1,1}\oplus T_0D \& \C^3
	\arrow["{devi_1}_{[u;0]}", from=1-1, to=1-2]
	\arrow["\Psi|_{E_0}",  , from=1-2, to=2-2]
	\arrow["\Psi_*", from=1-1, to=2-1]
	\arrow["f_{1,1}\oplus f_D", from=2-1, to=2-2]
\end{tikzcd}\]
So, it suffices to show that $f_{1,1}\oplus f_D$ preserves orientation.
Since $f_D:T_0D\rightarrow \C$ preserves orientation, and by Lemma~\ref{lm11bundle} the map $f_{1,1}$ preserves orientation, it follows that
$f_{1,1}\oplus f_D$ preserves orientation as desired.
\end{proof}

\section{Recursions}\label{section5}
\begin{proof}[Proof of Theorem \ref{Theorem 1}]
Recall $\Delta_i=[\omega^i]\in H^*(X;\R),$ for $i=0,\ldots,3.$ So, by the definition of $g^{ij}$ given in Section~\ref{subsection_g}, it follows that
 $g^{ij}=\delta_{i,3-j}.$ In order to derive recursion \ref{recusion a}, let $I = \{j_2,\ldots,j_l\}.$ Apply $\partial_s^{k-1} \partial_{t_I}$ = $\partial_s^{k-1} \partial_{t_{j_2}} \ldots \partial_{t_{j_l}}$ to equation~\eqref{eq16} with $v={j_1-1}$, $w=1$. We consider the coefficients of $T^{\beta}$ and evaluate at $s=t_j=0$. Using the closed zero axiom Proposition~\ref{prop:closed_axioms}\ref{zerooo}, we single out instances of $\ogw_{\beta, k}(\Gamma_{j_1},\Gamma_I)$ and obtain:
\begin{align*}
    &[T^{\beta}]\sum_{i=0}^3(\partial_s^{k-1} \partial_{t_I}(\partial_s \partial_{t_{3-i}} \Om \cdot \partial_{t_i} \partial_{t_w} \partial_{t_v} \Phi )|_{s=t_j=0})=\\
    &\qquad = \ogw_{\beta, k} (\Gamma_{j_1},\ldots, \Gamma_{j_l})+ \sum_{\substack{\varpi(\hat{\beta}) + \beta_1 = \beta \\  I_1 \sqcup I_2 = I \\ \beta_1 \neq \beta}} \sum_{i = 0}^{3}
    \gw_{\hat{\beta}}(\Delta_1, \Delta_{j_1-1}, \Delta_{I_1}, \Delta_i)
    \ogw_{\beta_1, k} (\Gamma_{3-i}, \Gamma_{I_2}), \\
    &[T^{\beta}](\partial_s^{k-1} \partial_{t_I}(\partial_s^2 \Om \cdot \partial_{t_{w}} \partial_{t_v} \Om )|_{s=t_j=0})= \\
    &\qquad =\sum_{\substack{\beta_1 + \beta_2 = \beta \\
    k_1 + k_2 = k-1 \\
    I_1 \sqcup I_2 = I }}
    \binom{k-1}{k_1}
    \ogw_{\beta_1, k_1}(\Gamma_{j_1-1}, \Gamma_1, \Gamma_{I_1})
    \ogw_{\beta_2, k_2+2} (\Gamma_{I_2}), \\
    &[T^{\beta}](\partial_s^{k-1} \partial_{t_I}(\partial_{s} \partial_{t_w} \Om \cdot \partial_{t_v} \partial_s \Om )|_{s=t_j=0})= \\
    & \qquad =\sum_{\substack{\beta_1 + \beta_2 = \beta \\
    k_1 + k_2 = k-1 \\
    I_1 \sqcup I_2 = I }}
    \binom{k-1}{k_1}
    \ogw_{\beta_1, k_1+1}(\Gamma_{j_1-1}, \Gamma_{I_1})
    \ogw_{\beta_2, k_2+1} (\Gamma_1, \Gamma_{I_2}).
\end{align*}
Substituting the expressions in~\eqref{eq16} gives the required recursion.

In order to derive recursion \ref{recursion b}, let $I = \{j_1,\ldots,j_l\}.$ Apply $\partial_s^{k-2} \partial_{t_I}$ = $\partial_s^{k-2} \partial_{j_1} \ldots \partial_{j_l}$ to equation \eqref{eq16} with $v=2$, $w=2$. We consider the coefficients of $T^{\beta+2}$ and evaluate at $s=t_j=0$. We obtain:
\begin{align*}
    &[T^{\beta+2}]\sum_{i=0}^3(\partial_s^{k-2} \partial_{t_I}(\partial_s\partial_{t_{3-i}}  \Om\cdot\partial_{t_i} \partial_{t_w} \partial_{t_v} \Phi)|_{s=t_j=0})= \\
    &\qquad =\sum_{\substack{\varpi(\hat{\beta}) + \beta_1 = \beta+2 \\  I_1 \sqcup I_2 = I }} \sum_{i = 0}^{3}
    \gw_{\hat{\beta}}(\Delta_2, \Delta_2, \Delta_{I_1}, \Delta_i)
    \ogw_{\beta_1, k-1} (\Gamma_{3-i}, \Gamma_{I_2}), \\
    &[T^{\beta+2}](\partial_s^{k-2} \partial_{t_I}(\partial_s^2 \Om \cdot \partial_{t_w} \partial_{t_v} \Om )|_{s=t_j=0})= \\
    &\qquad =\ogw_{2, 0}(\Gamma_2, \Gamma_2)
    \ogw_{\beta, k} (\Gamma_{j_1},\ldots,\Gamma_{j_l})\\
     & \ \ \ \ \ \ \ \ \ \ \ \ \ \ \ \ \ \
     +\sum_{\substack{\beta_1 + \beta_2 = \beta+2 \\
    k_1 + k_2 = k-2 \\ (\beta_1,k_1)\ne(2,0)\\
    I_1 \sqcup I_2 = I }}
    \binom{k-2}{k_1}
    \ogw_{\beta_1, k_1}(\Gamma_2, \Gamma_2, \Gamma_{I_1})
    \ogw_{\beta_2, k_2+2} (\Gamma_{I_2}), \\
    &[T^{\beta+2}](\partial_s^{k-2} \partial_{t_I}(\partial_{s} \partial_{t_w} \Om \cdot \partial_{t_v} \partial_s \Om) |_{s=t_j=0})= \\
    &\qquad  =\sum_{\substack{\beta_1 + \beta_2 = \beta+2 \\
    k_1 + k_2 = k-2 \\
    I_1 \sqcup I_2 = I }}
    \binom{k-2}{k_1}
    \ogw_{\beta_1, k_1+1}(\Gamma_2, \Gamma_{I_1})
    \ogw_{\beta_2, k_2+1} (\Gamma_2, \Gamma_{I_1}).
\end{align*}
Substituting the expressions in~\eqref{eq16} gives the required recursion.

Recursion \ref{recursion c} is derived in the same way as recursion (a) in Theorem 10 from  \cite{RelativeQuantumCohomology}.
\end{proof}
\begin{lm}\label{lm:ogw2022}
$\ogw_{2,0}(\Gamma_2, \Gamma_2)=\frac{35}{64}.$
\end{lm}
\begin{proof}
In Theorems ~\ref{thm:ogw11}-\ref{thm:ogw203}  we computed
\[\ogw_{1,1}=-3,\quad \ogw_{1,0}(\Gamma_2)=\frac{1}{4}, \quad \ogw_{2,0}(\Gamma_3)=1.\]
By recursion~\ref{recursion c} of Theorem~\ref{Theorem 1}, we get
\[
    \ogw_{2,0}(\Gamma_2, \Gamma_2)=\ogw_{2,0}(\Gamma_1, \Gamma_3)-\ogw_{1,1}(\Gamma_1)\cdot \ogw_{1,0}(\Gamma_1, \Gamma_2).
\]
By the open divisor axiom Proposition~\ref{prop:open_axioms}\ref{ax_divisor}, we get
\[
\ogw_{2,0}(\Gamma_2, \Gamma_2)=
    \frac{1}{2} + \frac{3}{4}\cdot \frac{1}{16}
    =\frac{35}{64}.
\]
\end{proof}

\begin{proof}[Proof of Theorem~\ref{thm:iv}]
This follows from Theorems ~\ref{thm:ogw11}-\ref{thm:ogw203} and Lemma~\ref{lm:ogw2022}.
\end{proof}
In order to prove Corollary~\ref{Corollary 1.2} we will need the following lemma.
\begin{lm}\label{lm:beta1}
Assume $|\Gamma_{i_j}|\ge 4.$ If $\ogw_{1,k}(\Gamma_{i_1},\ldots,\Gamma_{i_l})\ne0$, then $(k,l)=(1,0)$ or $(k,l)=(0,1)$ and $|\Gamma_{i_1}|=4$.
\end{lm}

\begin{proof}
By the open degree axiom Proposition~\ref{prop:open_axioms}\ref{ax_degree} and Lemma~\ref{maslov_index_lemma}
\begin{align*}
 & \beta=1 \implies 2+2l=2k+\sum_{j=1}^l |\Gamma_{i_j}| \\
 &\implies 2(k-1)+\sum_{j=1}^l (|\Gamma_{i_j}|-2)=0.
\end{align*}
Both summands are positive when $k>1.$ Hence,
equality can hold only if $(k,l)=(1,0),$ or
$(k,l)=(0,1)$ and $|\Gamma_{i_1}|=4$.

\end{proof}

\begin{proof}[Proof of Corollary \ref{Corollary 1.2}]
By Theorem \ref{wall_crossing} invariants with interior constraints in $\Gamma_{\diamond}$ are computable in terms of invariants with interior constraints of the form $\Gamma_j=[\omega^j]$. Furthermore, by the open unit and divisor axioms Proposition~\ref{prop:open_axioms}\ref{ax_unit},\ref{ax_divisor}, we may assume that $|\Gamma_j| \geqslant  4$. Finally, assume for convenience that interior constraints are written in ascending degree order.
It follows from the open degree axiom Proposition~\ref{prop:open_axioms}\ref{ax_degree} that for any $\beta,$ there are only finitely many values of $k,l,$ for which there may be nonzero invariants with constraints of the above type.
Hence, we give a process for computing $\ogw_{\beta, k}(\Gamma_{i_1}, \ldots , \Gamma_{i_l})$ which is inductive on $(\beta, k, l, i_1)$ with respect to the lexicographical order on $\mathbb{Z}^{\oplus 4}_{\geqslant 0}$.

Consider a triple $(\beta, k, l)$ with $ k + l < 2 $.
If $\beta=0$, by the open zero axiom Proposition~\ref{prop:open_axioms}\ref{ax_zero} all invariants vanish. For $\beta=1,2$ all possible values have been computed explicitly in Theorem \ref{thm:iv}. Indeed, if $\beta=1$ this follows from Lemma~\ref{lm:beta1}.
If $\beta=2,$ we have
\begin{align*}
 & \beta=2 \implies 4+2l=2k+\sum_{j=1}^l |\Gamma_{i_j}| \\
 &\implies 2(k-2)+\sum_{j=1}^l (|\Gamma_{i_j}|-2)=0.
\end{align*}
Equality can only hold if $k=0$, $l=1$ and $|\Gamma_{i_1}|=6$. Similarly, the open degree axiom implies that for $\beta \geqslant 3$ the only invariants that don't vanish have $k+l \geqslant 2$.

Consider a triple $(\beta, k, l)$ with $ k + l \geqslant 2 $.
If $l \geqslant 2$, by Theorem \ref{Theorem 1} \ref{recursion c} and the open divisor and zero axiom we can express the invariant $\ogw_{\beta, k}(\Gamma_{i_1}, \ldots , \Gamma_{i_l})$ as a combination of invariants $\ogw_{\beta',k'}(\Gamma_{j_1},\ldots,\Gamma_{j_{l'}})$ that either have
$\beta' < \beta,$ or $\beta' = \beta, k' = k$ and $l'< l,$ or $\beta' = \beta, k' = k, l'=l$ and $ j_1 < i_1$. Thus, the invariant is reduced to invariants with data of smaller lexicographical order, known by induction.

Note that by Lemma~\ref{lm:ogw2022} we have $\ogw_{2,0}(\Gamma_2,\Gamma_2)\ne0$. So, if $k \geqslant 2$, by Theorem \ref{Theorem 1} \ref{recursion b} we can express the required invariant in terms of invariants that are either of smaller degree or have equal degree and fewer boundary marked points. Indeed, in formula \ref{recursion b}, the closed zero axiom Proposition~\ref{prop:closed_axioms}\ref{ax_degree_GW} and the open zero axiom imply that all the products involving invariants with degree $ \beta +2$ vanish.
By Lemma~\ref{lm:relhom} the map $\varpi$ is given by multiplication by~$4.$ This and Lemma~\ref{lm:beta1} imply that products involving invariants with degree $\beta +1 $ do not occur.

If $k \geqslant 1$ and $ l\geqslant$ 1, by Theorem \ref{Theorem 1} \ref{recusion a}, the open zero axiom and the closed zero axiom, we can express the required invariant in terms of invariants that are of smaller degree.
\end{proof}

\section{Small relative quantum cohomology}\label{section7}
In this section we compute the small relative quantum cohomology of $(\x,L_\tr).$
\begin{lm}\label{ogw22}
$\ogw_{2,2}=-\frac{5}{4}.$
\end{lm}
\begin{proof}
By recursion~\ref{recusion a} of Theorem~\ref{Theorem 1}, we get
\[\ogw_{2,1}(\Gamma_2)=-(\ogw_{1,1}(\Gamma_1))^2.\]
In Theorem~\ref{thm:ogw11} we computed $\ogw_{1,1}=-3.$
Thus, by the open divisor axiom Proposition~\ref{prop:open_axioms}\ref{ax_divisor}, we get $\ogw_{2,1}(\Gamma_2)=-\frac{9}{16}.$
By recursion~\ref{recursion b} of Theorem~\ref{Theorem 1}, we get
\[
\ogw_{2,2}\cdot\ogw_{2,0}(\Gamma_2,\Gamma_2)=\gw_1(\Delta_2,\Delta_2, \Delta_3)\cdot\ogw_{0,1}(\Gamma_0)+\ogw_{2,1}((\Gamma_2))^2.
\]
Lemma~\ref{lm:ogw2022} gives $\ogw_{2,0}(\Gamma_2,\Gamma_2)=\frac{35}{64},$ and the number of lines through a point and two lines is $\gw_1(\Delta_2,\Delta_2, \Delta_3)=1.$
So, by the open unit axiom Proposition~\ref{prop:open_axioms}\ref{ax_unit}  we get
\[
\ogw_{2,2}= \frac{-1 +\frac{81}{256}}{\frac{35}{64}} = -\frac{5}{4}.
\]
\end{proof}
Recall the definition of $\Gamma_\diamond$ from Section~\ref{subsection_wallcrossing}. By Lemma~\ref{lm:relhom} we have $H_2(\cp^3, L_\triangle;\mathbb{Z})=\Z,$ and the map
\[\varpi:H_2(\cp^3;\mathbb{Z})\longrightarrow H_2(\cp^3, L_\triangle;\mathbb{Z}),\]
is given by $m\mapsto 4m.$
Hence, since $g^{ij}=\delta_{i,3-j}$, it follows that the relative small quantum product~\eqref{relqua} for $(X,L)=(\x, L_\triangle)$ is given by
\begin{multline*}
\mem(\Gamma_u, \Gamma_v)=\sum_{\substack{0\le l\le 3\\  \beta\in H_2(\x;\Z)}}T^{\varpi(\beta)}\cdot\gw_\beta(\Delta_u, \Delta_v,\Delta_l)\cdot \Gamma_{3-l}+\\
+\sum_{\beta\in H_2(\x,L_\triangle;\Z)}T^\beta \cdot\ogw_{\beta,0}(\Gamma_u, \Gamma_v)\cdot\Gamma_\diamond.
\end{multline*}

\begin{proof}[Proof of Theorem~\ref{relativequantumpro}]
We claim that $QH^*(\x, L_\triangle)$ is generated as a ring by $1,T, \Gamma_1, \Gamma_\diamond.$ Indeed,
by the open degree axiom Proposition~\ref{prop:open_axioms} \ref{ax_degree} and the closed degree axiom Proposition\ref{prop:closed_axioms}\ref{ax_degree_GW} we obtain
\[\mem(\Gamma_1, \Gamma_1)=\gw_0(\Delta_1, \Delta_1, \Delta_1)\cdot\Gamma_2.\]
\[
    \mem(\Gamma_1, \Gamma_2)= \gw_0(\Delta_1, \Delta_2, \Delta_0)\cdot \Gamma_3+ T\cdot\ogw_{1,0}(\Gamma_1, \Gamma_2)\cdot \Gamma_\diamond.
\]
So, by the closed zero axiom Proposition~\ref{prop:closed_axioms}~\ref{zerooo}, the open divisor axiom Proposition~\ref{prop:open_axioms}~\ref{ax_divisor} and Theorem~\ref{thm:ogw102}, we get
\[
\mem(\Gamma_1, \Gamma_1)=\Gamma_2,
\]
and
\[
\mem(\Gamma_1, \Gamma_2) = \Gamma_3+\frac{T}{16}\cdot \Gamma_\diamond.
\]
Consider the ring
$\R[x,y][[q^{1/4}]]/I$ with
\[I=( x^4-q-\frac{1}{2}q^{1/2}y-\frac{3}{64}q^{1/4}y,\ y^2+\frac{5}{4}q^{1/2}y,\  xy-\frac{3}{4}q^{1/4}y ).\]
Let
\[\phi:\R[x,y][[q^{1/4}]]/I \rightarrow QH^*(\x,L_\tr) \]
be the ring homomorphism defined by \[1\mapsto 1,\quad q^{1/4}\mapsto T,\quad x\mapsto \Gamma_1, \quad y\mapsto \Gamma_\diamond.\]
Hence,
\begin{equation}\label{eq:imphi}
\phi(x)^2=\Gamma_2,\quad \phi(x)^3-\frac{\phi(q^\frac{1}{4})}{16}\phi(y)=\Gamma_3.
\end{equation}
In order to show the map $\phi$ is well defined, it suffices to show that the generators of the ideal $I$ are sent to $0.$ Using the closed degree axiom, the wall crossing formula Theorem~\ref{wall_crossing} and the open degree axiom, we obtain
\begin{align*}
\mem(\Gamma_1,\mem(\Gamma_1, \Gamma_2))
    &=\mem(\Gamma_1, \Gamma_3) + \mem(\Gamma_1, \frac{T}{16} \Gamma_\diamond)\\
    &= T^4\cdot\gw_1(\Delta_1, \Delta_3, \Delta_3)\cdot \Gamma_0+ T^2\cdot\ogw_{2,0}(\Gamma_1, \Gamma_3)\cdot \Gamma_\diamond-\frac{T}{16}\ogw_{1,1}(\Gamma_1)\cdot \Gamma_\diamond.
    \end{align*}
The number of lines through two points and a plane is $\gw_1(\Delta_1,\Delta_3,\Delta_3)=1$.  Hence, by Theorem~\ref{thm:ogw203}, Theorem~\ref{thm:ogw11} and the open divisor axiom, we get
\begin{align*}
     \phi(x)^4&=\phi(q^{1/4})^4+\frac{1}{2}\phi(q^{1/4})^2\phi(y)+\frac{3}{64}\phi(q^{1/4})\phi(y).
    %&= \phi(q)-\frac{35}{64}\phi(q^{1/2})\phi(y).
\end{align*}
So,
\[
    \phi(x^4-q-\frac{1}{2}q^{1/2}y-\frac{3}{64}q^{1/4}y)=0.
\]
By the wall crossing formula Theorem~\ref{wall_crossing} and the open degree axiom, we get
\begin{align*}
    \mem(\Gamma_\diamond, \Gamma_\diamond)&=\sum_{\beta\in H_2(\x,L_\triangle;\Z)}T^\beta \cdot\ogw_{\beta,0}(\Gamma_\diamond, \Gamma_\diamond)\cdot\Gamma_\diamond\\
    &=\sum_{\beta\in H_2(\x,L_\triangle;\Z)}T^\beta\cdot \ogw_{\beta,2}\cdot\Gamma_\diamond\\
    &=T^2\cdot\ogw_{2,2}\cdot\Gamma_\diamond.
\end{align*}
Hence, by Lemma~\ref{ogw22} we get $\phi(y^2+\frac{5}{4}q^{1/2}y)=0.$ Using the wall crossing formula Theorem~\ref{wall_crossing}, the open degree and divisor axioms, and Theorem~\ref{thm:ogw11}, we get
\begin{align*}
    \mem(\Gamma_1, \Gamma_\diamond) &=\sum_{\beta\in H_2(\x,L_\triangle;\Z)}T^\beta\cdot \ogw_{\beta,0}(\Gamma_1,\Gamma_\diamond)\cdot\Gamma_\diamond\\
    &=-\sum_{\beta\in H_2(\x,L_\triangle;\Z)}T^\beta\cdot \ogw_{\beta,1}(\Gamma_1)\cdot\Gamma_\diamond\\
    &=-T\cdot\ogw_{1,1}(\Gamma_1)\cdot\Gamma_\diamond\\
    &=\frac{3}{4}T\cdot\Gamma_\diamond.
\end{align*}
So, $\phi(xy-\frac{3}{4}q^{1/4}y)=0.$ Thus, the map $\phi$ is well defined. It is surjective because its image generates $QH^*(\x, L_\triangle)$ by equation~\eqref{eq:imphi}.

Abbreviate
\[N:=QH^*(X,L) \quad M:= \R[x,y][[q^{1/4}]]/I, \quad m:=(q^{1/4})\triangleleft \R[[q^{1/4}]]. \]
We think of $M$ and $N$ as modules over the local ring $\R[[q^{1/4}]].$ Consider the induced map
  \[\Bar{\phi}:M/mM\rightarrow N/mN.\]
It follows from the definition of the ideal $I$ that $M/mM$ is a real vector space with basis $1,y,x,x^2,x^3$.
Since $\Bar{\phi}$ is surjective and
\[\dim M/mM=\dim N/mN=\dim \widehat{H}^*(\x,L_\tr;\R),\]
it follows that $\Bar{\phi}$ is injective.
Since $\ker\phi/m\ker\phi\subset \ker\Bar{\phi}=0,$
it follows that $m\ker\phi=\ker\phi.$
Since $\R[[q^{1/4}]]$ is a local ring with maximal ideal $m$, Nakayama's lemma gives $\ker\phi=0.$ Therefore, $\phi$ is an isomorphism.
\end{proof}

\section{Dependence of invariants on left inverse}\label{section8}
In this section we review in greater detail the dependence of the invariants $\ogw_{\beta,k}$ on the map $P_\R.$ We quantify this dependence in Proposition~\ref{prop:general_map_p}, from which we obtain Proposition~\ref{prop_map_p} as a special case. Finally, we prove Corollary~\ref{cor:map_p}.

In the present section, we continue in the setting of Section~\ref{subsection_ogw}. In particular, $L\subset X$ is a connected spin Lagrangian submanifold with $H^*(L;\R) \simeq H^*(S^n;\R)$ and $[L] = 0 \in H_n(X;\R).$ Recall the definition of the map $P_\R: \widehat{H}^{n+1}(X,L;\R) \to H^n(L;\R)$ from Section~\ref{subsection_ogw}.
In this section, we extend $P_\R$ to a map $P_\R: \widehat{H}^*(X,L;\R) \to H^n(L;\R)$ by setting it to zero outside $\widehat{H}^{n+1}(X,L;\R)$. We use Poincar\'e duality to identify $H^n(L;\R)\simeq \R.$ Let $\rho: \widehat{H}^{*}(X,L;\R)\rightarrow H^{*}(X;\R)$ denote the natural map. Consider $P_\R,P_\R'$ with associated invariants $\ogw_{\beta,k}$ and $\ogw\,'\!\!_{\beta,k}$ respectively. The long exact sequence of the pair $(X,L)$ implies there exists a unique map
\[\mathfrak{p}_\R:H^{*}(X;\R)\rightarrow \R\]
such that $\mathfrak{p}_\R\circ\rho=P_\R-P_\R'.$
Let $\Delta_0,\ldots,\Delta_N\in H^*(X;\R)$ be a basis.
Recall that
\[
g_{ij} = \int_X \Delta_i \smile \Delta_j, \qquad i, j = 0,\ldots,N,
\]
and $g^{ij}$ denotes the inverse matrix.
\begin{prop}\label{prop:general_map_p}
Let $A_1,..,A_l\in\widehat{H}^*(X,L;\R)$.
    \[\ogw_{\beta,0}(A_1,\ldots,A_l)-\ogw\,'\!\!_{\beta,0}(A_1,\ldots,A_l)=\sum_{\substack{\Tilde{\beta}\in H_2(X;\Z)\\ \varpi(\Tilde{\beta})=\beta}} g^{ij}\gw_{\Tilde{\beta}}(\Delta_i,\rho(A_1),\ldots,\rho(A_l))\mathfrak{p}_\R(\Delta_j).\]
\end{prop}
\begin{rem}
Proposition~\ref{prop:general_map_p} continues to hold in the more general setting of~\cite{RelativeQuantumCohomology} without change. We have formulated it only for the case of $L$ spin and $H^*(L;\R) \simeq H^*(S^n;\R)$ to streamline the exposition since the Chiang Lagrangian satisfies these assumptions.
\end{rem}
In order to prove Proposition~\ref{prop:general_map_p}, we will need a number of results from~\cite{RelativeQuantumCohomology}, which we now summarize. The proof appears toward the end of this section. Recall the definition of the rings $Q_W$ and $R_W$ from Section~\ref{subsection_novikovrings}.
Define \[\mathfrak{i}: A^*(X;Q_W)\rightarrow R_W[-n],\]
by
\[\eta\mapsto (-1)^{n+|\eta|}\int_Li^*\eta.\]
This is a map of complexes when $R_W$ is equipped with the trivial differential. The \textbf{cone} $C(\mathfrak{i})$ is the complex with underlying graded $Q_W$ module $A^*(X;Q_W)\oplus R_W[-n-1]$ and differential
\[d_C(\eta,\xi):= (d\eta, \mathfrak{i}(\eta)-d\xi)=(d\eta, \mathfrak{i}(\eta)).\]
Note that if $[L]=0\in H_n(X;\R),$ then $\coker\mathfrak{i}\simeq R_W[-n-1]$.
Thus, we consider the following commutative diagram with exact rows and columns, which is taken from Section 4.4 in \cite{RelativeQuantumCohomology} except that $\coker\mathfrak{i}$ is replaced with $R_W[-n-1].$

\begin{equation}\label{commutative_diagram}
\xymatrix{
        & 0 & 0 \\
0 \ar[r] & R_W/Q_W \ar[u] \ar[r]^\sim & R_W/Q_W \ar[u] \ar[r] & 0 \\
0 \ar[r] & R_W \ar[u]^{ \Bar{q}} \ar[r]^(.45){\Bar{x}} & H^*(C(\mathfrak{i})) \ar[u]^{\overline P} \ar[r]^(.55){\pi} & H^*(X;Q_W) \ar[u] \ar[r] & 0 \\
0 \ar[r] & Q_W \ar[r]^(.4){\bar{y}_Q}\ar[u]^{\bar{a} } & \widehat{H}^*(X,L;Q_W) \ar[r]^(.6){\rho_Q}\ar[u]^a & H^*(X;Q_W) \ar[u]^\wr \ar[r] & 0 \\
& 0 \ar[u] & 0 \ar[u] & 0 \ar[u]
}
\end{equation}
Here, $\bar a : Q_W \to R_W$ is the inclusion and $\bar q : R_W \to R_W/Q_W$ is the quotient map. Let \[P:H^*(C(\mathfrak{i}))\rightarrow R_W,\] be a left inverse to the map $\Bar{x}$ from the diagram~\eqref{commutative_diagram} satisfying the following two conditions. The first condition is that
\begin{equation}\label{first_condition}
    \Bar{q}\circ P=\overline{P}.
\end{equation}
This condition and the exactness of the diagram~\eqref{commutative_diagram} imply that there exists a unique $P_Q:\widehat{H}^*(X,L;Q_W)\rightarrow Q_W$ such that the following diagram commutes
\begin{equation}\label{eq:PQ}
\xymatrix{
R_W  & H^*(C(\mathfrak{i})) \ar[l]_P \\
 Q_W \ar[u]^{\Bar{a}}  & \widehat{H}^*(X,L;Q_W) \ar[u]^a \ar[l]_(.55){P_Q} \\
}
\end{equation}
The second condition is that there exists $P_\R : \widehat{H}^*(X,L;\R) \rightarrow\R$, such that
\begin{equation}\label{eq:PR}
P_Q = P_\R \otimes \id_Q.
\end{equation}
The following is Lemma 4.10 from \cite{RelativeQuantumCohomology}.
\begin{lm}
    $P_Q\circ \Bar{y}_Q=\id$.
\end{lm}
The following is Lemma 4.11 from \cite{RelativeQuantumCohomology}.
\begin{lm}\label{lm:PRP}
Let $l:\widehat{H}^*(X,L;\R)\rightarrow \R$ satisfy $l\circ y=\id.$ There exists a unique choice of $P:H^*(C(\mathfrak{i}))\rightarrow R_W$ satisfying conditions~\eqref{first_condition} and~\eqref{eq:PR} such that $l=P_\R$. Moreover, $\ker P= a(\ker P_Q)$.
\end{lm}
Let $\Psi \in C(\mathfrak{i})$ be the relative potential defined in Section 1.3.1 in \cite{RelativeQuantumCohomology}. The definition of $\overline{\Omega}$ given in Section 1.3.3 in~\cite{RelativeQuantumCohomology} is
\begin{equation}\label{P_psi}
    \overline{\Omega}= P\Psi.
\end{equation}
This together with~\eqref{eq:esup} and Lemma~\ref{lm:PRP} makes precise the dependence of the invariants $\ogw_{\beta,k}$ on $P_\R.$ To quantify this dependence, we recall another lemma.

Denote by $\rho^*: Q_U\rightarrow Q_W$ the inclusion map. Recall the map $\pi$ from diagram~\eqref{commutative_diagram}.
The following is Lemma 5.9 from \cite{RelativeQuantumCohomology}
\begin{lm}\label{pi_psi}
    $\pi(\Psi)=\rho^*(\nabla \Phi).$
\end{lm}
\begin{proof}[Proof of Proposition~\ref{prop:general_map_p}]
Apply Lemma~\ref{lm:PRP} to obtain maps $P,P,'$ corresponding to $P_\R,P_\R',$ that are left inverses to $\Bar{x}$ and satisfy conditions~\eqref{first_condition} and~\eqref{eq:PR}. Denote by $P_Q, P_Q'$ the corresponding maps from diagram \eqref{eq:PQ}. Since
\[
(P-P')\circ \Bar{x}=0, \quad (P_Q-P_Q')\circ\Bar{y}_Q =0,
\]
the maps $P-P'$, $P_Q-P_Q',$ factor through $H^*(X;Q_W)$. Consequently, there exist unique maps $\mathfrak{p},\mathfrak{p}_Q,$ such that the following diagram commutes.
\begin{equation}\label{diagram_curve}
\begin{tikzcd}
	R_W && H^*(C(\mathfrak{i})) && H^*(X;Q_W) \\
	\\
	Q_W && \widehat{H}^*(X,L;Q_W) && H^*(X;Q_W)
	\arrow["{P_Q-P_Q'}"', from=3-3, to=3-1]
	\arrow["\rho_Q", from=3-3, to=3-5]
	\arrow["\mathfrak{p}_Q", shift left=1, curve={height=-30pt}, from=3-5, to=3-1]
	\arrow["a", from=3-3, to=1-3]
	\arrow["\Bar{a}", from=3-1, to=1-1]
	\arrow["\wr", from=3-5, to=1-5]
	\arrow["{P-P'}"', from=1-3, to=1-1]
	\arrow["\mathfrak{p}"', shift right=1, curve={height=30pt}, from=1-5, to=1-1]
	\arrow["\pi", from=1-3, to=1-5]
\end{tikzcd}
\end{equation}
Indeed, since \[(P_Q-P_Q')\circ \Bar{y}_Q=0, \quad\rho_Q\circ \Bar{y}_Q=0,\] then
$(P_Q-P_Q')(\ker\rho_Q)=0.$ By diagram~\eqref{commutative_diagram} $\rho_Q$ is surjective. Thus, for every $\Tilde{\eta}\in \rho_Q^{-1}(\eta)$ we define $\mathfrak{p}_Q(\eta)= (P_Q-P_Q')(\Tilde{\eta})$. Hence, $\mathfrak{p}_Q$ is determined by $P_Q$ and $P_Q'.$ A similar argument applies for $\mathfrak{p}.$
Since
\begin{align*}
(\mathfrak{p}_\R\otimes\id_Q)\circ \rho_Q&=(\mathfrak{p}_\R\otimes\id_Q)\circ (\rho\otimes \id_Q)\\
&= (\mathfrak{p}_\R\circ\rho)\otimes\id_Q\\
&=(P_\R-P_\R')\otimes \id_Q\\
&=P_Q-P_Q',
\end{align*}
it follows by the uniqueness of $\mathfrak{p}_Q$ that
\[
\mathfrak{p}_Q=\mathfrak{p}_\R\otimes \id_Q.
\]
Let $\ogw,$ $ \ogw'$ be the Gromov-Witten invariants correspond to $P$ and $P'$ respectively, and let $\overline{\Omega}$, $\overline{\Omega}'$ be the corresponding superpotentials.

From multilinearity of $\ogw_{\beta,k}$ is suffices to prove this proposition for $A_i=\Gamma_i.$
By~\eqref{P_psi}, diagram~\eqref{diagram_curve}, and Lemma~\ref{pi_psi} we get
\begin{align*}
    \overline{\Omega}-\overline{\Omega}'&=(P-P')\Psi\\
    &=(\mathfrak{p}\circ \pi)\Psi\\
    &=\mathfrak{p}(\rho^*(\nabla \Phi))\\
    &=\Bar{a}\circ \mathfrak{p}_Q(\rho^*(\nabla \Phi)).
\end{align*}
Recall that $[T^\beta]$ denotes the coefficient of $T^\beta.$ So,
\[[T^\beta]\partial_{t_{i_1}}\cdots\partial_{t_{i_l}}(\overline{\Omega}-\overline{\Omega}')|_{s=t_i=0}=[T^\beta]\partial_{t_{i_1}}\cdots\partial_{t_{i_l}}(\Bar{a}\circ \mathfrak{p}_Q(\rho^*(\nabla \Phi)))|_{s=t_i=0}.\]
By~\eqref{eq:esup} we have
\[[T^\beta]\partial_{t_{i_1}}\cdots\partial_{t_{i_l}}(\overline{\Omega}-\overline{\Omega}')|_{s=t_i=0}=\ogw_{\beta,0}(\Gamma_{t_{i_1}},\ldots,\Gamma_{t_{i_l}})-\ogw\,'\!\!_{\beta,0}(\Gamma_{i_1},\ldots,\Gamma_{i_l}).\]
Since $\mathfrak{p}_Q=\mathfrak{p}_\R\otimes \id_Q$ and $\nabla\Phi=g^{ij}\Delta_j\partial_i\Phi,$ it follows that
\begin{align*}
    [T^\beta]\partial_{t_{i_1}}\cdots\partial_{t_{i_l}}(\Bar{a}\circ \mathfrak{p}_Q(\rho^*(\nabla \Phi)))|_{s=t_i=0}&=[T^\beta]\partial_{t_{i_1}}\cdots\partial_{t_{i_l}}(\Bar{a}\circ \mathfrak{p}_Q(\rho^*(g^{ij}\Delta_j\partial_i\Phi)))|_{s=t_i=0}\\
    &=[T^\beta]\partial_{t_{i_1}}\cdots\partial_{t_{i_l}}( \mathfrak{p}_\R(\Delta_j)g^{ij}\partial_i\Phi)|_{s=t_i=0}\\
    &=\sum_{\substack{\Tilde{\beta}\in H_2(X;\Z)\\ \varpi(\Tilde{\beta})=\beta}} g^{ij}\gw_{\Tilde{\beta}}(\Delta_i,\Delta_{i_1},\ldots,\Delta_{i_l})\mathfrak{p}_\R(\Delta_j),
\end{align*}
 which completes the proof.
\end{proof}
\begin{proof}[Proof of Proposition~\ref{prop_map_p}]Lemma~\ref{lm:relhom} asserts that $\varpi$ is given by multiplication by $4$. Recall the choice of the basis $\Gamma_i \in \widehat H^*(\cp^3,L_\tr;\R)$ from Section~\ref{section1}. Since $P_\R(\Gamma_i)\ne0$ only if $i=2,$ it follows that $\mathfrak{p}_\R(\Delta_i)\ne0$ only if $i=2.$
Hence, since $g^{ij}=\delta_{i,3-j}$, it follows by Proposition~\ref{prop:general_map_p}, Lemma~\ref{divisor_for_chiang} and the closed divisor axiom Proposition~\ref{prop:closed_axioms}\ref{closed_divisor_ax}, that
\[
\ogw_{\beta,0}(\Gamma_{i_1},\ldots,\Gamma_{i_l})-\ogw\,'\!\!_{\beta,0}(\Gamma_{i_1},\ldots,\Gamma_{i_l})=\frac{\beta}{4}\gw_{\frac{\beta}{4}}(\Delta_{i_1},\ldots,\Delta_{i_l})\mathfrak{p}_\R(\Delta_2).
\]
\end{proof}
\begin{proof}[Proof of Corollary~\ref{cor:map_p}]
Assume there exists a map $P'_\R$ such that $\ogw_{\beta,k}'$ vanishes where $k=0$ and $\beta\in\Ima\varpi.$ Hence, by Proposition~\ref{prop_map_p} we get
\[\ogw_{\beta,0}(\Gamma_{i_1},\ldots,\Gamma_{i_l})=\frac{\beta}{4}\gw_{\frac{\beta}{4}}(\Delta_{i_1},\ldots,\Delta_{i_l})\mathfrak{p}_\R(\Delta_2)\] where $\beta\in\Ima\varpi.$ By Table~\ref{table:interior} we have
\[
\ogw_{4,0}(\Gamma_3,\Gamma_3)=\frac{1}{4},\quad \ogw_{4,0}(\Gamma_2,\Gamma_2,\Gamma_3)=\frac{11}{32}.
\]
Since the number of lines through two points is $\gw_{1}(\Delta_3,\Delta_3)=1,$ we get
\[
\mathfrak{p}_\R(\Delta_2) =\frac{1}{4}.
\]
On the other hand, the number of lines through a point and two lines is $\gw_{1}(\Delta_2,\Delta_2,\Delta_3)=1,$  so
\[
\mathfrak{p}_\R(\Delta_2) =\frac{11}{32},
\]
which is a contradiction.

\end{proof}

\bibliography{references.bib}
\bibliographystyle{amsabbrvcnobysame}
\end{document}